\documentclass[11pt,a4paper]{amsart}

\usepackage[a4paper,top=2cm,bottom=2cm,left=2cm,right=2cm]{geometry}
\usepackage[english]{babel}
\usepackage{newlfont}
\usepackage{fancyhdr}
\usepackage{graphicx}
\usepackage{amsmath}
\usepackage{amsfonts}
\usepackage[all]{xypic}
\usepackage{amsthm}
\usepackage{latexsym}
\usepackage[utf8]{inputenc}
\usepackage{enumerate}
\usepackage{stmaryrd}
\usepackage{tikz}
\usetikzlibrary{matrix,arrows}

\theoremstyle{plain}

\newtheorem{thm}{Theorem}[section]

\newtheorem{cor}[thm]{Corollary}
\newtheorem{lem}[thm]{Lemma}
\newtheorem{prop}[thm]{Proposition}

\theoremstyle{definition}
\newtheorem{defn}{Definition}[section]
\newtheorem{oss}{Remark}[section]

\author{Camilla Felisetti}
\title{Intersection cohomology of the moduli space of Higgs bundles on a genus 2 curve}

\DeclareMathOperator{\im}{Im}
\DeclareMathOperator{\rank}{rank}

\newcommand{\ita}{\textit}
\newcommand{\gr}{\textbf}

\newcommand{\re}{\mathbb{R}}
\newcommand{\co}{\mathbb{C}}

\newcommand{\spartial}{\bar{\partial}}

\newcommand{\cinfty}{\mathcal{C}^{\infty}}

\newcommand{\sla}{sl(2)}

\newcommand{\rdol}{\mathcal{R}_{Dol}}
\newcommand{\rdolg}{\mathcal{R}_{Dol}^G}

\newcommand{\mdol}{\mathcal{M}_{Dol}}
\newcommand{\mdolg}{\mathcal{M}_{Dol}^G}

\newcommand{\tmdol}{\tilde{\mathcal{M}}_{Dol}}
\newcommand{\tmdolg}{\tilde{\mathcal{M}}_{Dol}^{G}}

\newcommand{\tomega}{\tilde{\Omega}}
\newcommand{\tsigma}{\tilde{\Sigma}}

\newcommand{\luniv}{\mathcal{L}}
\newcommand{\vuniv}{\mathcal{V}}
\newcommand{\muniv}{\mathcal{M}}
\newcommand{\jup}{\mathcal{K}^0}
\newcommand{\jo}{\mathcal{K}_0}

\begin{document}
\maketitle

\begin{abstract} Let $C$ be a smooth projective curve of genus $2$. Following a method by O' Grady, we construct a semismall desingularization $\tmdol^G$ of the moduli space $\mdol^G$ of semistable $G$-Higgs bundles of degree 0 for $G=\mathrm{GL}(2,\co), \mathrm{SL}(2,\co)$. By the decomposition theorem of Beilinson, Bernstein and Deligne, one can write the cohomology of $\tmdol^G$ as a direct sum of the intersection cohomology of $\mdol^G$ plus other summands supported on the singular locus. We use this splitting to compute the intersection cohomology of $\mdol^G$ and prove that the mixed Hodge structure on it is pure, in analogy with what happens to ordinary cohomology in the smooth case of coprime rank and degree.\\
	

\smallskip
\noindent \tiny{CLASSIFICATION:} 14D20, 14C30
\end{abstract}

\tableofcontents
\section{Introduction}

Let $G=\mathrm{GL}(2,\co), \mathrm{SL}(2,\co)$. We denote by $\mdolg$ be the moduli space of $G$-Higgs bundles of rank 2 and degree 0 on genus 2 curve $C$.
In this paper we want to study the cohomological properties of this space from a Hodge theoretic point of view.
The main problem in doing that is that the moduli space is singular and many fundamental theorems like Poincaré duality or Hard-Lefschetz theorem fail for ordinary cohomology groups. To overcome this fact one might opt for two solutions: to resolve singularities or to consider a different cohomological invariant, namely, intersection cohomology. Here we adopt both strategies, showing how they are strongly interrelated. 
We first show the following result. 
\begin{thm}
	Let $C$ be a curve of genus 2 and let $\mdolg$ be the moduli space of $G$-Higgs bundles of rank 2 and degree 0 on $C$. Then there exists a symplectic desingularization $\tilde{\pi}:\tmdol^G\rightarrow\mdolg$.
\end{thm}

This theorem was obtained independently by Bellamy and Schedler in \cite{BS} in the language of character varieties. In fact it is well known that, by the non abelian Hodge theorem \cite{S}, $\mdolg$ is real analytic isomorphic to the character variety of representations of the fundamental group of $C$ into $G$.
The equivalence with our result is obtained by applying the isosingularity principle \cite[Theorem 10.6]{S}, which states that two corresponding singular points on the moduli spaces admit isomorphic étale neighbourhoods.
Having a symplectic desingularization is indeed a special feature of this moduli space: in fact, apart from rank 2 and degree 0 $G$-Higgs bundles on a genus 2 curve, the only other case in which a symplectic resolution exists is that of degree 0 $G$-Higgs bundles of arbitrary rank on an elliptic curve. For higher genera and ranks, the moduli spaces do not admit a symplectic desingularization, see \cite{T} and \cite{KY}.

By a theorem of Kaledin (see \cite{Kal}) all symplectic resolutions are \ita{semismall} (cfr. section 1) and such a property allows to apply a special version of the Decomposition Theorem of Beilinson, Bernstein and Deligne \cite{BBD} that asserts that the cohomology of the resolution $\tmdol^G$ splits as a direct sum of the intersection cohomology of $\mdolg$ with some summands supported on the singular locus, which we are able to determine. Given this theorem, if we compute the ordinary cohomology of the resolution and subtract the contributions coming from the summands on the singular locus, we end up with the intersection cohomology of $\mdolg$.
More precisely, we prove the following result.
\begin{thm}\label{poincare} Let $C$ be a smooth projective complex curve of genus 2 and let $\mdol^\mathrm{G}$ be the moduli space of $G$-Higgs bundles on $C$ for $G=\mathrm{GL}(2,\co), \mathrm{SL}(2,\co)$. Then the intersection Poincaré polynomials of $\mdolg$ are:
	\begin{align*}
	IP_t(\mdol^{\mathrm{SL}}) &=1+t^2+17t^4+17t^6;\\
	IP_t(\mdol^{\mathrm{GL}}) &= 1+4t+7t^2+8t^3+9t^4+12t^5+15t^6+16t^7+14t^8+8t^9+2t^{10}.
	\end{align*}
	Moreover the Hodge structure on the intersection cohomology of $\mdolg$ is pure and the Hodge diamond is 
	$$G=\mathrm{SL}(2,\co): \begin{tabular}{l|ccc}
	0&(0,0)&\\
	&&\\
	2&(1,1)&\\
	&&\\
	4&17(2,2)&\\
	&&\\
	6&17(3,3)&\\
	\end{tabular}
	$$
	$$G=\mathrm{GL(2,\co)}: \begin{tabular}{l|ccc}
	0&&(0,0)&\\
	1&2(1,0)&&2(0,1)\\
	2&(2,0)&5(1,1)&(0,2)\\
	3&4(2,1)&&4(1,2)\\
	4&(1,3)&7(2,2)&(3,1)\\
	5&6(3,2)&&6(2,3)\\
	6&2(4,2)&11(3,3)&2(2,4)\\
	7&8(4,3)&&8(3,4)\\
	8&2(5,3)&10(2,2)&2(3,5)\\
	9&4(5,4)&&4(4,5)\\
	10&&2(5,5)&.\\
	\end{tabular}
	$$
\end{thm}

Let us now give a sketch of the proof of these two results. 
The resolution is constructed using a strategy developed by O' Grady in \cite{OG} and \cite{OG2} to desingularize moduli spaces of sheaves on K3 or abelian surfaces. With this procedure, O' Grady obtains two new examples of compact hyperk\"{a}hler manifolds up to birational equivalence, which are usually denoted by OG10 and OG6. Such examples have dimension, respectively 10 and 6, as $\mdol^{\mathrm{GL}}$ and $\mdol^{\mathrm{SL}}$. Indeed more is true: applying similar methods to those in \cite{dCMS} and \cite{DEL}, one can show that there exists a degeneration of hyperk\"{a}hler manifolds from O' Grady moduli spaces to those of Higgs bundles. In view of this, the singularities of $\mdolg$ have isomorphic local description to those of O' Grady examples, as we show in section 3. 

The singular locus $\Sigma^G$ is given by Higgs bundles which can be written as a direct sum of two Higgs line bundles. For example, for $G=\mathrm{GL}(2,\co)$ the singular locus stratifies as
\begin{align*} 
\Sigma^G &=\left\lbrace (L,\phi)\oplus (M,\psi)\mid L,M\in Jac(C), \text{ and } \phi,\psi\in H^0(K_C) \right \rbrace;\\
\Omega^G &=\left\lbrace (L,\phi)\oplus (L,\phi)\mid L\in Jac(C), \text{ and } \phi\in H^0(K_C) \right \rbrace.
\end{align*}
It is easy to see that in this case $\Sigma^G$ is isomorphic to the symmetric product $Sym^2(Jac(C)\times H^0(K_C))$, thus it is a singular 8-fold with finite quotient singularities along $\Omega^G$.
The description for $G=\mathrm{SL}(2,\co)$ is analogous mutatis mutandis. \\
In both cases, the resolution is obtained by a single blow up along the locus $\Sigma^G$: a crucial point in the construction is the description of the singularities of O' Grady moduli spaces provided by Lehn and Sorger in \cite{LS}.  

To prove the main theorem \ref{poincare}, we first show that the Hodge structure on intersection cohomology groups is pure. To do that we extend the natural $\co^*$-action on $\mdolg$, given by rescaling the Higgs field, to $\tmdolg$: this yields an isomorphism of mixed Hodge structures between the cohomology of $\tmdolg$ and that of a compact subvariety, namely the preimage $(\pi\circ \chi)^{-1}(0)$ of the nilpotent cone in the resolution.
On the one hand the smoothness of $\tmdolg$ implies that the weights are greater or equal than the cohomological degree; on the other hand, since the cohomology of $\tmdolg$ is that of a compact variety, they cannot exceed it. 
As a result $H^j(\tmdolg)$ carries a pure Hodge structure of weight $j$. Since the Hodge structure on $IH^*(\tmdolg)$ is a sub-Hodge structure of that on $H^*(\tmdolg)$, it is pure as well. 
Indeed one can show in the same way that the Hodge structure on the intersection cohomology of moduli space of rank 2 Higgs bundles is pure in any genus. Ultimately this is due to the fact that the resolution is obtained by blowing up along $\co^*$-equivariant subsets, so that the above argument still works. This will be shown by the author in a forthcoming paper. It is likely that such result should hold also in higher rank, however the structure of the singular locus is more complicated and requires further attention. 

Knowing the purity of the Hodge structure, we can obtain cohomology by computing other cohomological invariants, called \ita{E-polynomials} (cfr. section 1), generally much easier to determine than Poincaré polynomials. 

The $E-$polynomial of a variety $X$ is defined as 
$$ E(X)= \sum_{j,p,q}(-1)^j h^{j,p,q}_c u^pv^q,$$
where $h^{j,p,q}_c=\dim\mathrm{Gr}^F_p Gr^W_{p+q}H_c^j(X)$.
Moreover, up to replace $H^j_c(X)$ with $IH^j_c(X)$, one can define an analogous invariant for intersection cohomology called \ita{intersection $E-$polynomial}.
$E-$polynomials satisfy the an additivity property:
 $E(X)=E(Y)+E(X\setminus Y)$ for all $Y\subset X$.\\
By stratifying $\tmdolg$ and computing $E-$polynomials for all strata we end up with $E(\tmdolg)$. Moreover we also compute the $E-$polynomials of the contributions supported on the singular loci. By subtracting them one get the $E-$polynomial for  the intersection cohomology of $\mdolg$.\\
We remark that, though the additivity property is false in general for intersection E-polynomial, our method applies anyway as we compute the intersection E-polynomial as a sum of actual E-polynomials, for which additivity property holds. 
Notice that, by the purity of the Hodge structure, intersection Betti numbers with compact support are given by
$$ ib_{j,c}=\sum_{p+q=j} ih^{j,p,q}_c, $$ 
where $ih^{j,p,q}_c=\dim\mathrm{Gr}^F_p Gr^W_{p+q}IH_c^j(X)$.  Theorem \ref{poincare} now follows by Poincaré duality.

In the smooth coprime case, the cohomology of these spaces have been widely studied: Poincaré polynomials for $\mathrm{SL}(2, \co)$ were computed by Hitchin in his seminal paper on Higgs bundles \cite{H}, for $\mathrm{SL}(3, \co)$ by Gothen in \cite{Go} and in rank 4 by Garcia-Prada, Heinloth and Schmitt \cite{GPHS}. 
Furthermore, in \cite{HR} Hausel and Rodriguez-Villegas derived a conjectural formula for the $E$-polynomials of twisted $G-$character varieties focusing on $G =\mathrm{GL}(n, \co),\mathrm{SL}(n, \co)$, In \cite{Sc} Schiffmann provided a closed formula for the Poincaré polynomial of the moduli spaces in any coprime rank and degree. Such formula was shown to imply the conjectural formula of Hausel and Rodriguez-Villegas by Mellit in \cite{M}. 

In the singular case, Logares, Mu\~{n}oz and Newstead \cite{LMN} computed the $E-$polynomial of the character varieties of $\mathrm{SL}(2,\co)$ and $\mathrm{GL}(2,\co)$ on curves of genus $g=1,2$, while Martinez and Mu\~{n}oz extended it to $g\geq3$. In \cite{BH} Baraglia and Hekmati gave a new proof of these, extending it to rank 3. Furthermore, they showed how to extend the approach of Hausel and Rodriguez-Villegas used for nonsingular twisted character varieties to the singular (untwisted) case.\\

To the author's knowledge, this is the first result of computation of intersection cohomology for Higgs bundles or character varieties. For vector bundles, where the moduli spaces involved are compact, intersection Betti numbers were computed by Kirwan in \cite{K2}.
The main motivation for this work was provided by the celebrated $P=W$ conjecture by De Cataldo, Hausel and Migliorini (see \cite{dCHM}) which asserts that the \ita{Weight filtration} on the cohomology of the character variety corresponds via non abelian Hodge theorem to the \ita{Perverse filtration} arising from the Hitchin fibration. 
Though the conjecture is formulated for smooth moduli spaces, it would be interesting to see whether an analogue of the $P=W$ conjecture exists in the singular case of moduli of Higgs bundles of non coprime rank and degree and the corresponding character varieties.
Indeed, for moduli spaces with a symplectic resolution, the conjecture has been proved by the author and Mauri in \cite{FM}, relying on the results of this paper. 

The article is organized as follows: in Section 2 we briefly review the theory of intersection cohomology and decomposition theorem; in Section 3 we describe the local geometry of the moduli space focusing on the singularities and their normal cones. In Section 4, we construct a semismall desingularization and apply the decomposition theorem to split the cohomology of the desingularization as a direct sum of the intersection cohomology of $\mdolg$ plus some other summands supported on the singular locus. 
In Section 5, we extend the natural $\co^*$-action on $\mdolg$ to the desingularization and state a localization lemma that yields to the triviality of the weight filtration both on the cohomology of the desingularization and on the intersection cohomology of $\mdolg$.
In Sections 6 and 7, we compute the $E-$polynomial for the intersection cohomology of $\mdolg$ and show that from it, by the triviality of the weight filtration, one can recover the intersection Betti numbers of $\mdolg$ both in the case of $G=\mathrm{SL}(2,\co)$ and $G=\mathrm{GL}(2,\co)$.

\section{Quick review of intersection cohomology and decomposition theorem}

Pure Hodge theory allows to use analytic methods to study algebro-geometric and topological properties of a smooth algebraic variety and comes with the Hodge-Lefschetz package, which includes deep results such as Hard Lefschetz theorem, Poincaré duality and Deligne's theorem for families of projective manifolds.\\
When working with singular or non compact varieties theorems in the Hodge-Lefschetz package fail. To overcome this fact, there are two somewhat complementary approaches: \ita{mixed Hodge theory} and \ita{intersection cohomology}. \\

In mixed Hodge theory, introduced by Deligne in \cite{D} and \cite{D2}, one still investigates the same topological invariant, namely, the cohomology groups, whereas the structure with which it is endowed changes. In particular the $(p,q)$-decomposition of the cohomology of smooth projective varieties is replaced by a more complicated structure. More precisely, the rational cohomology groups are endowed with an increasing filtration $W_{\bullet}$, such that the complexifications of the graded pieces admit a $(p,q)$-decomposition.
\begin{defn}
	Let $X$ be an algebraic variety. A mixed Hodge structure on $H^i(X,\co)$ is the datum of 
	\begin{itemize}
		\item a decreasing filtration $F^{\bullet}$ on $H^i(X,\co)$, called the \ita{Hodge filtration};
		\item an increasing filtration $W_{\bullet}$ on $H^i(X,\mathbb{Q})$, called the \ita{Weight filtration}, such that 
		$$ W_k/W_{k-1}\otimes\co \text{ admits a pure Hodge structure of weight }k\text{ induced by }F^{\bullet}$$
		where the induced filtration on $W_k/W_{k-1}\otimes\co$ is defined as 
		$$ F^p(W_k/W_{k-1}\otimes\co):=(F^p\cap W_k\otimes \co +W_{k-1}\otimes\co)/W_{k-1}\otimes\co.$$
	\end{itemize}
If $W_k/W_{k-1}\otimes\co\cong \bigoplus_{p+q=k}V^{k,p,q}$ we say that a class in $V^{k,p,q}$ has weight $k$ and type $(p,q)$.
\end{defn}
Similarly, one can define a mixed Hodge structure on compactly supported cohomology. This leads to the definition of $E-$polynomials. 
	\begin{defn}
		Let $X$ be an algebraic variety. The \ita{E-polynomial} of $X$ is defined as 
		$$ E(X)(u,v)=\sum_{h=0}^{2\dim X}(-1)^i\sum_{h,p,q} h_c^{i,p,q}u^pv^q,$$
		where $h_c^{i,p,q}=\dim \mathrm{Gr}^F_p Gr^W_{p+q}H_c^i(X)$ and satisfies the following properties:
		\begin{enumerate}[(i)]
			\item if $Z\subset X$ then $E(X)=E(Z)+E(X\setminus Z)$
			\item $E(X\times Y)=E(X)E(Y)$
		\end{enumerate}
	\end{defn}
\begin{oss}
If $X$ is smooth of complex dimension $n$, then mixed Hodge structures are compatible with Poincaré duality, i.e. a class in $H^i(X)$ of weight $k$ and type $(p,q)$ corresponds to a class in $H^{2n-i}_c(X)$ of weight $2n-k$ of type $(n-p,n-q)$.
\end{oss}
\begin{oss}[\gr{Yoga of weights}]\label{yoga}
	In general finding the weight of a cohomology class is a nontrivial task. However, there are some fundamental weight restrictions:
	\begin{enumerate}[i)]
		\item if $X$ is nonsingular, but possibly non-compact, then weights are \ita{high}, i.e.
		$$W_kH^i(X)=0 \text{ for all }k<i;$$
		\item if $X$ is compact but possibly singular then weights are \ita{low} i.e. $$W_kH^i(X)=W_i H^i(X)=H^i(X) \text{ for all }k\geq i.$$
	\end{enumerate}
\end{oss}

In intersection cohomology theory, by contrast, it is the topological invariant which is changed, whereas the $(p,q)$-decomposition turns out to be the same. Intersection cohomology groups are defined as the hypercohomology of some complexes, called \ita{intersection complexes}, that live in the derived category of constructible complexes. Intersection complexes are constructed from local systems defined on nonsingular locally closed subsets of an algebraic variety with a procedure called \ita{intermediate extension} (see \cite[1.4.25, 2.1.9, 2.1.11]{BBD}, \cite{GM}, \cite{GM2}). For a beautiful introduction with also an historical point of view, we refer to \cite{Kl}.

 There is a natural morphism $H^i(X)\rightarrow IH^i(X)$ which is an isomorphism when $X$ has at worst finite quotient singularities.
 Intersection cohomology groups are finite dimensional, satisfy Mayer-Vietoris theorem and K\"{u}nneth formula. Although they are not homotopy invariant, they satisfy analogues of Poincaré duality, Hard Lefschetz theorem and, if $X$ is projective, they admit a pure Hodge structure.
The definition of intersection cohomology is very flexible as it allows for twisted coefficients: given a local system $\mathcal{L}$ on a locally closed nonsingular subvariety $Y$ of $X$ we can define the cohomology groups $IH(\overline{Y},\mathcal{L})$.

\begin{defn}
	Let $X$ be an algebraic variety and let $Y\subset X$ be a locally closed subset contained in the regular part of $X$. Let $\mathcal{L}$ be a local system  on $Y$.
	The \ita{intersection complex $IC_{\overline{Y}}(\mathcal{L})$ associated with $\mathcal{L}$} is a complex of sheaves on $Y$ which extends the complex $\mathcal{L}[\dim Y]$ and is determined up to unique isomorphism in the derived category of constructible sheaves by the conditions
	\begin{itemize}
		\item $\mathcal{H}^j(IC_{\overline{Y}}(\mathcal{L}))=0 \quad 	\text{ for all } j< -\dim Y$,
		\item $\mathcal{H}^{-\dim Y}(IC_{\overline{Y}}(\mathcal{L}_{\mid U}))\cong \mathcal{L}$,
		\item $\dim \mathrm{Supp}\mathcal{H}^j(IC_{\overline{Y}}(\mathcal{L}))<-j, \text{ for all }j>-\dim Y$,
		\item $\dim \mathrm{Supp}\mathcal{H}^j(\mathbb{D}IC_{\overline{Y}}(\mathcal{L}))<-j, \text{ for all }j>-\dim Y$, where $\mathbb{D}IC_{\overline{Y}}\mathcal{L}$ denotes the \ita{Verdier dual} of $IC_{\overline{Y}}\mathcal{L}$.
	\end{itemize}
\end{defn}

\begin{oss}
	Let $X$ be an algebraic variety with regular locus $X_{reg}$. When $\mathcal{L}=\mathbb{Q}_{X_{reg}}$ one just writes $IC_X$ for $IC_X(\mathcal{L})$ and call it \ita{intersection cohomology complex of $X$}. 
	If $X$ is nonsingular, then $IC_X\cong \mathbb{Q}_{X}[\dim X]$.
\end{oss}

\begin{defn}
	Let $X$ be an algebraic variety. The \ita{intersection cohomology groups} of $X$ are defined as  
	$$IH^*(X)=H^{*-\dim X}(X,IC_X).$$
	In general, given any local system $\mathcal{L}$ supported on a locally closed subset $Y$ of $X$, the cohomology groups of $Y$ with coefficients in $\mathcal{L}$ are shifted hypercohomology groups of the intersection complex associated to $\mathcal{L}$:
	$$IH^*(\overline{Y},\mathcal{L})=H^{*-\dim Y}(\overline{Y},IC_{\overline{Y}}(\mathcal{L})).$$
	Taking hypercohomology with compact support, one likewise defines \ita{intersection cohomology groups with compact support} $IH^*_c(X)$ and $IH^*_c(\overline{Y},\mathcal{L})$.
\end{defn}

\begin{oss}
	Here the shift is made so that for a nonsingular variety  intersection cohomology groups coincide with ordinary cohomology groups. 
\end{oss}
\begin{oss}
Just as ordinary cohomology, intersection cohomology groups carry a mixed Hodge structure (see \cite{Sa}). As a result it is possible to define an analogue of $E-$polynomial for intersection cohomology, called \ita{intersection $E-$polynomial}:
$$ IE(X)(u,v)=\sum_{h=0}^{2\dim X}(-1)^i\sum_{h,p,q} ih_c^{i,p,q}u^pv^q$$
where $ih_c^{i,p,q}:=\dim \mathrm{Gr}^p_F Gr^W_{p+q}IH_c^i(X)$.\\
\end{oss}

Along with theorems of Hodge-Lefschetz package, intersection cohomology groups satisfy an analogue of Deligne's theorem for projective manifolds, the \ita{decomposition theorem}. The general statement of this theorem is complicated and will not be discussed here (see for example \cite{dCM} for an extensive survey on the topic). 
Roughly speaking, the decomposition theorem asserts that, given a proper map of algebraic varieties $f:X\rightarrow Y$, the derived pushforward of the intersection complex of $X$ splits as a direct sum of the intersection complex of $Y$ and other intersection complexes associated to local systems supported on some nonsingular locally closed  subsets of $Y$. These subsets are called \ita{supports}.\\
In general it is complicated to determine the supports and the local systems appearing in the splitting. However the decomposition theorem takes a particularly simple form when dealing with a special kind of maps, namely \ita{semismall maps}.

\begin{defn}
Let $f:X\rightarrow Y$ be a map of algebraic varieties. A \ita{stratification for f} is a decomposition of $Y$ into finitely many locally closed nonsingular subsets $Y_{\alpha}$ such that $f^{-1}(Y_{\alpha})\rightarrow Y_{\alpha}$ is a topologically trivial fibration. The subsets $Y_{\alpha}$ are called the \ita{strata} of $f$.
\end{defn}

\begin{defn}
Let $f:X\rightarrow Y$ be a proper map of algebraic varieties. We say that $f$ is \ita{semismall} if there exists a stratification $Y=\bigsqcup Y_{\alpha}$ such that for all $\alpha$
$$ d_{\alpha}\leq \dfrac{1}{2}(\dim X-\dim Y_{\alpha})$$
where $d_{\alpha}:=\dim f^{-1}(y_{\alpha})$ for some $y_{\alpha}\in Y_{\alpha}$.
A stratum is called \ita{relevant} if $$ d_{\alpha}= \dfrac{1}{2}(\dim X-\dim Y_{\alpha}).$$
\end{defn}

For semismall maps, the only supports are the relevant strata and their contributions to the pushforward of $IC_X$ consist of nontrivial summands $IC_{\overline{Y}_{\alpha}}(\mathcal{L_{\alpha}})$, where the local systems $\mathcal{L}_{\alpha}$ are given by the top cohomology of the fibres and turn out to have finite monodromy. 
More precisely, let $Y_{\alpha}$ be a relevant stratum, $y\in Y_{\alpha}$ and let $F_1,\ldots, F_l$ be the irreducible $(\dim Y_{\alpha})-$dimensional components of the fibre $f^{-1}(y)$. The monodromy of the $F_i$'s defines a group homomorphism $\rho_{\alpha}:\pi_1(Y_{\alpha})\rightarrow \mathfrak{S}_l$ from the fundamental group of $Y_{\alpha}$ to the group of permutations of the $F^i$'s. The representation $\rho_{\alpha}$ defines a local system $\mathcal{L}_{\alpha}$ on $Y_{\alpha}$. In this case the semisimplicity of the local system $\mathcal{L}_{\alpha}$ is an elementary consequence of the fact that the monodromy factors through a finite group, so by Maschke theorem it is a direct sum of irreducible representations. As a result, the local systems $\mathcal{L}_{\alpha}$ will be semisimple, i.e. they will be a direct sum of simple local systems. With this notation, the statement of the decomposition theorem for semismall maps is the following. For the proof we refer to \cite{BBD}, \cite{Sa} and \cite{dCM1}.

\begin{thm}[\gr{Decomposition theorem for semismall maps}]\label{ssmall}
Let $f:X\rightarrow Y$ be a semismall map of algebraic varieties and let $\Lambda_{rel}$ the set of relevant strata. For each $Y_{\alpha}\in \Lambda_{rel}$ let $\mathcal{L}_{\alpha}$ the corresponding local system with finite monodromy defined above. Then there exists a canonical isomorphism in the derived category of constructible sheaves
$$Rf_*IC_X\cong \bigoplus_{Y_{\alpha}\in \Lambda_{rel}} IC_{\overline{Y}_{\alpha}}(\mathcal{L}_{\alpha}),$$

with $(\mathcal{L}_{\alpha})_y= H^{2(\dim X - \dim Y_{\alpha})}(f^{-1}(y))$ for all $y\in Y_{\alpha}$.\\
Moreover this is an isomorphism of mixed Hodge structures. 
\end{thm}

\section{Local structure of the moduli space}

Consider a curve $C$ of genus 2 and let $G=\mathrm{GL}(2,\co)$ or $\mathrm{SL}(2,\co)$. We define $\mdolg$ to be the moduli space of $G-$Higgs bundles on $C$: for $G=\mathrm{GL}(2,\co)$ these are just ordinary Higgs bundles of rank 2 and degree 0, while for $G=\mathrm{SL}(2,\co)$ one also asks for the determinant to be trivial.

We shall briefly recall the construction by Simpson of these moduli spaces. 

\begin{itemize}
	\item \cite[Theorem 3.8]{S} Fix a sufficiently large integer $N$ and set $p:=2N-2$.
	There exist a quasi-projective scheme $Q^G$ representing the moduli functor  which parametrizes the isomorphism classes of triples $(V,\Phi,\alpha)$ where $(V,\Phi)$ is a semistable Higgs pair (with $\det V\cong\mathcal{O}_X$, $tr(\Phi)=0$ when $G=\mathrm{SL}(2,\co)$) and $\alpha:\co^p\rightarrow H^0(C,V\otimes \mathcal{O}(N))$ is an isomorphism of vector spaces.\\
	\item \cite[Theorem 4.10]{S} Fix $x\in C$ and let $T^G$ be the frame bundle at $x$ of the universal bundle $\mathcal{V}$ on $Q\times C$ restricted to $x$. Then $G\times \mathrm{GL}(p,\co)$ acts on $Q^G$: indeed $G$ acts as automorphisms of $(V,\Phi)$ while $\mathrm{GL}(p,\co)$ acts on the $\alpha$'s. The action of $\mathrm{GL}(p,\co)$ on $Q^G$ lifts to $T^G$. Since this action is free and every point in $T^G$ is stable with respect to it, one can define
	$$\rdolg=T^G/\mathrm{GL}(p,\co),$$ 
	which represents triples $(V,\Phi,\beta)$ where $\beta$ is an isomorphism $V_x\rightarrow \co^2$.\\
	\item \cite[Theorem 4.10]{S} Every point in $\rdolg$ is semistable with respect to the action of $G$  and the closed orbits correspond to the polystable pairs $(V,\Phi,\beta)$ such that 
	$$(V,\Phi)=(L,\phi)\oplus(M,\psi)$$
	with $L,M\in Jac(C)$ and $\phi,\psi\in H^0(K_C)$. For $G=\mathrm{SL}(2,\co)$, the condition $\det V=\mathcal{O}$ yields $M\cong L^{-1}$, $\psi=-\phi$.
	
	\begin{prop}\cite[Theorem 4.10]{S}
		The GIT quotient $\rdolg \sslash G$ is $\mdolg$. 
	\end{prop}
	
\end{itemize}

As it is well known (for example see \cite[section 1]{S}), the singularities of $\mdolg$ correspond to strictly semistable bundles. If a Higgs bundle $(V,\Phi)$ is strictly semistable, then there exists a $\Phi$-invariant line bundle $L$ of degree 0. 

\begin{prop}
	Let $\mdolg$ be the moduli space of $G-$Higgs bundles.
	\begin{enumerate}[(i)]
		\item If $G=\mathrm{GL}(2,\co)$ then the singularities of $\mdolg$ are
        \begin{itemize}
         	\item $\Sigma^{\mathrm{GL}}:=\{(V,\Phi)\mid (V,\Phi)=(L,\phi)\oplus(M,\psi)$ with $L,M\in Jac(C)$ and $\phi,\psi\in H^0(K_C)\}.$
	     	\item $\Omega^{\mathrm{GL}}:=\{(V,\Phi)\mid (V,\Phi)=(L,\phi)\oplus (L,\phi)$ with $L\in Jac(C),$ and $\phi\in H^0(K_C)$\}.
         \end{itemize}
        \item If $G=\mathrm{SL}(2,\co)$ the singularities of $\mdolg$ are 
        \begin{itemize}
        	\item $\Sigma^{\mathrm{SL}}:=\{(V,\Phi)\mid (V,\Phi)=(L,\phi)\oplus(L^{-1},\phi)$ with $L\in Jac(C)$ and $\phi\in H^0(K_C)\}.$
        	\item $\Omega^{\mathrm{SL}}:=\{(V,\Phi)\mid (V,\Phi)=(L,0)\oplus (L,0)$ with $L^2= \mathcal{O}\in Jac(C)$\}.
        \end{itemize}
     \end{enumerate}
\end{prop}
\begin{proof}
	Clearly, Higgs bundles in $\Sigma^G$ are semistable but not stable thus they lie in the singular locus. The result follows after noticing that nontrivial extensions as Higgs bundles do not appear in $\mdol^G$ as their $G$-orbit in $\rdol^G$ is not closed. 
	\end{proof}
Observe that in both cases $\Omega^G\subset \Sigma^G$. 
In the general case of $G=\mathrm{GL}(2,\co)$, $\Sigma^G$ is parametrized by the symmetric product $Sym^2(Jac(C)\times H^0(K_C))$ where $\mathbb{Z}_2$ acts as the involution that switches summands. $\Omega^G$ is given by the fixed points of the involution and it is parametrized by $Jac(C)\times H^0(K_C)$.\\
In the trivial determinant case, when $G=\mathrm{SL}(2,\co)$, $\Sigma^G\cong  (Pic^0(C)\times H^0(K_C))/\mathbb{Z}_2$ and $\Omega^G$ consists again of the fixed points of the involution, which are the 16 roots of the trivial bundle.

\subsection{Local structure of singularities}

Remarkably, the singularities of $\mdol^G$ have the same local description as the singularities of O' Grady's examples in \cite{OG}, \cite{OG2} (see also \cite{BS} and \cite{KY}). Thanks to this fact, one can copy O' Grady's method almost verbatim to obtain a desingularization of $\mdolg$. 
In this subsection the singularities of $\mdol^G$ and their normal cones are studied, leading to the construction of a desingularization and the proof of its semismallness.

Let $G=\mathrm{SL}(2,\co)$ or $\mathrm{GL}(2,\co)$ and let $\mathfrak{g}$ be its Lie algebra.
We shall describe the singularities of the moduli space of Higgs bundles $\mdol^G$ with $G=\mathrm{GL}(2,\co)$. The trivial determinant case of $G=\mathrm{SL}(2,\co)$ is analogous, provided that we replace $End(V)$ by $End_0(V)$. \\

Let $A^i$ denote the sheaf of $\cinfty$ $i-$forms on $C$. For a polystable Higgs pair $(V,\Phi)$ consider the complex
\begin{equation}\label{diag}
	\begin{tikzpicture}[anchor=base,baseline]
	\node (O) at (0,0) {0};
	\node (A) at (3.5,0) {$End(V)\otimes A^0$};
	\node (B) at (7,0) {$End(V)\otimes A^1$};
	\node (C) at (10.5,0) {$End(V)\otimes A^2$};
	\node (O') at (14,0) {0};
	\draw[->,thick] (O) -- (A) node [midway,above] {};
	\draw[->,thick] (A) -- (B) node [midway,above] {};
	\draw[->,thick] (B) -- (C) node [midway,above] {};
	\draw[->,thick] (C) -- (O') node [midway,above] {};
	\end{tikzpicture}
\end{equation}
with differential $D''=\bar{\partial}+[\phi,-]$. Splitting in $(p,q)$-forms, the cohomology of this complex is equal to the hypercohomology of the double complex

\begin{center}	
	\begin{tikzpicture}[descr/.style={fill=white,inner sep=1.5pt}]
	\matrix (m) [
	matrix of math nodes,
	row sep=2.5em,
	column sep=2.5em,
	text height=1.5ex, text depth=0.25ex
	]
	{  &  0 & 0 &0 &\\
		0& 0& End(V)\otimes A^{1,0}& End(V)\otimes A^{1,1}  &0 \\
		0& End(V)\otimes A^0& End(V)\otimes A^1& End(V)\otimes A^2  &0 \\
		0& End(V)\otimes A^{0,0}& End(V)\otimes A^{0,1}& 0&0 \\
		&  0 & 0 &0 &\\
	};
	
	\path[overlay,->, font=\scriptsize,>=latex]
    (m-2-1) edge (m-2-2)
     (m-3-1) edge (m-3-2)
     (m-4-1) edge (m-4-2)
   
	(m-2-2) edge (m-2-3)
	(m-2-3) edge node[above] {$\spartial$}(m-2-4)
	(m-2-4) edge (m-2-5)
	(m-3-2) edge node[above] {$D''$} (m-3-3)
	(m-3-3) edge node[above] {$D''$} (m-3-4)
	(m-3-4) edge (m-3-5)
	(m-4-2) edge node[above] {$\spartial$} (m-4-3)
	(m-4-3) edge (m-4-4)
	(m-4-4) edge (m-4-5)
	
	(m-1-2) edge (m-2-2)
	(m-2-2) edge (m-3-2)
	(m-3-2) edge node[right] {$=$}(m-4-2)
	(m-4-2) edge (m-5-2)
	
	(m-1-3) edge (m-2-3)
	(m-2-3) edge (m-3-3)
	(m-3-3) edge (m-4-3)
	(m-4-3) edge (m-5-3)
	
	(m-1-4) edge (m-2-4)
	(m-2-4) edge node[right]{$=$}(m-3-4)
	(m-3-4) edge (m-4-4)
	(m-4-4) edge (m-5-4);
	\end{tikzpicture}
	
\end{center}

This means that the cohomology groups $T^i$ of \eqref{diag} fit the long exact sequence
\begin{equation}\label{ti}
\begin{tikzpicture}[baseline ,descr/.style={fill=white,inner sep=1.5pt}]
\matrix (m) [
matrix of math nodes,
row sep=1em,
column sep=2.5em,
text height=1.5ex, text depth=0.25ex
]
{ 0 &  T^0 & H^0(End(V)) &H^0(End(V)\otimes K_C)& \\
	T^1 & H^1(End(V)) &H^1(End(V)\otimes K_C)&  T^2 &0.\\
};

\path[overlay,->, font=\scriptsize,>=latex]
(m-1-1) edge (m-1-2)
(m-1-2) edge (m-1-3)
(m-1-3) edge node[above]{$[\Phi,-]$}(m-1-4)
(m-1-4) edge[out=355,in=175] node[descr,yshift=0.3ex]{}(m-2-1)
(m-2-1) edge (m-2-2)
(m-2-2) edge (m-2-3)
(m-2-3) edge node[above]{$[\Phi,-]$} (m-2-4)
(m-2-4) edge (m-2-5);

\end{tikzpicture}
\end{equation}
\begin{oss}
	Observe also that, by deformation theory for Higgs bundles, the $T^i$'s parametrize extensions of Higgs bundles, i.e. $T^i=\mathrm{Ext}^i_H(V,V)$ in the category of Higgs sheaves.\footnote{ Since we are considering extensions in the category of Higgs sheaves, it would be more natural to denote them by $\mathrm{Ext}^i((V,\Phi),(V,\Phi))$. However, since there is no ambiguity on the Higgs fields involved, in order to lighten the notation we decided just to use the pedex $H$ to distinguish extensions as Higgs bundles from those as vector bundles.}\\ 
	In the trivial determinant case one has to consider traceless extensions $\mathrm{Ext}^i_H(V,V)^0$.
\end{oss}

The following results, due to Simpson, provide a local description of $\mdolg$ and of the normal cones of the singular loci in terms of extensions.
\begin{thm}\cite[Theorem 10.4]{S}\label{cono} Consider $G$ acting on $\rdolg$ and suppose $v=(V,\phi,\beta)\in \rdolg$ is a point in a closed orbit. Let $C$ be the quadratic cone in $T^1$ defined by the map $\eta\mapsto [\eta,\eta]$ (where $[-,-]$ is the graded commutator) and $\mathfrak{h}^{\perp}$ be the perpendicular space to the image of $T^0$ in $\mathfrak{g}$ under the morphism $H^0(End(V))\rightarrow \mathfrak{g}$ . Then the formal completion $(\rdolg,(V,\Phi\beta))\hat{}$ is isomorphic to the formal completion $(C\times \mathfrak{h}^{\perp},0)\hat{}.$
Furthermore if $\mathcal{U}$ is the normal slice at $v$ to the $G$-orbit in $\rdolg$, then 
$$(\mathcal{U},v)\hat{}\cong (C,0)\hat{}.$$
\end{thm}

\begin{prop}\cite[Proposition 10.5]{S}
	Let $w=(V,\Phi)$ be a point $\mdolg$ and let $C$ be the quadratic cone in a point $(V,\Phi,\beta)\in \rdolg$ in the $G$-orbit of $w$. Then the formal completion of $\mdolg$ at $w$ is isomorphic to the formal completion of the GIT quotient $C\sslash H$ of the cone by the stabilizer of $(V,\Phi,\beta)$.
\end{prop}

Observe that since there is a local isomorphism $End V\cong \mathfrak{g}$, an element of $T^1$ can be thought of as a matrix in $\mathfrak{g}$ with coefficient in $H^1(C)\cong H^0(K_C)\oplus H^1(\mathcal{O})$.
In this interpretation, the bracket in theorem \ref{cono} is the Lie bracket of $\mathfrak{g}$ coupled with the perfect pairing 
$$ H^0(K_C)\times H^1(\mathcal{O})\rightarrow H^1(K_C).$$

\subsubsection{Interpretation with extensions}

It is also possible to describe the spaces $T^i$ and the graded commutator more explicitly: consider the Higgs bundle $(V,\Phi)$ as an extension
$$0\rightarrow (L,\phi)\rightarrow (V,\Phi)\rightarrow (M,\psi)\rightarrow 0.$$
The deformation theory of the above Higgs bundle is controlled by the hypercohomology of the complex 
$$\begin{array}{llll}
\mathcal{C}^{\bullet}:&M^{-1}L &\overset{\theta}{\longrightarrow} &M^{-1}L\otimes K_C\\
& f &\longmapsto & \phi f - f\psi\\
\end{array}$$
and there is a long exact sequence 
\begin{equation}\label{dc}
\begin{tikzpicture}[baseline, descr/.style={fill=white,inner sep=1.5pt}]
\matrix (m) [
matrix of math nodes,
row sep=1.5 em,
column sep=2.5em,
text height=1.5ex, text depth=0.25ex
]
{ 0 & \mathrm{Ext}^0_H(L,M)  & H^0(M^{-1}L) &H^0(M^{-1}L K_C)&\\
\mathrm{Ext}^1_H(L,M)	& H^1(M^{-1}L) & H^1(M^{-1}L K_C) &  \mathrm{Ext}^2_H(L,M)&0 \\
};

\path[overlay,->, font=\scriptsize,>=latex]
(m-1-1) edge (m-1-2)
(m-1-2) edge (m-1-3)
(m-1-3) edge node[above]{$\theta$} (m-1-4)
(m-1-4) edge[out=355,in=175] node[descr,yshift=0.3ex] {}(m-2-1)
(m-2-1) edge (m-2-2)
(m-2-2) edge node[above]{$\theta$} (m-2-3)
(m-2-3) edge  (m-2-4)
(m-2-4) edge (m-2-5);
\end{tikzpicture}
\
\end{equation}
where $\mathrm{Ext}^{i}_H(L,M):=\mathbb{H}^i(\mathcal{C}^{\bullet}$) are extensions of $(M,\psi)$ with $(L,\phi)$ as Higgs sheaves.\\
Observe that 
\begin{equation}\label{ext}
\mathrm{Ext}^i_H(V,V)=\mathrm{Ext}^i_H(L,L)\oplus \mathrm{Ext}^i_H(L,M)\oplus \mathrm{Ext}^i_H(M,L)\oplus \mathrm{Ext}^i_H(M,M).
\end{equation}
When considering bundles with trivial determinant and traceless endomorphisms, $M=L^{-1}$ and $\psi=-\phi$. Moreover 
$$\mathrm{Ext}^i_H(V,V)^0=\mathrm{Ext}^i_H(L,L)\oplus \mathrm{Ext}^i_H(L,L^{-1})\oplus \mathrm{Ext}^i_H(L^{-1},L). \footnote{Here the terms $\mathrm{Ext}^i_H(L^{-1},L^{-1})$ are not considered because of the traceless condition.}$$
On $\mathrm{Ext}$ groups there is a natural cup product, called the \ita{Yoneda product},
$$
\begin{array}{cclc}
Yon: &\mathrm{Ext}^1_H(V,V)\times \mathrm{Ext}^1_H(V,V) &\rightarrow &\mathrm{Ext}^2_H(V,V)^0\\
& (\alpha,\beta) &\mapsto & \alpha\cup \beta\\	
\end{array}
$$
and its associated Yoneda square
$$\Upsilon: \mathrm{Ext}^1_H(V,V)\rightarrow \mathrm{Ext}^2_H(V,V)^0, \quad \alpha\mapsto \alpha \cup \alpha.$$

Thinking of elements in $\mathrm{Ext}^1_H(V,V)$ locally as matrices of 1-forms in $\mathfrak{g}$, such a product coincides with the graded commutator in theorem \ref{cono}. 
This is precisely the same situation described in \cite[Section  1.3]{OG}: in fact, by means of decomposition \eqref{ext}, Yoneda product reads as
$$
\begin{array}{clc}
\mathrm{Ext}^1_H(L,L)\oplus \mathrm{Ext}^1_H(M,L)\oplus \mathrm{Ext}^1_H(L,M)\oplus \mathrm{Ext}^1(M,M)&\xrightarrow{\Upsilon} &\mathrm{Ext}^2_H(L,L)\oplus \mathrm{Ext}^2_H(M,L)\oplus \mathrm{Ext}^2_H(L,M)\\
 (a,b,c,d)&\mapsto& (b\cup c, a \cup b+b\cup d, c\cup a +d\cup c ).\\
\end{array}
$$

\subsection{Normal cones of $\Sigma^G$ and $\Omega^G$}
\subsubsection{Cones of elements in $\Sigma^G$}
\begin{prop}\label{normalcones}
	Let $(V,\Phi)$ be an element of $\Sigma^{\mathrm{GL}}\setminus \Omega^{\mathrm{GL}}$. The spaces $\mathrm{Ext}^i_H(V,V)$ are 
	$$ 
	\begin{array}{cl}
	\mathrm{Ext}^0_H(V,V) &=\mathrm{Ext}^0_H(L,L)\oplus \mathrm{Ext}^0_H(M,M) \cong \co^2\\
	\mathrm{Ext}^1_H(V,V) &=\mathrm{Ext}^1_H(L,L) \oplus \mathrm{Ext}^1_H(M,L)\oplus \mathrm{Ext}^1_H(L,M) \oplus \mathrm{Ext}^1(M,M)\cong \co^{12} \\
	\mathrm{Ext}^2_H(V,V) &=\mathrm{Ext}^2_H(L,L) \oplus \mathrm{Ext}^2_H(M,M) \cong\co^2\\
	\end{array}
	$$
	Moreover the normal cone to the orbit of $\Sigma^{\mathrm{GL}}$ in $\rdol^{\mathrm{GL}}$ is $\Upsilon^{-1}(0)$ and its fibre in $v=(V,\Phi,\beta)\in \rdol^{\mathrm{GL}}$ is 
	$$ (C_{\Sigma}\rdol^{\mathrm{GL}})_v\cong \{(b,c)\in \mathrm{Ext}^1_H(L^{-1},L)\oplus \mathrm{Ext}^1_H(L,L^{-1})\mid b\cup c = 0\}.$$
	At the level of $\mdol^{\mathrm{GL}}$ the same holds up to quotient by the stabilizer $\co^*$ of points in $\Sigma^{\mathrm{GL}}$.
\end{prop}

\begin{proof}
	We first compute $\mathrm{Ext}^i_H(L,L)$.
    One has
	\begin{equation*}
	\begin{tikzpicture}[baseline, descr/.style={fill=white,inner sep=1.5pt}]
	\matrix (m) [
	matrix of math nodes,
	row sep=1.5 em,
	column sep=2.5em,
	text height=1.5ex, text depth=0.25ex
	]
	{ 0 & \mathrm{Ext}^0_H(L,L)  & H^0(\mathcal{O}) &H^0(K_C)&\\
		\mathrm{Ext}^1_H(L,L)	& H^1(\mathcal{O}) & H^1(K_C) &  \mathrm{Ext}^2_H(L,L)&0, \\
	};
	
	\path[overlay,->, font=\scriptsize,>=latex]
	(m-1-1) edge (m-1-2)
	(m-1-2) edge (m-1-3)
	(m-1-3) edge node[above]{$\theta$} (m-1-4)
	(m-1-4) edge[out=355,in=175] node[descr,yshift=0.3ex] {}(m-2-1)
	(m-2-1) edge (m-2-2)
	(m-2-2) edge node[above]{$\theta$} (m-2-3)
	(m-2-3) edge  (m-2-4)
	(m-2-4) edge (m-2-5);
	\end{tikzpicture}
	\
	\end{equation*}
	
	where the map $\theta$ sends an element $f\in H^0(\mathcal{O})$ to $f\phi-\phi f$. As $\phi$ is $\cinfty$-linear, every $f\in H^0(\mathcal{O})$ commutes with it, thus $\theta$ is the $ 0 $ map and $Ext_H^0(L,L)\cong H^0(\mathcal{O})\cong \co$. Moreover $\mathrm{Ext}_H^0(L,L)\cong \mathrm{Ext}_H^2(L,L)$ by Serre duality \footnote{We mean Serre duality for Higgs bundles.} and $\mathrm{Ext}_H^1(L,L)\cong H^0(K_C)\oplus H^1(\mathcal{O})$. Thus
	$$\mathrm{Ext}_H^0(L,L)\cong \co, \quad \mathrm{Ext}^1_H(L,L)\cong \co^{4}, \quad \mathrm{Ext}_H^0(L,L)\cong \co .$$
	
	We now compute $\mathrm{Ext}^i_H(L,M)$.
	One has
	\begin{equation*}
	\begin{tikzpicture}[baseline, descr/.style={fill=white,inner sep=1.5pt}]
	\matrix (m) [
	matrix of math nodes,
	row sep=1.5 em,
	column sep=2.5em,
	text height=1.5ex, text depth=0.25ex
	]
	{ 0 & \mathrm{Ext}^0_H(L,M)  & H^0(LM^{-1}) &H^0(LM^{-1} K_C)&\\
		\mathrm{Ext}^1_H(L,M)	& H^1(LM^{-1}) & H^1(LM^{-1} K_C) &  \mathrm{Ext}^2_H(L,M)&0 \\
	};
	
	\path[overlay,->, font=\scriptsize,>=latex]
	(m-1-1) edge (m-1-2)
	(m-1-2) edge (m-1-3)
	(m-1-3) edge node[above]{$\theta$} (m-1-4)
	(m-1-4) edge[out=355,in=175] node[descr,yshift=0.3ex] {}(m-2-1)
	(m-2-1) edge (m-2-2)
	(m-2-2) edge node[above]{$\theta$} (m-2-3)
	(m-2-3) edge  (m-2-4)
	(m-2-4) edge (m-2-5);
	\end{tikzpicture}
	\
	\end{equation*}
   Although $(L,\phi)$ and $(M,\psi)$ are not isomorphic as Higgs bundles, $L$ and $M$ might be as vector bundles. However one can see that this does not change the nature of the description of the normal cone. 
	Suppose first $L\not\cong M$: then $LM^{-1}$ is a nontrivial degree $ 0 $ line bundle, so it has no non-zero global sections and $\mathrm{Ext}^0_H(L,M)=\mathrm{Ext}^2_H(L,M)=0$. Also, $\mathrm{Ext}^1_H(L,M)\cong H^0(LM^{-1}K_C)\oplus H^1(LM^{-1})\cong \co^{2} $. \\
	If $L\cong M$ then $H^0(LM^{-1})\cong H^0(\mathcal{O})\cong \co$ and the map $\theta$ sends $f$ to $\phi f - f\psi$. Since $\phi\neq \psi$ there are no non-zero elements in $H^0(\mathcal{O})$ that commute with the Higgs fields, then $\mathrm{Ext}_H^0(L,M)\cong \mathrm{Ext}^2_H(L,M) = 0 $ as before. Then the  alternate sum of the dimensions of vector spaces in the sequence yields $\mathrm{Ext}^1_H(L,M)\cong \co^{2}$ in both cases. As a result
	$$\mathrm{Ext}_H^0(L,M)=Ext_H^2(L,M)=0, \qquad  \mathrm{Ext}^1_H(L,M)\cong \co^{2}.$$
	
Clearly, since $(L,\phi)$ and $(M,\psi)$ are switched by the involution, there are isomorphisms $\mathrm{Ext}^i_H(M,L)\cong \mathrm{Ext}^i_H(L,M)$ and $\mathrm{Ext}^i_H(M,M)\cong \mathrm{Ext}^i_H(L,L)$. Then
	$$ 
	\begin{array}{cl}
	T^0 &=\mathrm{Ext}^0_H(L,L) \oplus \mathrm{Ext}^0_H(M,M)=\co^2;\\
	T^1 &=\mathrm{Ext}^1_H(L,L) \oplus \mathrm{Ext}^1_H(M,L)\oplus \mathrm{Ext}^1_H(L,M) \oplus \mathrm{Ext}^1_H(M,M)\cong \co^{12};\\
	T^2 &=\mathrm{Ext}^2_H(L,L) \oplus \mathrm{Ext}^2_H(M,M)\cong\co^2.\\
	\end{array}
	$$
This complete the first part of the proof. \\

To prove the second statement one needs to describe the zero locus of Yoneda square and the proof of \cite[Proposition 1.4.1]{OG} applies mutatis mutandis. For ease of the reader we sketch it in terms of Higgs extensions.\\
Consider the map
	$$\begin{array}{cclc}
	\overline{\Upsilon}:& \mathrm{Ext}^1_H(M,L)\oplus \mathrm{Ext}^1_H(L,M) &\longrightarrow &\mathrm{Ext}^2_H(L,L)\cong \co \\\
	&(b,c) &\longmapsto &b \cup c.\\
	\end{array}
	$$
	Observe that, since $\mathrm{Ext}^2_H(L,M)\cong \mathrm{Ext}^2_H(M,L)=0$, $\overline{\Upsilon}$ is the map induced by $\Upsilon$ on $\mathrm{Ext}^1_H(V,V)/(\mathrm{Ext}^1(L,L)\oplus \mathrm{Ext}^1_H(M,M))$.\\
	As a consequence of Serre duality $\overline{\Upsilon}$ is a perfect pairing, so $\overline{\Upsilon}^{-1}(0)$ is a smooth quadric surface in $\co^4$.\\
	Now, the isomorphism $C_v\rdol^{\mathrm{GL}}\cong \Upsilon^{-1}(0)$ is a general fact of deformation theory. For determining the fibre one can use Luna's slice theorem: let $\mathcal{U}$ be the normal slice to $\rdol^{\mathrm{GL}}$ in $v$ and $\mathcal{W}:=\mathcal{U}\cap \mathrm{GL}\Sigma$: then
	$$ (C_{\Sigma}\rdol)_v \cong (C_{\mathcal{W}}\mathcal{U})_v.$$
	Since $T_v\mathcal{W}\cong \mathrm{Ext}^1_H(L,L)\oplus \mathrm{Ext}^1(M,M)$ (cfr. \cite[Claim 1.4.12]{OG}) and $C_v\mathcal{U}$ is the cone over $T_v\mathcal{W}$ with fibre $(C_{\mathcal{W}}\mathcal{U})_v$, the fibre of the cone is $\overline{\Upsilon}^{-1}(0)$.
\end{proof}

The description of the cone in $\rdol^{\mathrm{SL}}$ is identical, provided that one replaces $(M,\psi)$ by $(L^{-1},-\psi)$ and take traceless extensions, so we just state the result. 

\begin{prop}
	Let $(V,\Phi)$ be an element of $\Sigma^{\mathrm{SL}}$. The spaces $\mathrm{Ext}^i_H(V,V)^0$ are 
	$$ 
	\begin{array}{cl}
	\mathrm{Ext}^0_H(V,V)^0 &=\mathrm{Ext}^0_H(L,L) \cong \co;\\
	\mathrm{Ext}^1_H(V,V)^0 &=\mathrm{Ext}^1_H(L,L) \oplus \mathrm{Ext}^1_H(L^{-1},L)\oplus \mathrm{Ext}^1_H(L,L^{-1}) \cong \co^{8};\\
	\mathrm{Ext}^2_H(V,V)^0 &=\mathrm{Ext}^2_H(L,L) \cong\co.\\
	\end{array}
	$$
	Moreover the normal cone to the orbit of $\Sigma^{\mathrm{SL}}$ in $\rdol^{\mathrm{SL}}$ is $\Upsilon^{-1}(0)$ and its fibre in $v=(V,\Phi,\beta)\in \rdol^{\mathrm{SL}}$ is 
	$$ (C_{\Sigma}\rdol^{\mathrm{SL}})_v\cong \{(b,c)\in \mathrm{Ext}^1_H(L^{-1},L)\oplus \mathrm{Ext}^1_H(L,L^{-1})\mid b\cup c = 0\}.$$
		At the level of $\mdol^{\mathrm{SL}}$ the same holds up to quotient by the stabilizer $\co^*$ of points in $\Sigma^{\mathrm{SL}}$.
\end{prop}

\subsubsection{Cone of elements in $\Omega^G$}

Let $v=(V,\Phi)$ be an element of $\Omega^G$. Then 
$$ (V,\Phi)=(L,\phi)\oplus(L,\phi)$$ 
and the bundle $End(V)$ is holomorphically trivial. One has that $H^0(End(V))\cong\mathfrak{g}$ and a generic element of this space can be thought of as a matrix
$$\left( \begin{array}{cc}
a & b\\
c & d\\
\end{array}\right)
$$
with $a,b,c,d\in H^0(\mathcal{O})$. We shall compute the $\mathrm{Ext}^i_H$'s and the quadratic cone defined by the graded commutator. Notice that the second line of the long exact sequence \eqref{ti} is the Serre dual of the first one.
Now, $T^0$ is given by the elements in $\mathfrak{g}$ that commute with the Higgs field, which is diagonal, and then
$$\mathrm{Ext}^0_H(V,V)\cong \mathrm{Ext}^2_H(V,V)\cong \mathfrak{g},$$
i.e. the first map and the last map of the sequence are isomorphisms.\\
To compute $\mathrm{Ext}^1_H(V,V)$, consider the central part of the sequence \eqref{ti}, which in this case is

\begin{equation*}
	\begin{tikzpicture}
	\node (O) at (0,0) {0};
	\node (A) at (3,0) {$H^0(End(V)\otimes K_C)$};
	\node (B) at (6,0) {$T^1 $};
	\node (C) at (9,0) {$H^1(End(V))$};
	\node (O') at (12,0) {0.};
	\draw[->] (O) -- (A) node [midway,above] {};
	\draw[->] (A) -- (B) node [midway,above] {};
	\draw[->] (B) -- (C) node [midway,above] {};
    \draw[->] (C) -- (O') node [midway,above] {};
	\end{tikzpicture}
\end{equation*}

Note that $H^0(End(V)\otimes K_C) \cong H^0(\mathcal{O}^{\oplus \dim \mathfrak{g}}\otimes K_C)\cong H^0(K_C)\otimes \mathfrak{g}$. It follows from Serre duality that $H^1(End(V))\cong H^1(\mathcal{O})\otimes \mathfrak{g}$, so $Ext^1_H(V,V)$ has dimension $4 \dim\mathfrak{g}$
and
$$ \mathrm{Ext}^1_H(V,V) \cong (H^0(K_C)\oplus H^1(\mathcal{O}))\otimes \mathfrak{g} \cong\mathrm{Ext}^1_H(L,L)\otimes \mathfrak{g}.$$
Consider now the composition of the Yoneda product on $\mathrm{Ext}^1_H(L,L)$ with the isomorphism $\mathrm{Ext}^2_H(V,V)^0\cong \mathrm{Ext}^2_H(L,L)\cong\co$ given by integration:
\begin{equation}\label{omega}\omega: \mathrm{Ext}^1_H(L,L)\times \mathrm{Ext}^1_H(L,L)\rightarrow \mathrm{Ext}^2_H(L,L)^0\cong\co.
\end{equation}
This defines a skew-symmetric bilinear form $\omega$ which is non-degenerate by Serre duality.
Set $\Lambda^1:=\mathrm{Ext}^1_H(L,L)$ and 
$$Hom^{\omega}(\sla,\Lambda^1):=\{f:\sla\rightarrow \Lambda^1\mid f^*\omega=0\}.$$
There is a natural action of the automorphism group $G$ of $(V,\Phi)$ on this space given by the composition with the adjoint representation on $\sla$. Note that this action is meaningful also when $G=\mathrm{GL}(2,\co)$: indeed, since the action of $\mathrm{GL}(2,\co)$ factors trough $\mathrm{PGL}(2,\co)$, the action on $\sla$ by adjoint representation is well defined.
\begin{oss}
	Observe that $Hom^{\omega}(\sla,\Lambda^1)$ is precisely the set of those $f\in Hom(\sla,\Lambda^1)$ whose image is an isotropic subspace of $\Lambda^1$ with respect to the symplectic form $\omega$.
\end{oss}
\begin{prop}\label{2.4.8}
	Let $v=(V,\Phi)\in \Omega^G$. Then the normal cone to its orbit in $\rdolg$ is $\Upsilon^{-1}(0)$ and there exist a $G$-equivariant isomorphism 
	$$(C_{\Omega_R}\rdolg)_v\cong Hom^{\omega}(\sla,\Lambda^1).$$
	At the level of $\mdolg$ the same holds up to quotient by the stabilizer $ G $ of points in $\Omega$.
\end{prop}

\begin{proof}
	Again, the proof is similar to that of \cite[Proposition 1.5.1]{OG}. We restate the idea in terms of Higgs extensions. 
	As noticed in the previous paragraph, there are natural isomorphisms
	$$\mathrm{Ext}^1_H(V,V)=\Lambda^1\otimes \mathfrak{g}, \qquad \mathrm{Ext}^2_H(V,V)^0\cong\sla$$
	and the Yoneda product on $\mathrm{Ext}^1$ is the composition of the form $ \omega $ on $ \Lambda^1 $ with the bracket of $\mathfrak{g}$.
	Hence, the Yoneda square $\Upsilon: \Lambda^1\otimes \mathfrak{g} \rightarrow \sla$ reads as
	$$\Upsilon(\sum_i \lambda_i\otimes m_i)=\sum_{i,j}\omega(\lambda_i,\lambda_j) [m_i,m_j] .$$
	Now if $\mathfrak{g}=gl(2)$ then
	$$ \Lambda^1\otimes gl(2)\cong (\Lambda^1\otimes sl(2))\oplus (\Lambda^1\otimes \co \mathrm{Id}_V).$$
	Let $\overline{\Upsilon}:=\Upsilon_{\mid \Lambda^1\otimes \co \mathrm{Id}_V}$.
	Thanks to the self duality of $\sla$ as an algebra and to the identifications 
	$$ \sla\otimes \Lambda^1\cong Hom(\sla,\Lambda^1), \qquad \sla\cong \bigwedge^2 \sla$$
	there is a well defined map
	$$
	\begin{array}{cclc}
	\overline{\Upsilon}: &Hom(\sla,\Lambda^1)&\rightarrow &\bigwedge^2 \sla\\
	&f &\mapsto &2f^*\omega\\
	\end{array}
	$$
	and $\overline{\Upsilon}^{-1}(0)= Hom^{\omega}(\sla,\Lambda^1)$.
	
    As before, the isomorphism $C_v\rdolg\cong \Upsilon^{-1}(0)$ is a general fact of deformation theory. The equality at the level of fibres is a consequence of Luna's slice theorem: let $\mathcal{U}$ be the normal slice to $\rdol^{\mathrm{GL}}$ in $v$ and $\mathcal{W}:=\mathcal{U}\cap\mathrm{GL}\Omega^{\mathrm{GL}}$. Then
    $$ (C_{\Omega}\rdol^{\mathrm{GL}})_v \cong (C_{\mathcal{W}}\mathcal{U})_v.$$
    Since $T_v\mathcal{W}\cong \mathrm{Ext}^1_H(L,L)$ (cfr. \cite[Claim 1.5.14]{OG}) and $C_v\mathcal{U}$ is the cone over $T_v\mathcal{W}$ with fibre $(C_{\mathcal{W}}\mathcal{U})_v$, then  $(C_{\mathcal{W}}\mathcal{U})_v\cong\overline{\Upsilon}^{-1}(0)$.
    
If $\mathfrak{g}\cong \sla$, then $\Upsilon=\overline{\Upsilon}$ and $C_v\rdol^{\mathrm{SL}}\cong(C_\Omega\rdol^{\mathrm{SL}})_v. $
	\end{proof}
\begin{oss}
	Even though for simplicity we consider just the case of genus 2, all the results in this sections hold for arbitrary genus $g\geq 2$, provided that we change dimensions of $\mathrm{Ext}$ groups accordingly to the genus. 
\end{oss}
	
\section{Construction of the desingularization and proof of semismallness }
	
We now briefly recall the construction of the desingularization following O' Grady's strategy. This heavily relies on the results of Lehn and Sorger \cite{LS}. The same description has been used also by Bellamy and Schleder in \cite{BS} to construct desingularizations of the character varieties of $\mathrm{SL}(2,\co)$ and $\mathrm{GL}(2,\co)$. 

\subsection{Local model for the desingularization}
Let ($\Lambda,\omega$) be a symplectic 4 dimensional vector space and let $\mathfrak{sp}(\Lambda)$ be the symplectic Lie algebra of ($\Lambda,\omega$), i.e. the Lie algebra of the Lie group of automorphisms of $\Lambda$ that preserve the symplectic form $\omega$. Let
$$ Z:=\left\lbrace A\in \mathfrak{sp}(\Lambda)\mid A^2=0\right\rbrace$$
be the subvariety of square zero matrices in $Z$. Observe that this implies that any $A\in Z$ has rank $\leq 2$.

By \cite[Théorém 4.5]{LS}, if $v\in \Omega^{\mathrm{SL}}$, there exists an euclidean neighbourhood of $v$ in $\mdol^{\mathrm{SL}}$ biholomorphic to a neighbourhood of the origin in $Z$. The same argument shows that there exists a local analytic isomorphism between $\mdol^{\mathrm{GL}}$ and $Z\times \co^4$. Hence the local geometry of a desingularization $\tmdol^{G}$  is encoded in the local geometry of a symplectic desingularization of $Z$. \\

Let $\Sigma$ be the singular locus of $Z$ and $\Omega$ be the singular locus of $\Sigma$. Observe that $\dim Z=6$, while $\dim \Sigma=4$ and $\dim \Omega=0$. In fact,
$$ \Sigma=\left\lbrace A\in Z \mid \rank A\leq 1 \right\rbrace,\qquad \Omega=\{0\}.$$
Let $\mathcal{G}\subset Gr(2, \Lambda )$ be the lagrangian grassmannian of 2-dimensional $\omega$-isotropic subspaces of $\Lambda$. Notice that $\mathcal{G}$ is a smooth irreducible
3-dimensional quadric and set
$$\tilde{Z}:=\left\lbrace (A,U)\in Z\times \mathcal{G}\mid A(U)=0\right\rbrace.$$
The restriction $\pi_{\mathcal{G}}$ of the second projection of $Z\times \mathcal{G}$ to $ \tilde{Z} $ makes it the total space of a 3-dimensional vector bundle, the cotangent bundle of $\mathcal{G}$. In particular, $\tilde{Z}$ is a smooth symplectic variety and the restriction of the first projection of $Z\times \mathcal{G}$
$$f:\tilde{Z}\rightarrow Z$$
is an isomorphism on the locus of rank 2 matrices of $Z$, thus it is a resolution of singularities. \\
The fibre of the desingularization over a point $A\in\Sigma$ is a $\mathbb{P}^1$ corresponding to the 2-dimensional lagrangian subspaces $U$ contained in the 3-dimensional kernel of $A$, while the central fibre over $0=\Omega$ is the whole $\mathcal{G}$. 
As $Z$ has a $A_1$ singularity along $\Sigma\setminus \Omega$ and $\mathcal{G}$ has dimension 3, it follows that $f:\tilde{Z}\rightarrow Z$ is a symplectic resolution.\\

The following theorem, due to Lehn and Sorger (\cite{LS}, see also \cite[Theorem 8.11]{BS}), gives an intrinsic reformulation of the symplectic desingularization $f:\tilde{Z}\rightarrow Z$. 

\begin{thm}Let $v=(V,\Phi)\in \Omega^G$ be a point in the singular locus of $\mdolg$. Then
	\begin{enumerate}
		\item (\cite[Theorem 4.5]{LS}) There exist local analytic isomorphisms
		$$ (Z,0)\cong (\mdol^{\mathrm{SL}},v)\qquad (Z\times \co^4,0)\cong (\mdol^{\mathrm{GL}},v);$$
		\item (\cite[Theorem 3.1]{LS}) The resolution $f:\tilde{Z}\rightarrow Z$ defined above coincides with the blow up of $Z$ along the singular locus $\Sigma$.
	\end{enumerate}
\end{thm}

Given a point $v\in \Omega^G$, the local isomorphism is given by taking $\Lambda=\Lambda^1$ and $\omega$ to be the symplectic form in (\ref{omega}).

\subsection{Global description of the desingularization and proof of semismallness}
Observe that since the blow up is a local construction, by blowing up $\mdolg$ along $\Sigma^G$ one obtains a symplectic desingularization $\tilde{\pi}^G:\tmdol^G\rightarrow \mdolg$ of $\mdolg$.

\begin{prop}\label{fibres}
	Let $\tilde{\pi}^G:\tmdol^G\rightarrow \mdolg$ be the symplectic desingularization obtained by blowing up $\mdolg$ along $\Sigma^G$. The fibres of the desingularization are as follows:
	\begin{enumerate}
		\item over smooth locus of stable Higgs bundles $\mdol^{G,s}$, $\tilde{\pi}^G$ is an isomorphism;
		\item let $v=(L,\phi)\oplus(M,\psi)\in\Sigma^G\setminus \Omega^G$. Then 
		$$(\tilde{\pi}^G)^{-1}(v)\cong \mathbb{P}^1\cong \mathbb{P}\left(\left\lbrace(b,c)\in \mathrm{Ext}^1(M,L)\oplus \mathrm{Ext}^1(L,M)\mid b\cup c=0 \right\rbrace\sslash\co^*\right);$$
		\item let $v=(L,\phi)\oplus(L,\phi)\in\Omega^G$. Then 
		$(\tilde{\pi}^G)^{-1}(v)\cong \mathcal{G}$, where $\mathcal{G}$ is the lagrangian grassmannian of isotropic 2-dimensional subspaces in the symplectic vector space $(\mathrm{Ext}^1_H(L,L),\omega)$.
	\end{enumerate}
	
\end{prop}

By \cite[Proposition 1.2]{Kal} any symplectic resolution of singularities is semismall. However one can also check semismallness by direct computation. 
\begin{prop}\label{semismallness}
	Consider $\tilde{\pi}^G:\tmdol^G\rightarrow \mdol$. Then $\tilde{\pi}^G$ is semismall. 
\end{prop}
\begin{proof}
	Recall that a proper map $f:X\rightarrow Y$ of algebraic varieties is semismall if and only if
	\begin{equation}\label{condsmall}
	k\leq\dfrac{1}{2}(\dim X-\dim Y_k) \quad \text{ for any }k\text{ such that }Y_k\neq \emptyset,
	\end{equation}
	where $Y_k=\{ y\in Y\mid \dim f^{-1}(y)= k \}$. First of all notice that $\tilde{\pi}^G$ is a proper birational map and set 
	$$\mathcal{M}^G_{Dol,k}:=\left\lbrace v\in \mdolg\mid \dim(\tilde{\pi}^G)^{-1}(v)=k\right\rbrace .$$ 
	We stratify $\mdolg$ as 
	$$\mdolg=\mdol^{G,s}\sqcup\Sigma^G\setminus \Omega^G \sqcup \Omega^G$$
	and show that $\mdol^{G,s}=\mathcal{M}^G_{Dol,0}$, $\Sigma^G=\mathcal{M}^G_{Dol,1}$ and $\Omega^G=\mathcal{M}^G_{Dol,3}$.
	We distinguish two cases, depending on the group $G$.
	\begin{itemize}
		\item If \gr{$G=\mathrm{SL}(2,\co)$},  $\mdol^{\mathrm{SL}}$ is a quasi-projective variety of dimension 6. By definition,
 	$$\Sigma^{\mathrm{SL}}\setminus \Omega^{\mathrm{SL}}=\{(V,\Phi)\mid (V,\Phi)=(L,\phi)\oplus (L^{-1},-\phi),\text{ with }(L,\phi)\not \cong (L^{-1},-\phi)\}$$
	is isomorphic to $\left[(Jac(C)\times H^0(K_C))\setminus(16 \text { points})\right]/\mathbb{Z}_2$. $Jac(C)$ is a 2-dimensional torus while $H^0(K_C)\cong\co^2$, therefore $\Sigma^{\mathrm{SL}}\setminus \Omega^{\mathrm{GL}}$ has dimension 4. 
	The singular locus 
	$$\Omega^{\mathrm{SL}}= \{(V,\Phi)\mid (V,\Phi)=(L,0)\oplus (L,0) \text{ with }L\cong L^{-1}\}$$
	parametrizing the fixed points of the involution $(L,\phi)\mapsto (L^{-1},-\phi)$ consists just of 16 points, corresponding to the roots of the trivial bundle on $C$. 
		On $\mdol^{\mathrm{SL},s}$, $\tilde{\pi}^{\mathrm{SL}}$ is an isomorphism and every point has just one pre-image, thus $\mdol^{\mathrm{SL},s}=\mathcal{M}^{\mathrm{SL}}_{Dol,0}$. Thus it satisfies \eqref{condsmall}.
	Let now $v\in\Sigma^{\mathrm{SL}}\setminus \Omega^{\mathrm{SL}}$. By proposition (\ref{fibres}), the fibre is $\mathbb{P}^1$.  Then $\Sigma^{\mathrm{SL}}$ corresponds to the stratum $\mathcal{M}^{\mathrm{SL}}_{Dol,1}$ and satisfies \eqref{condsmall}.\\
	The fibre $\mathcal{G}$  over each one of the 16 points of $\Omega^{\mathrm{SL}}$ is 3-dimensional so, $\Omega^{\mathrm{SL}}=\mathcal{M}^{\mathrm{SL}}_{Dol,3}$ and \eqref{condsmall} is satisfied as well. \\
	
	\item If \gr{$G=\mathrm{GL}(2,\co)$}, $\mdol^{\mathrm{GL}}$ is a quasi-projective variety of dimension 10. Note that
	$$\Sigma^{\mathrm{GL}}\setminus \Omega^{\mathrm{GL}}=\{(V,\Phi)\mid (V,\Phi)=(L,\phi)\oplus (M,\psi),\text{ with }(L,\phi)\not \cong (M,\psi)\}$$
	is parametrized by $\left[(Jac(C)\times H^0(K_C))^{(2)}\setminus(Jac(C)\times H^0(K_C))\right]/\mathbb{Z}_2$. The symmetric product $Sym^2(Jac(C)\times H^0(K_C))$ has dimension 8. 
	The singular locus 
	$$\Omega^{\mathrm{GL}}= \{(V,\Phi)\mid (V,\Phi)=(L,\phi)\oplus (L,\phi) \text{ with }L\in Jac(C),\phi \in H^0(K_C)\}$$
	is isomorphic to $Jac(C)\times H^0(K_C)$ so it has dimension 4.
	Proceeding as before, one shows that $\mdol^{\mathrm{GL},s}=\mathcal{M}^{\mathrm{GL}}_{Dol,0}$,  $\Sigma^{\mathrm{GL}}\setminus \Omega^{\mathrm{GL}}=\mathcal{M}^{\mathrm{GL}}_{Dol,1}$ and $\Omega^{\mathrm{GL}}=\mathcal{M}^{\mathrm{GL}}_{Dol,3}$.
	\end{itemize}
\end{proof}

\begin{oss}
	Observe that in both cases all the strata satisfy the equality $$k=\dfrac{1}{2}(\dim\tmdol^G-\mathcal{M}^G_{Dol,k}),$$
	that is, they are \ita{relevant strata} in the decomposition theorem for $\tilde{\pi}^G$.
\end{oss}

\subsection{Decomposition theorem for $\tilde{\pi}^G$}
We have constructed a semismall desingularization $\tmdol^G\xrightarrow{\tilde{\pi}^{G}} \mdolg$.
As seen in the proof of proposition \eqref{semismallness}, all the strata of the map $\tilde{\pi}^{G}:\tmdolg\rightarrow \mdolg$ are relevant. In particular 
$$\mdol^{G,s}=\mathcal{M}^{G}_{Dol,0}, \quad \Sigma^{G}=\mathcal{M}^{G}_{Dol,1}, \quad \Omega^{G}=\mathcal{M}^{G}_{Dol,3}.$$
Stratify $\tmdol^G$ as follows
$$\tmdol^{G}=(\tilde{\pi}^{G})^{-1}(\mdol^{G,s})\sqcup (\tsigma^{G}\setminus \tomega^{G})\sqcup \tomega^{G};$$
where $\tsigma^{G}:=(\tilde{\pi}^{G})^{-1}(\Sigma)$ and $\tomega^{G}:=(\tilde{\pi}^{G})^{-1}(\Omega^G)$.
%

By the decomposition theorem one obtains the splitting
\begin{equation}\label{dtg}
IC_{\tmdol^{G}}=IC_{\mdol^{G}}(\mathcal{L}_{\mdol^{G}})\oplus IC_{\Sigma^{G}}(\mathcal{L}_{\Sigma^{G}})\oplus IC_{\Omega^{G}}(\mathcal{L}_{\Omega^{G}}).
\end{equation}

Notice that, up to dimensional shifts, the stalks of the local systems $\mathcal{L}_{\Sigma^{G}}$ and $\mathcal{L}_{\Omega^{G}}$ in a generic point of the corresponding stratum are isomorphic to the top cohomology groups of the fibres, which are, respectively, $H^2(\mathbb{P}^1)$ and $H^6(\mathcal{G})$. Moreover, as the fibres of $\tilde{\pi}^{G}$ over $\Sigma^G$ and $\Omega^G$ are irreducible, then the monodromy of $\mathcal{L}_{\Sigma^{G}}$ and $\mathcal{L}_{\Omega^{G}}$ is trivial. Finally, since $\Omega^G$ is non singular and $\Sigma^G$ has finite quotient singularities, intersection cohomology and cohomology coincide. One has 
$$IC_{\mdol}(\mathcal{L}_{\mdol})_{\mid\mdol^s}=\mathbb{Q}[\dim \mdolg],\quad IC_{\Sigma}(\mathcal{L}_{\Sigma})\cong \mathbb{Q}[\dim \Sigma^G](-1),\quad IC_{\Omega}(\mathcal{L}_{\Omega})\cong\mathbb{Q}[\dim \Omega^{G}](-3), $$
where the Tate shifts $(-1)$ and $(-3)$ correspond to the Hodge structures $\mathbb{Q}(-1)$ of $H^2(\mathbb{P}^1)$ and $\mathbb{Q}(-3)$ of $H^6(\mathcal{G})$.\\
Taking hypercohomology on both sides of \eqref{dtg} one has 
\begin{equation}
H^*(\tmdol^G)=IH^*(\mdolg)\oplus H^{*-2}(\Sigma^G)(-1)\oplus H^{*-6}(\Omega^G)(-3).
\end{equation}


\section{Purity of the Hodge structure}

The aim of this section is to show that, although $\tmdol^G$ is non compact, the Hodge structure on its cohomology is pure, i.e. $H^k(\tmdol^G)$ has weight $k$ for any $k=0\ldots 2\dim \tmdol^G$.\\
Recall that there exists the natural $\co^*$ action on $\mdol^G$ by scalar multiplication on the Higgs field $\Phi$.  The Hitchin map $$\chi:\mdolg\rightarrow \mathbb{A}=H^0(K_C)\oplus H^0(K^2_C)$$ is equivariant with respect to this action if we let $\co^*$ act on $H^0(C,K_C^{i})$ with weight $i$. \\
The following localization result appears in several variants (e.g. \cite[Lemma 6.5]{DL} or \cite[Lemma 4.2]{dCMM}).
\begin{lem}\label{loclem}
Let $\rho:\co^*\times \mathbb{A}^n\rightarrow \mathbb{A}^n$ be a linear action on an affine space such that all the weights of the action are positive. Denote by $s_0:\mathrm{Spec}\co\rightarrow \mathbb{A}^n$ the inclusion of the origin in $\mathbb{A}^n$ and by $p:\mathbb{A}^n\rightarrow \mathrm{Spec}\co$ the projection. 
Then for any $\co^*$-equivariant complex $K$ of sheaves on $\mathbb{A}^n$ one has
$$ Rp_*K=s_0^*K \text{ and } Rp_!K=s_0^!K.$$
\end{lem}
As a corollary one has the following result.
\begin{prop}\label{purity}
	Let $X$ be a smooth variety with an action of $\co^*$. Assume that $f:X\rightarrow \mathbb{A}^n$ is a proper map, equivariant with respect to a linear action of $\co^*$ on the affine space $\mathbb{A}^n$, such that all the weights of this action are positive. Let $X_0:=f^{-1}(0)$ be the fibre of $f$ over $0=\mathrm{Spec}\co\subset \mathbb{A}^n$. Then there is an isomorphism of mixed Hodge structures
	$$ H^*(X,\mathbb{Q})\cong H^*(X_0,\mathbb{Q}).$$
\end{prop}
\begin{proof}
	Since $f$ is $\co^*$-equivariant we can apply lemma \ref{loclem} to $K=Rf_*\mathbb{Q}$ and by base change theorem for proper maps we get:
	$$ H^*(X,\mathbb{Q})\cong H^*(\mathbb{A}^n,Rf_*\mathbb{Q})\cong H^*(\mathrm{Spec}\co,s_0^*Rf_*\mathbb{Q})\cong H^*(X_0,\mathbb{Q}).$$
\end{proof}

For smooth moduli spaces $\mathcal{M}_{Dol}(d,n)$ of Higgs bundles of coprime rank and degree the variety $X_0:=\chi^{-1}(0)$ is the \ita{nilpotent cone} of $\chi$ and the previous proposition yields the purity of the Hodge structure on cohomology groups of those moduli spaces. In fact, on the one hand, $\mathcal{M}_{Dol}(d,n)$ is smooth so by the weight restrictions in remark \ref{yoga} the weights of $H^k(\mathcal{M}_{Dol}(d,n))$ are $\geq k$.
On the other hand, proposition \ref{purity} gives an isomorphism of mixed Hodge structures between the cohomology of $\mathcal{M}_{Dol}(d,n)$ and that of the nilpotent cone, which is compact, thus the weights of $H^k(\mathcal{M}_{Dol}(d,n))$ are  $\leq k$.
Combining these two conditions one has that $H^k(\mathcal{M}_{Dol}(d,n))$
has weight $k$. \\
	
Let $\tilde{\chi}:\tmdol^G\rightarrow \mathbb{A}$ be the composition of $\chi$ with $\tilde{\pi}^G$. Suppose that one can extend the $\co^*$ action on $\mdolg$ to $\tmdol^G$ in a way such that $\tilde{\chi}$ is equivariant.
Since $\tilde{\chi}$ is a proper map, the variety $\tilde{\chi}^{-1}(0)$ is compact.
As a result the above weight trick applies to $\tmdol^G$ as well proving the following result.
\begin{thm}\label{purezza}
	Let $\tmdol^G$ be the semismall desingularization of $\mdolg$ constructed in the section 4. Then the Hodge structure on $H^*(\tmdol^G)$ is pure.
\end{thm}
Moreover, as a consequence of the decomposition theorem, the Hodge structure on $IH^*(\mdolg)$ is a sub-Hodge structure of that on $H^*(\tmdol^G)$, so the former is pure whenever the latter is. 
\begin{cor}
The Hodge structure on $IH^*(\mdolg)$ is pure.
\end{cor}

The above considerations reduce the proof of theorem \ref{purezza} to showing the following lemma.
\begin{lem}
There exists a $\co^*$ action on $\tmdol^G$ extending the natural $\co^*$ action on $\mdolg$ with respect to which the map $\tilde{\chi}$ is equivariant.
\end{lem}

\begin{proof}
Since $\tilde{\pi}^G:\tmdol^G\rightarrow \mdolg$ is an isomorphism on the smooth locus, to prove the lemma one needs to extend the $\co^*$ action to the fibres of $\tilde{\pi}^G$ over the singular loci.\\
Given a point $v=(V,\Phi)\in\mdolg$ then $T_v\mdolg\cong T^1\sslash \mathrm{Stab}(v)$, where $T^1=\mathrm{Ext}^1_H(V,V)$.
By construction, the fibre of $\tilde{\pi}^G$ over $\Sigma^G$(resp. $\Omega^G)$ is the fibre of the normal cone $C_{\Sigma}\mdolg$ to $\Sigma^G$(resp. $\Omega^G)$ in $\mdolg$.
Recall that by proposition \ref{normalcones} $$(C_\Sigma\mdolg)_v\cong \overline\Upsilon^{-1}(0)\sslash \mathrm{Stab}(v)$$
for all $v\in \Sigma^G$(resp. $\Omega^G$) and that $\overline\Upsilon^{-1}(0)\subset T^1$.
As a consequence, to describe the action on the fibres one first needs to describe it over $T^1$.
Consider the diagram 
\begin{equation*}
\begin{tikzpicture}[anchor=base,baseline ,descr/.style={fill=white,inner sep=1.5pt}]
\matrix (m) [
matrix of math nodes,
row sep=1em,
column sep=2.5em,
text height=1.5ex, text depth=0.25ex
]
{ 0 &  T^0 & H^0(End(V)) &H^0(End(V)\otimes K_C)& \\
	T^1 & H^1(End(V)) &H^1(End(V)\otimes K_C)&  T^2 &0.\\
};

\path[overlay,->, font=\scriptsize,>=latex]
(m-1-1) edge (m-1-2)
(m-1-2) edge (m-1-3)
(m-1-3) edge node[above]{$[\Phi,-]$}(m-1-4)
(m-1-4) edge[out=355,in=175] node[descr,yshift=0.3ex]{$\alpha$}(m-2-1)
(m-2-1) edge node[above]{$\beta$ }(m-2-2)
(m-2-2) edge (m-2-3)
(m-2-3) edge node[above]{$[\Phi,-]$} (m-2-4)
(m-2-4) edge (m-2-5);

\end{tikzpicture}
\end{equation*}
An element $\lambda\in\co^*$ acts on the diagram by scalar multiplication on the Higgs field in the commutator. As $T^1\cong \im \alpha\oplus \mathrm{coker}\beta$, $\lambda\in \co^*$ acts on $T^1$ as scalar multiplication on $ \im \alpha$ and as the identity on $\mathrm{coker}\beta$.
Since both $\Sigma^G$ and $\Omega^G$ are $\co^*$-invariant, one can easily show that $\overline{\Upsilon}^{-1}(0)$ is invariant under the $\co^*$ action on $T^1$. 
Moreover, such action commutes with that of the stabilizer, so one has a well defined action on the fibres of $\tilde{\pi}^G$. 
Observe that by construction the map $\tilde{\chi}$ is $\co^*$-equivariant.
 \end{proof}

Purity of the Hodge structure on cohomology groups yields, by Poincaré duality, purity of Hodge structure on cohomology groups with compact supports. The following lemma implies that intersection Betti numbers of $\mdolg$ and Betti numbers of $\tmdol^G$ can be computed just by knowing $E-$polynomials.
\begin{lem}\label{etob}
	Let $X$ be a complex algebraic variety, possibly singular, and let 
	$$IE(X)=\sum_{p,q} \alpha^{p,q} u^pv^q$$
	denote its intersection $E-$polynomial. If $IH^*(X)$ admits a pure Hodge structure then
	$$ib_k(X):=\dim IH^k(X)=\sum_{p+q=2\dim X-k} \alpha^{p,q}$$
	for any $k=0\ldots2\dim X$.
\end{lem}
\begin{oss}
	When $X$ is nonsingular, $IH^*(X)=H^*(X)$ so the formula in lemma \ref{etob} holds for $E-$polynomial and usual Betti numbers.
\end{oss}
\begin{proof}[Proof of Lemma \ref{etob}]
	If $IH^*(X)$ admits a pure Hodge structure then $IH^k(X)$ has weight $k$ for any $k=0\ldots 2\dim X$. Recall that intersection cohomology groups satisfy Poincaré-Verdier duality
	$$ IH^k(X)\cong IH_c^{2\dim X-k}(X)$$
	and that this isomorphism maps classes of type $(p,q)$ in classes of type $(\dim X-p,\dim X-q)$. This concludes the proof. 	
\end{proof}

\subsection{Computation of the intersection $E-$polynomial}\label{etobsec}
The aim of the next two sections is to compute the intersection E-polynomial $IE(\mdolg)$ for $G=\mathrm{SL}(2,\co)$ and $G=\mathrm{GL}(2,\co)$; as a consequence of lemma \ref{etob} they will give the intersection Betti numbers of the corresponding moduli spaces of Higgs bundles. Before proceeding with computations we shall describe the general strategy. \\
Computing hypercohomology with compact support on both sides of \eqref{dtg}, the splitting in the decomposition theorem becomes
$$ H_c^*(\tmdol^G)=IH_c^*(\mdolg) \oplus H_c^{*-2}(\Sigma^G)(-1)\oplus H_c^{*-6}(\Omega^G)(-3).$$ 
This equality holds also at the level of $E-$polynomials:
\begin{equation}\label{edt} E(\tmdolg)=IE(\mdol^G)+E^{top}(\tsigma^G)+E^{top}(\tomega^G),
\end{equation}
where $E^{top}(\tsigma^G)=E(\Sigma^G\times H^2(\mathbb{P}^1))$ and $E^{top}(\tomega^G)=E(\Omega^G\times H^6(\mathcal{G}))$.\\
By additivity of $E-$polynomials, $E(\tmdolg)$ is given by
\begin{equation}\label{etmdol}
E(\tmdolg)=E(\mdol^{G,s})+E(\tsigma^G\setminus \tomega^G)+E(\tomega^G).
\end{equation}
In order to obtain $IE(\mdol^G)$, by \eqref{edt} it is sufficient to subtract from $E(\tmdolg)$ the contributions $E^{top}$ coming from the top cohomology of the fibres over the singular loci.

\section{Intersection cohomology of $\mdol^{\mathrm{SL}}$}

Consider the semismall desingularization $\tmdol^{\mathrm{SL}}\xrightarrow{\tilde{\pi}^{\mathrm{SL}}} \mdol^{\mathrm{SL}}$ of the moduli space $\mdol^{\mathrm{SL}}$ of Higgs bundles of rank 2, degree 0 and trivial determinant over C. The aim of this subsection is to prove the following theorem.

\begin{thm}[\gr{Intersection cohomology of $\mdol^{\mathrm{SL}}$}]\label{ihsl} 
The intersection Poincaré polynomial of $\mdol^{\mathrm{SL}}$ is 
$$IP_t(\mdol^{\mathrm{SL}})=1+t^2+17t^4+17t^6.$$
Moreover the Hodge diamond is:
\begin{center}
\begin{tabular}{l|ccc}
	$H^0$&&(0,0)&\\
	&&&\\
	$H^2$&&(1,1)&\\
	&&&\\
	$H^4$&&17(2,2)&\\
	&&&\\
	$H^6$&&17(3,3)&\\
\end{tabular}
\end{center}

\end{thm}
The proof of the theorem consists in computing the intersection $E-$polynomial $IE(\mdol^{\mathrm{SL}})$ and applying lemma \ref{etob} to get intersection Betti numbers.
By \eqref{edt} one needs first to compute $E(\tmdol^{\mathrm{SL}})$ using the stratification in \eqref{etmdol}. 

\subsection{Cohomology of $\mdol^{\mathrm{SL},s}$}

Consider the smooth part $\mdol^{\mathrm{SL},s}$ of the moduli space $\mdol^{\mathrm{SL}}$, which parametrizes stable pairs $(V,\Phi)$. 
\begin{prop}\label{emdols}
Let $\mdol^{\mathrm{SL},s}\subset\mdol^{\mathrm{SL}}$ be the locus of stable Higgs pairs. Then the $E-$polynomial of $\mdol^{\mathrm{SL},s}$ is
$$E(\mdol^{\mathrm{SL},s})=u^6v^6+u^5v^5+16u^4v^4+13u^3v^3-u^2v^4-u^4v^2-17u^2v^2.$$
\end{prop}

It is well known that $\mdol^{\mathrm{SL},s}$ contains the cotangent bundle of  the locus $\mathcal{S}$ of stable vector bundles with trivial determinant as an open dense subset, but there are several stable Higgs bundles whose underlying vector bundle is not stable. This is due to the fact that not all vector subbundles of $V$ are Higgs subbundles: for example one may consider the bundle $$V=K_C^{-1}\oplus K_C$$ where $K_C$ denotes the canonical bundle on $X$. This vector bundle is not stable because the subbundle $K_C$ has slope greater than the slope of $V$; however $K_C$ is not a Higgs subbundle because, in order for it to be $\Phi$ invariant, $Hom(K_C,K_C^{-1})\cong K_C^{-2}$ should have global sections, which is not the case as it has negative degree.\\
To determine the $E$-polynomial of $\mdol^{\mathrm{SL},s}$ one constructs a suitable stratification, computes $E$-polynomials of all strata and sum them. 
In particular, as it is of its own interest, we also compute Betti numbers of the strata by systematic employment of the following well known result. 
\begin{prop}[\gr{Addivity property of compact support cohomology}]\label{ap}
Let $Y$ be a quasi-projective variety. Let $Z$ be a closed subset of $Y$ and call $U$ its complement. Then, given the inclusions $\xymatrix{ 
	U\ar@{^{(}->}[r]^j & Y &Z\ar@{_{(}->}[l]_i  
}
$, there is a long exact sequence in cohomology 
\begin{equation*}
\begin{tikzpicture}[anchor=base,baseline]
\node (O) at (0,0) {$\cdots$};
\node (A) at (2,0) {$H^i_c(U)$};
\node (B) at (4,0) {$H^i_c(X)$};
\node (C) at (6,0) {$H^i_c(Z)$};
\node (O') at (8,0) {$\cdots$};
\draw[->] (O) -- (A) node [midway,above] {};
\draw[->] (A) -- (B) node [midway,above] {$j_!$};
\draw[->] (B) -- (C) node [midway,above] {$i^!$};
\draw[->] (C) -- (O') node [midway,above] {};
\end{tikzpicture}
\end{equation*}

\end{prop}

We stratify the locus of stable Higgs pairs with respect to the stability of the underlying vector bundle:
\begin{itemize}
\item pairs $(V,\Phi)$ with $V$ stable vector bundle;
\item pairs $(V,\Phi)$ with $V$ strictly semistable vector bundle;
\item pairs $(V,\Phi)$ with $V$ unstable vector bundle.
\end{itemize}

\subsection{The stable case}

We shall parametrize all stable Higgs bundles $(V,\Phi)$ where V is a stable vector bundle.
\begin{prop} 
Let $\mathcal{S}$ be the locus of stable vector bundles with trivial determinant. The locus of stable Higgs pairs $(V,\Phi)$ with $V\in \mathcal{S}$ is isomorphic to $T^*\mathcal{S}$ and its $E-$ polynomial is 
$$E(\mathcal{T^*S})(u,v)=u^6v^6-u^3v^5-u^5v^3-3u^4v^4$$
\end{prop}
\begin{proof}
Clearly, if $V$ is a stable vector bundle then $(V,\Phi)$ is a stable Higgs pair. As a consequence the locus of stable  Higgs pairs with stable underlying vector bundle is isomorphic to $T^*\mathcal{S}$.
Narasimhan and Ramanan \cite{NR} proved that the moduli space of semistable vector bundles with trivial determinant on a nonsingular projective curve $C$ of genus 2 is isomorphic to $\mathbb{P}^3$. A semistable vector bundle $V$ is non stable if and only if is of the form
$$V=L\oplus L^{-1},\quad L\in Jac(C)$$
therefore strictly semistable vector bundles are parametrized by $\mathcal{K}:=Jac(C)/\mathbb{Z}_2$ where $\mathbb{Z}_2$ acts as the involution $L\mapsto L^{-1}$. The variety $\mathcal{K}$ is a compact Kummer surface with 16 singularities, corresponding to the fixed points of the involution, whose desingularization is a K3 surface obtained by blowing up $\mathcal{K}$ in the singular points.
As a result, the locus $\mathcal{S}$ of stable bundles is the complement of $\mathcal{K}$ inside $\mathbb{P}^3$. 

Observe that the cohomology of $\mathcal{K}$ is the $\mathbb{Z}_2$ invariant part of the cohomology of $Jac(C)$. The cohomology of the Jacobian is generated by $H^1(Jac(C))$ and the Betti numbers are 
$$ b_0=1\quad b_1=4 \quad b_2=6 \quad b_3=4 \quad b_4=1.$$
The action of $\mathbb{Z}_2$ on the cohomology sends every generator $\gamma$ of $H^1$ in $-\gamma$ thus the even cohomology groups are all $\mathbb{Z}_2$-invariant, while the odd ones are never. As a result, the Betti numbers of $\mathcal{K}$ are 
$$ b_0=1\quad b_1=0 \quad b_2=6 \quad b_3=0 \quad b_4=1.$$
Alternatively, one can notice that the cohomology of $\mathcal{K}$ differs from the one of its desingularization just in the $H^2$ part, which has in addition the contribution of the 16 exceptional divisors isomorphic to $\mathbb{P}^1$, and the Betti numbers of a K3 surface are
$$ b_0=1\quad b_1=0 \quad b_2=22 \quad b_3=0 \quad b_4=1.$$
Observe that since the Hodge structure on the cohomology of $Jac(C)$ is pure and so is the cohomology of $\mathcal{K}$. In particular $ H^0(\mathcal{K})$ has weight 0, $H^2(\mathcal{K})$ has weight 2 and types $4(1,1)+(2,0)+(0,2)$, and $H^4(\mathcal{K})$ has weight 4 of type $(2,2)$.
Consider now the inclusions $\xymatrix{ 
	\mathcal{S}\ar@{^{(}->}[r]^j & \mathbb{P}^3 &\mathcal{J}\ar@{_{(}->}[l]_i  
}
$.
As both $\mathbb{P}^3$ and $\mathcal{K}$ are compact, we have the long exact sequence:
\begin{equation*}
\begin{tikzpicture}
\node (O) at (0,0) {$\cdots$};
\node (A) at (2,0) {$H^k_c(\mathcal{S})$};
\node (B) at (4,0) {$H^k(\mathbb{P}^3)$};
\node (C) at (6.5,0) {$H^k(\mathcal{K})$};
\node (O') at (8,0) {$\cdots$};
\draw[->,thick] (O) -- (A) node [midway,above] {};
\draw[->,thick] (A) -- (B) node [midway,above] {$j_{!}$};
\draw[->,thick] (B) -- (C) node [midway,above] {$i^!$};
\draw[->,thick] (C) -- (O') node [midway,above] {};
\end{tikzpicture}
\end{equation*}
which splits in the following sequences
\begin{equation}\label{eq1}
\begin{tikzpicture}[baseline]
\node (O) at (0,0) {0};
\node (A) at (2,0) {$H^0_c(\mathcal{S})$};
\node (B) at (4,0) {$\co$};
\node (C) at (6.0,0) {$\co$};
\node (D) at (8,0) {$H^1_c(\mathcal{S})$};
\node (E) at (10,0) {0;};
\draw[->,thick] (O) -- (A) node [midway,above] {};
\draw[->,thick] (A) -- (B) node [midway,above] {$j_{!}$};
\draw[->,thick] (B) -- (C) node [midway,above] {$i^!$};
\draw[->,thick] (C) -- (D) node [midway,above] {};
\draw[->,thick] (D) -- (E) node [midway,above] {};
\end{tikzpicture}
\end{equation}

\begin{equation}\label{eq2}
\begin{tikzpicture}[baseline]
\node (O) at (0,0) {0};
\node (A) at (2,0) {$H^2_c(\mathcal{S})$};
\node (B) at (4,0) {$\co$};
\node (C) at (6.0,0) {$\co^6$};
\node (D) at (8,0) {$H^3_c(\mathcal{S})$};
\node (E) at (10,0) {0;};
\draw[->,thick] (O) -- (A) node [midway,above] {};
\draw[->,thick] (A) -- (B) node [midway,above] {$j_{!}$};
\draw[->,thick] (B) -- (C) node [midway,above] {$i^!$};
\draw[->,thick] (C) -- (D) node [midway,above] {};
\draw[->,thick] (D) -- (E) node [midway,above] {};
\end{tikzpicture}
\end{equation}
%

\begin{equation}\label{eq3}
\begin{tikzpicture}[baseline]
\node (O) at (0,0) {0;};
\node (A) at (2,0) {$H^4_c(\mathcal{S})$};
\node (B) at (4,0) {$\co$};
\node (C) at (6.0,0) {$\co$};
\node (D) at (8,0) {$H^5_c(\mathcal{S})$};
\node (E) at (10,0) {0};
\draw[->,thick] (O) -- (A) node [midway,above] {};
\draw[->,thick] (A) -- (B) node [midway,above] {$j_{!}$};
\draw[->,thick] (B) -- (C) node [midway,above] {$i^!$};
\draw[->,thick] (C) -- (D) node [midway,above] {};
\draw[->,thick] (D) -- (E) node [midway,above] {};
\end{tikzpicture}
\end{equation}

\begin{equation}\label{eq4}
\begin{tikzpicture}[baseline]
\node (O) at (0,0) {0};
\node (A) at (2,0) {$H^6_c(\mathcal{S})$};
\node (B) at (4,0) {$\co$};
\node (C) at (6,0) {0.};
\draw[->,thick] (O) -- (A) node [midway,above] {};
\draw[->,thick] (A) -- (B) node [midway,above] {};
\draw[->,thick] (B) -- (C) node [midway,above] {};
\end{tikzpicture}
\end{equation}

First consider \eqref{eq1}: the map $i^{!}=i^*$ is a restriction to a hyperplane sections, therefore it is an isomorphism by Lefschetz hyperplane theorem and $H_c^0(\mathcal{S})=H_c^1(\mathcal{S})=0$.\\
For \eqref{eq2}, $i^{!}$ is the restriction of the fundamental class of $\mathbb{P}^1$ inside $\mathbb{P}^2$ which remains non-zero when intersecting it generically with $\mathcal{K}$, so $i^{!}$ is an injection. One has $H^2_c(\mathcal{S})=0$ and $H^3_c(\mathcal{S})=\co^5$.
A similar argument shows that, in \eqref{eq3}, $i^{!}$ is an isomorphism and that $H_c^4(\mathcal{S})=H_c^5(\mathcal{S})=0$.\\
Clearly, \eqref{eq4} shows that $H^6_c(\mathcal{S})\cong \co$.
By Poincaré duality the Betti numbers are
$$ b_0=1\quad b_1=0 \quad b_2=0 \quad b_3=5 \quad b_4=0 \quad b_5=0 \quad b_6=0.$$

As $T^*\mathcal{S}$ is a vector bundle over $\mathcal{S}$, it inherits the cohomology of its base space, so the compact support cohomology groups of $T^*\mathcal{S}$ are 
\begin{align*}
H^9_c(T^*\mathcal{S})&=5 \quad \text{of types }(3,5)+(5,3)+3(4,4)\\
H^{12}_c(T^*\mathcal{S})&=1  \quad \text{of type }(6,6)\\
H^i_c(T^*\mathcal{S})&=0 \quad \text{otherwise.}
\end{align*}

As a result, the $E$-polynomial of Higgs bundles with stable underlying vector bundle is 
$$ E(T^*\mathcal{S})(u,v)=u^6v^6-u^3v^5-u^5v^3-3u^4v^4.$$
\end{proof}

\subsection{Strictly semistable case}

Suppose $V$ is a strictly semistable vector bundle. We would like to investigate when $V$ occurs in a stable Higgs pair $(V,\Phi)$.
Again, one has to distinguish different cases:
\begin{enumerate}[(i)]
\item $V=L\oplus L^{-1}$ where $L\in Jac(C)$ and $L\not\cong L^{-1}$;
\item $V$ is a non trivial extension $\xymatrix{
	0\ar[r] &L \ar[r] &V\ar[r] & L^{-1}\ar[r]&0
} $ with $L\not\cong L^{-1}$;
\item $V=L\oplus L$ where $L\in Jac(C)$ and $L\cong L^{-1}$;
\item $V$ is a non trivial extension $\xymatrix{
	0\ar[r] &L \ar[r] &V\ar[r] & L\ar[r]&0
} $ with $L\cong L^{-1}$;
\end{enumerate}

\subsubsection{Type $\mathrm{(i)}$}

We shall determine stable Higgs pairs having underlying vector bundle of $V=L\oplus L^{-1}$ with $L\in Jac(C)$ such that $L\not\cong L^{-1}$. Vector bundles of this form are parametrized by $\mathcal{K}=Jac(C)/\mathbb{Z}_2$. We denote by $\mathcal{K}_0$ locus in $\mathcal{K}$ fixed by the involution and by $\mathcal{K}^0:=\mathcal{K}-\mathcal{K}_0$ its complement. 
Then the locus of stable Higgs bundles with underlying vector bundle of type (i) will is a fibre bundle over $\mathcal{K}^0$. 

\begin{prop}
Let $\mathcal{S}_1$ be the locus of stable Higgs bundles with underlying vector bundle of type $\mathrm{(i)}$. Then $\mathcal{S}_1$ is a $(\co^2\times \co^*)$-bundle over $\mathcal{K}^0$ and its $E-$polynomial is 
$$ E(\mathcal{S}_1)(u,v)=u^5v^5+u^3v^5+u^5v^3+3u^4v^4-19u^3v^3-u^2v^4-u^4v^2+15 u^2v^2.$$
\end{prop}

\begin{proof}

Consider $V=L\oplus L^{-1}$ with $L\in Jac(C)$ such that $L\not \cong L^{-1}$. 
We have that
$$ H^0(End_0(V)\otimes K_C)= H^0(K_C)\oplus H^0(L^2K_C)\oplus H^0(L^{-2}K_C)$$
thus a Higgs field $\Phi\in H^0(End_0(V)\otimes K_C)$ will be of the form 
$$\Phi=\left( \begin{array}{cc}
a&b\\
c&-a\\
\end{array}\right)$$
with $a\in H^0(K_C)$, $b\in H^0(L^2K_C)$, $c\in H^0(L^{-2}K_C)$. A pair $(V,\Phi)$ is stable if and only if both $L$ and $L^{-1}$ are not preserved by $\Phi$, that is $b,c\neq 0$.
Then one needs to understand when two different Higgs fields give rise to isomorphic Higgs bundles: since the automorphisms group of $V$ is $\co^*$, two Higgs pairs $(V,\Phi_1)$ and $(V,\Phi_2)$ with $\Phi_i=(a_i,b_i,c_i)$ are isomorphic if and only if
$$\Phi_1=\left( \begin{array}{cc}
t&0\\
0&t^{-1}\\
\end{array}\right)\Phi_2\left( \begin{array}{cc}
t^{-1}&0\\
0&t\\
\end{array}\right)$$
that is $a_1=a_2$, $b_1=t^2 b_2$, $c_1=t^{-2}c_2$. Therefore, stable Higgs pairs $(V,\Phi)$ with fixed underline vector bundle $V$ are parametrized by
$$H^0(K_C)\times \dfrac{(H^0(L^2K)-\{0\})\times (H^0(L^{-2}K_C)-\{0\})}{\co^*} \cong\co^2\times \co^*.$$
This is an actual quotient as all the points are semistable with respect to the action of $\co^*$.\\
 Letting $V$ vary, one obtains a $(\co^2\times \co^*)$-bundle $\mathcal{S}_1$ over $\mathcal{K}^0$. 
Considering $\mathcal{S}_1$ as a sphere bundle over $\mathcal{K}^0$, the cohomology of the total space is computed by Gysin sequence.
Consider the inclusions  $\xymatrix{ 
	\mathcal{K}^0\ar@{^{(}->}[r]^j & \mathcal{K} &\mathcal{K}_0\ar@{_{(}->}[l]_i  
}
$ and the long exact sequence in cohomology 
\begin{equation*}
\begin{tikzpicture}
\node (O) at (0,0) {$\cdots$};
\node (A) at (2,0) {$H^k_c(\mathcal{K}^0)$};
\node (B) at (4,0) {$H^k(\mathcal{K})$};
\node (C) at (6,0) {$H^k(\mathcal{K}_0)$};
\node (D) at (8,0) {$\cdots$};
\draw[->,thick] (O) -- (A) node [midway,above] {};
\draw[->,thick] (A) -- (B) node [midway,above] {$j_{!}$};
\draw[->,thick] (B) -- (C) node [midway,above] {$i^{!}$};
\draw[->,thick] (C) -- (D) node [midway,above] {};
\end{tikzpicture}
\end{equation*}
which splits in
\begin{equation*}
\begin{tikzpicture}
\node (O) at (0,0) {0};
\node (A) at (2,0) {$H^0_c(\mathcal{K}^0)$};
\node (B) at (4,0) {$\co$};
\node (C) at (6,0) {$\co^{16}$};
\node (D) at (8,0) {$H^1(\mathcal{K}_0)$};
\node (E) at (10,0) {0,};
\node (F) at (12,0) {$H^k_c(\mathcal{K}^0)\cong H^k_c(\mathcal{K})$};
\node (G) at (14.5,0) {$\forall k\geq 2$};
\draw[->,thick] (O) -- (A) node [midway,above] {};
\draw[->,thick] (A) -- (B) node [midway,above] {};
\draw[->,thick] (B) -- (C) node [midway,above] {$i^{!}$};
\draw[->,thick] (C) -- (D) node [midway,above] {};
\draw[->,thick] (D)--(E) node [midway,above] {};
\end{tikzpicture}
\end{equation*}

As $\mathcal{K}^0$ is not compact $H^0_c(\mathcal{K}^0)=0$, so $H^1_c(\mathcal{K}^0)\cong \co^{15}$. By Poincaré duality
$$H^0(\mathcal{K}^0)\cong\co\quad H^1(\mathcal{K}^0)=0\quad H^2(\mathcal{K}^0)=\co^6\quad H^3(\mathcal{K}^0)=\co^{15}\quad H^4(\mathcal{K}^0)=0$$
with the same weights as the cohomology of $\mathcal{K}$.\\
The Gysin sequence 
\begin{equation*}
\begin{tikzpicture}
\node (O) at (0,0) {$\cdots$};
\node (A) at (2.5,0) {$H^k(\mathcal{S}_1)$};
\node (B) at (5,0) {$H^{k-1}(\mathcal{K}^0)$};
\node (C) at (7.5,0) {$H^{k+1}(\mathcal{K}^0)$};
\node (D) at (10,0) {$\cdots$};
\draw[->,thick] (O) -- (A) node [midway,above] {};
\draw[->,thick] (A) -- (B) node [midway,above] {};
\draw[->,thick] (B) -- (C) node [midway,above] {};
\draw[->,thick] (C) -- (D) node [midway,above] {};
\end{tikzpicture}
\end{equation*}
splits in the following subsequences
\begin{equation}
H^0(\mathcal{S}_1)\cong \co \quad H^3(\mathcal{S}_1)\cong\co^{21} \quad H^{4}(\mathcal{S}_1)\cong\co^{15};
\end{equation}
\begin{equation}\label{eul1}
\begin{tikzpicture}[baseline]
\node (O) at (0,0) {0};
\node (A) at (2,0) {$H^1(\mathcal{S}_1)$};
\node (B) at (4,0) {$\co$};
\node (C) at (6,0) {$\co^{6}$};
\node (D) at (8,0) {$H^2(\mathcal{S}_1)$};
\node (E) at (10,0) {0;};
\draw[->,thick] (O) -- (A) node [midway,above] {};
\draw[->,thick] (A) -- (B) node [midway,above] {};
\draw[->,thick] (B) -- (C) node [midway,above] {};
\draw[->,thick] (C) -- (D) node [midway,above] {};
\draw[->,thick] (D)--(E) node [midway,above] {};
\end{tikzpicture}
\end{equation}
\begin{equation}
H^i(\mathcal{S}_1)=0\quad  \forall i\geq 5.
\end{equation}
In (\ref{eul1}) the map $\co\rightarrow \co^6$ is the product by the Euler class of a nontrivial bundle, which is non zero, therefore $H^1(\mathcal{S}_1)=0$ and $H^2(\mathcal{S}_1)=\co^5$. 
Recall that in this case both the cup product with the Euler class and the pushforward increase weights of (1,1). As a consequence of Poincaré duality the compact support cohomology groups of $\mathcal{S}_1$ are 
\begin{align*}
H^i_c(\mathcal{S}_1)&=0 &\forall i=0,\ldots 5 \text{ and }i=9&\quad\\
H^6_c(\mathcal{S}_1)&=\co^{15}  & \text{ of type }(2,2)&\quad \\
H^7_c(\mathcal{S}_1)&=\co^{21} &\text{ of types }19(3,3)+(2,4)+(4,2)&\\
H^8_c(\mathcal{S}_1)&=\co^5 & \text{ of types }3(4,4)+(3,5)+(5,3)&\\
\quad H^{10}_c(\mathcal{S}_1)&=\co &\text{ of type }(5,5),&
\end{align*}
and the $E-$polynomial of $\mathcal{S}_1$ is 
$$ E(\mathcal{S}_1)(u,v)=u^5v^5+u^3v^5+u^5v^3+3u^4v^4-19u^3v^3-u^2v^4-u^4v^2+15 u^2v^2.$$
\end{proof}
\subsubsection{Type $\mathrm{(ii)}$}

Now we want to compute the cohomology of the locus of stable pairs $(V,\Phi)$ where $V$ is a nontrivial extension of $L$ by $L^{-1}$ with $L\not\cong L^{-1}$. 

\begin{prop}
Let $V$ be a semistable vector bundle of type $\mathrm{(ii)}$. Then there is no Higgs field $\Phi$ such that the pair $(V,\Phi)$ is stable.
\end{prop}
\begin{proof}
Consider the universal line bundle $\luniv\rightarrow \mathcal{K}^0\times C$ and let $p:\mathcal{K}^0\times C \rightarrow \mathcal{K}^0$ be the projection onto the first factor. It is well known that non trivial extensions of $\luniv$ by $\luniv^{-1}$ are parametrized by $\mathbb{P}(R^1p_{*}\luniv^2)$: as  $R^1p_{*}\luniv^2$ is a local system on $\mathcal{K}^0$ of rank one, there exists a unique nontrivial extension up to isomorphism. Consider the universal extension bundle $\vuniv$: it fits in the short exact sequence
\begin{equation}\label{ses}
0\rightarrow \luniv\rightarrow \vuniv \rightarrow \luniv^{-1}\rightarrow 0
\end{equation}
and parametrizes all the vector bundles $V$ on $C$ of type (ii).
In order for $(\vuniv,\Phi)$ to be a stable Higgs bundle, the Higgs field must not preserve the subbundle $\luniv$.
By an abuse notation, we denote by $K_C$ the pullback of the canonical bundle on $C$ under the projection $\mathcal{K}^0\times C\rightarrow C$. Tensoring the sequence (\ref{ses}) by $K_C$ and applying the covariant functor $Hom(\vuniv,-)$ restricted to traceless endomorphisms gives
$$0\rightarrow Hom(\vuniv,\luniv\otimes K_C) \rightarrow End_0(\vuniv)\otimes K_C \rightarrow Hom(\vuniv,\luniv^{-1}\otimes K_C)\rightarrow 0.$$ 
Pushing forward to $\jup$ gives a long exact sequence
\begin{equation}\label{les1}
\begin{tikzpicture}[baseline, descr/.style={fill=white,inner sep=1.5pt}]
\matrix (m) [
matrix of math nodes,
row sep=1em,
column sep=2.5em,
text height=1.5ex, text depth=0.25ex
]
{ 0 &p_*Hom (\vuniv,\luniv K_C) & p_* End_0(\vuniv)\otimes K_C &p_*\luniv^{-2}K_C &\\
	& R^1p_*Hom(\vuniv,\luniv K_C) & R^1p_* Hom(\vuniv,\luniv K_C) & R^1p_*\luniv^{-2}K_C&0. \\
};

\path[overlay,->, font=\scriptsize,>=latex]
(m-1-1) edge (m-1-2)
(m-1-2) edge (m-1-3)
(m-1-3) edge node[above] {$\rho$} (m-1-4) 
(m-1-4) edge[out=355,in=175] node[descr,yshift=0.3ex] {$ext$} (m-2-2)
(m-2-2) edge (m-2-3)
(m-2-3) edge (m-2-4)
(m-2-4) edge (m-2-5);
\end{tikzpicture}
\end{equation}
A Higgs pair $(\vuniv,\Phi)$ is stable if and only if the Higgs field $\Phi$ lies in the complement of the kernel of the restriction map $\rho: p_* End_0(\vuniv)\otimes K_C\rightarrow p_*\luniv^{-2}K_C$.
In order to prove the proposition, we show that $\rho=0$. \\
Applying the contravariant functor $p_*Hom(-,\luniv K_C)$ to \eqref{ses} gives a long exact sequence

\begin{equation}\label{les2}
\begin{tikzpicture}[baseline, descr/.style={fill=white,inner sep=1.5pt}]
\matrix (m) [
matrix of math nodes,
row sep=1em,
column sep=2.5em,
text height=1.5ex, text depth=0.25ex
]
{ 0 & p_*\luniv K_C & p_* Hom(\vuniv,\luniv K_C) &p_*K_C &\\
	& R^1p_*\luniv K_C) & R^1p_* Hom(\vuniv,\luniv K_C) & R^1p_*K_C &0. \\
};

\path[overlay,->, font=\scriptsize,>=latex]
(m-1-1) edge (m-1-2)
(m-1-2) edge (m-1-3)
(m-1-3) edge (m-1-4)
(m-1-4) edge[out=355,in=175] node[descr,yshift=0.3ex] {$ext$} (m-2-2)
(m-2-2) edge (m-2-3)
(m-2-3) edge (m-2-4)
(m-2-4) edge (m-2-5);
\end{tikzpicture}
\end{equation}
Consider the fibre of (\ref{les2})  on a point $L\in \mathcal{K}^0$. One has
\begin{equation*}
\begin{tikzpicture}
\node (O) at (0,0) {$H^1(L^2K_C)$};
\node (A) at (3,0) {$H^1(V^*LK_C)$};
\node (B) at (6,0) {$H^1(K_C)$};
\node (C) at (8,0) {0.};
\draw[->,thick] (O) -- (A) node [midway,above] {};
\draw[->,thick] (A) -- (B) node [midway,above] {};
\draw[->,thick] (B) -- (C) node [midway,above] {};
\end{tikzpicture}
\end{equation*}
As $H^1(L^2K_C)=0$, then $H^1(V^*LK_C)\cong H^1(K_C)\cong \co$, thus $R^1p_*Hom(\vuniv,\luniv K_C)$ is a local system of rank 1 on $\mathcal{K}^0\times C$.
Now consider (\ref{les1}) in the fibre over $L\in \mathcal{K}^0$:

\begin{equation*}
\begin{tikzpicture}[baseline, descr/.style={fill=white,inner sep=1.5pt}]
\matrix (m) [
matrix of math nodes,
row sep=1em,
column sep=2.5em,
text height=1.5ex, text depth=0.25ex
]
{ 0 & H^0(V^*LK_C) & H^0(End_0(V)\otimes K_C) &H^0(L^{-2}K_C) &\\
	& H^1(V^*LK_C) & H^1(End_0(V)\otimes K_C) & H^1(L^{-2}K_C) &0. \\
};

\path[overlay,->, font=\scriptsize,>=latex]
(m-1-1) edge (m-1-2)
(m-1-2) edge (m-1-3)
(m-1-3) edge (m-1-4)
(m-1-4) edge[out=355,in=175] node[descr,yshift=0.3ex] {$ext$} (m-2-2)
(m-2-2) edge (m-2-3)
(m-2-3) edge (m-2-4)
(m-2-4) edge (m-2-5);
\end{tikzpicture}
\end{equation*}

As seen before, $H^1(V^*LK_C)\cong H^1(K_C)\cong\co$ and $H^0(L^{-2}K_C)\cong \co$, so the map \ita{ext} in \eqref{les1} is either 0 or an isomorphism. However, as $V$ is a nontrivial extension, such a map has to be non zero, thus it is an isomorphism. As a result $\rho$ is zero.\\
\end{proof}

\subsubsection{Type $\mathrm{(iii)}$}
We consider stable Higgs bundles with underlying vector bundle $V=L\oplus L$ with $L\cong L^{-1} \in \mathcal{K}_0$.

\begin{prop} Let $\gr{S}_3$ be the locus of stable Higgs bundles with underlying vector bundle $V=\mathcal{O}\oplus \mathcal{O}$. Then the locus $\mathcal{S}_3$ of stable Higgs pairs of type $\mathrm{(iii)}$ is the union of 16 copies of $\gr{S}_3$ and its $E$-polynomial is
$$ E(\mathcal{S}_3)(u,v)=16u^3v^3-16u^2v^2$$
\end{prop}
\begin{proof}
Up to tensor by $L\in \mathcal{K}_0$ one may restrict to the case $L=\mathcal{O}$, so that $V$ is just the trivial bundle $\mathcal{O}\oplus\mathcal{O}$. In this case $H^0(\mathcal{E}nd_0(V)\otimes K_C)\cong H^0(K_C)\otimes \sla\cong \co^2\otimes \sla$ and the Higgs field is of the form $$\Phi=\left( \begin{array}{cc}
a&b\\
c&-a\\
\end{array}\right) \text{ with }  a,b,c\in H^0(K_C)$$
The bundle is not stable if and only if $\Phi$ is conjugate to an upper triangular matrix of elements of $H^0(K_C)$. As the action of $\mathrm{SL}(2,\co) $ on $H^0(K_C)\otimes \sla$ is trivial on the first factor, one can consider $\Phi\in H^0(K_C)\otimes \sla $ as a pair of matrices $(A,B)\in \sla\oplus\sla$ on which $\mathrm{SL}(2,\co)$ acts by simultaneous conjugation. Then $\Phi$ is conjugate to an upper triangular matrix of elements of $H^0(K_C)$ if and only if $A$ and $B$ are simultaneously triangulable, i.e. if they have a common eigenspace. By a result of Shemesh \cite{Sh} this happens if and only if $\ker[A,B]\neq 0$, that is $\det([A,B])=0$.
Writing \begin{equation}\label{mat}
A=\left( \begin{array}{cc}
x_1&x_2\\
x_3&-x_1\\
\end{array}\right)\qquad B=\left( \begin{array}{cc}
y_1&y_2\\
y_3&-y_1\\
\end{array}\right)\end{equation}
then
$$[A,B]=\left( \begin{array}{cc}
x_2y_3-y_2x_3&2(x_1y_2-x_2y_1)\\
2(x_3y_1-x_1y_3)&-(x_2y_3-y_2x_3)\\
\end{array}\right).$$
One can interpret the locus of simultaneously triangulable matrices $(A,B)\in \sla\oplus \sla$ as a quartic hypersurface in $\co^6$ with coordinates $(x_1,x_2,x_3,y_1,y_2,y_3)$ with equation
$$Q: (x_2y_3-y_2x_3)^2+4(x_1y_2-x_2y_1)(x_3y_1-x_1y_3)=0, $$
given by annihilation of $\det[A,B]$.

\begin{lem}
	A Higgs bundle $(V,\Phi)$ with $V=\mathcal{O}\oplus\mathcal{O}$ is stable if and only if $\Phi$ lies in 
	$$ \gr{S}_3:=(\co^6-Q)\sslash \mathrm{SL}(2,\co)$$ 
	where $\mathrm{SL}(2,\co)$ acts as the simultaneous conjugation on the matrices $A$ and $B$ as in (\ref{mat}).
\end{lem}
\begin{cor}
	The locus of stable Higgs bundles of type $\mathrm{(iii)}$ is isomorphic to 16 copies of $\gr{S}_3$, one for each point of $\mathcal{K}_0$. 
\end{cor}

Consider the quartic hypersurface $Q$ in $\co^6$. 
Setting
\begin{align*}
\alpha &=x_2y_3-y_2x_3\\
\beta &=x_1y_2-x_2y_1\\
\gamma &=x_3y_1-x_1y_3
\end{align*}
then for every $(x_1,x_2,x_3,y_1,y_2,y_3)\in Q$, $(\alpha,\beta,\gamma)$ satisfy the equation
$$\alpha^2+4\beta\gamma=0.$$
This defines a map $f:Q \rightarrow\mathcal{C}:=\{(\alpha,\beta,\gamma)\in \co^3\mid \alpha^2+4\beta\gamma=0\}$
$$
f(x_1,x_2,x_3,y_1,y_2,y_3)= (x_2y_3-y_2x_3,x_1y_2-x_2y_1,x_3y_1-x_1y_3).
$$
Our strategy to compute the cohomology of $(\co^6-Q)$ is the following:
\begin{enumerate}[1)]
	\item we decompose $Q$ as a disjoint union of the close set $Q_0=f^{-1}(0)$ and its open complement $Q-Q_0= f^{-1}(\mathcal{C}-\{0\})$;
	\item  we compute the cohomology with compact support of both $Q_0$ and $Q-Q_0$ and use the additivity property to compute the cohomology with compact support of $Q$;
	\item again, as $\co^6=Q\sqcup(\co^6-Q)$, by the additivity property of the cohomology with compact support one can compute the cohomology of $\co^6-Q$.
\end{enumerate}

Observe that $\alpha,\beta,\gamma$ are the minors of order 2 of the matrix
\begin{equation}\label{minors}\left(\begin{array}{cc}
x_1 &y_1\\
x_2 &y_2\\
x_3 &y_3\\
\end{array}
\right). 
\end{equation}
Also, fixing $(x_1,x_2,x_3, y_1,y_2,y_3)$ and the corresponding point $(\alpha,\beta,\gamma) \in \mathcal{C}$, one notice that both $(x_1,x_2,x_3)$ and $(y_1,y_2,y_3)$ are orthogonal to $(\alpha,\frac{\gamma}{2},\frac{\beta}{2})$, i.e. they satisfy the equations
$$2\alpha x_1+\gamma x_2+\beta x_3=0\qquad 2\alpha y_1+\gamma y_2+\beta y_3=0$$

If $(\alpha,\beta,\gamma)\neq (0,0,0)$, say $\beta\neq 0$, then 
$$x_3=\dfrac{-2\alpha x_1-\gamma x_2}{\beta}, \qquad y_3=\dfrac{-2\alpha y_1-\gamma y_2}{\beta}.$$
Substituting these values in (\ref{minors}) and annihilating the minors of order two gives three equations all identical to 
$$x_1y_2-x_2y_1=\dfrac{\beta}{2}.$$

Therefore the fibre of the map  $f$ in a point of $\mathcal{C}-\{0\}$ is a quadric in $\co^4$, isomorphic to $\mathrm{SL}(2,\co)$. Also $\mathcal{C}-\{0\}$ is homotopy equivalent to $\re\mathbb{P}^3$, thus it has fundamental group $\mathbb{Z}_2$ and the monodromy outside the origin is trivial as it the same as the one described in \cite[Section 3.1]{FK}. As a result, we can compute the cohomology with compact support of $Q-Q_0=f^{-1}(\mathcal{C}-0)$ via K\"unneth formula. 
We have:
$$ \quad H^4_c(Q-Q_0)=\co\quad H^7_c(Q-Q_0)=\co^2\quad H^{10}_c(Q-Q_0)=\co \quad   H^i_c(Q-Q_0)=0 \text{ otherwise }.$$

Now, we need to compute the cohomology of $Q_0$: first observe that the condition $\alpha,\beta,\gamma=0$ implies that the matrix (\ref{minors}) has rank $\leq 1$, that is, $(y_1,y_2,y_3)$ is a multiple of $(x_1,x_2,x_3)$. Thus points in $Q_0$ are parametrized by $(\co^3-\{0\})\times \co \sqcup \{0\}\times \co^3$. We observe that $Q_0$ has dimension 4 and the former is an open set in it, while the latter is closed. Therefore by proposition \ref{ap} one has
$$\begin{array}{lll}
H_c^3((\co^3-\{0\})\times \co)\cong\co &\quad H^8_c((\co^3-\{0\})\times \co)\cong \co\quad H^i_c((\co^3-\{0\})\times \co)=0 \text{ otherwise}\\
H^6_c(\{0\}\times\co^3)\cong \co &\quad H^i_c(\{0\}\times \co^3)=0 \text{ otherwise}. &
\end{array},$$
Hence $$H^3_c(Q_0)\cong H^6_c(Q_0)\cong H^8_c(Q_0)\cong\co, \quad  H^i_c(Q_0)=0 \text{ otherwise} .$$

Again, by proposition \ref{ap} there is a long exact sequence 
\begin{equation*}
\begin{tikzpicture}[baseline]
\node (O) at (0,0) {$\cdots$};
\node (A) at (2.5,0) {$H^i_c(Q-Q_0)$};
\node (B) at (5,0) {$H^i_c(Q)$};
\node (C) at (7.5,0) {$H^i_c(Q_0)$};
\node (O') at (10,0) {$\cdots$};
\draw[->,thick] (O) -- (A) node [midway,above] {};
\draw[->,thick] (A) -- (B) node [midway,above] {};
\draw[->,thick] (B) -- (C) node [midway,above] {};
\draw[->,thick] (C) -- (O') node [midway,above] {};
\end{tikzpicture}
\end{equation*}

Now, $H^i_c(Q)=0$ for any $i\geq 5$ since $Q$ is affine and from the long exact sequence one can conclude that $H^7_c (Q)\cong H^8_c(Q)\cong H^{10}_c(Q)\cong\co$ and $H^i_c(Q)=0$ otherwise. 

Finally from proposition \ref{ap} one gets the cohomology with compact support of $\co^6\setminus Q$:
$$H^8_c (\co^6-Q)\cong H^9_c(\co^6- Q)\cong H^{11}_c(\co^6-Q)\cong H^{12}_c(\co^6-Q)\cong \co,\quad H^i_c(\co^6-Q)=0.$$

Observe that $\mathrm{SL}(2,\co$) acts on $\co^6-Q$ with a stabilizer which is at worst $\mathbb{Z}_2$, therefore one can compute its cohomology by considering it as a fibre bundle with fibre $\mathrm{SL}(2,\co)$ on $\mathcal{S}_3$


As $\mathrm{SL}(2,\co)$ has the same homotopy type as $S^3$, the Gysin sequence 
\begin{equation*}
\begin{tikzpicture}[baseline]
\node (O) at (0,0) {$\cdots$};
\node (A) at (2.5,0) {$H^i(\co^6-Q)$};
\node (B) at (5,0) {$H^{i-3}(\gr{S}_3)$};
\node (C) at (7.5,0) {$H^{i+1}(\gr{S}_3)$};
\node (O') at (10,0) {$\cdots$};
\draw[->,thick] (O) -- (A) node [midway,above] {};
\draw[->,thick] (A) -- (B) node [midway,above] {};
\draw[->,thick] (B) -- (C) node [midway,above] {};
\draw[->,thick] (C) -- (O') node [midway,above] {};
\end{tikzpicture}
\end{equation*}
gives
\begin{equation}
H^0(\gr{S}_3)\cong H^1(\gr{S}_3)\cong\co\qquad  H^2(\gr{S}_3)=0
\end{equation}
\begin{equation}\label{lses}
\begin{tikzpicture}[baseline]
\node (O) at (0,0) {0};
\node (A) at (1.5,0) {$H^3(\gr{S}_3)$};
\node (B) at (3,0) {$\co$};
\node (C) at (4.5,0) {$\co$};
\node (D) at (6,0) {$H^4(\gr{S}_3)$};
\node (E) at (7.5,0) {$\co$};
\node (F) at (9,0) {$\co$};
\node (G) at (10.5,0) {$H^5(\gr{S}_3)$};
\node (O') at (12,0) {0};
\draw[->,thick] (O) -- (A) node [midway,above] {};
\draw[->,thick] (A) -- (B) node [midway,above] {};
\draw[->,thick] (B) -- (C) node [midway,above] {};
\draw[->,thick] (C) -- (D) node [midway,above] {};
\draw[->,thick] (D) -- (E) node [midway,above] {};
\draw[->,thick] (E) -- (F) node [midway,above] {};
\draw[->,thick] (F) -- (G) node [midway,above] {};
\draw[->,thick] (G) -- (O') node [midway,above] {};
\end{tikzpicture}
\end{equation}
\begin{equation}
H^6(\gr{S}_3)=0,\quad H^4(\gr{S}_3)\cong H^8(\gr{S}_3)\cong H^{12}(\gr{S}_3)
\end{equation}
\begin{equation}
H^3(\gr{S}_3)\cong H^7(\gr{S}_3)\cong H^{11}(\gr{S}_3)\quad H^5(\gr{S}_3)\cong H^9(\gr{S}_3)\quad H^6(\gr{S}_3)\cong H^{10}(\gr{S}_3)=0.
\end{equation}
Since $\gr{S}_3$ is nonsingular connected but not compact, $H^{12}(\gr{S}_3)\cong H^0_c(\gr{S}_3)=0$, thus  $H^4(\gr{S}_3)\cong H^8(\gr{S}_3)=0$. Therefore from \eqref{lses} we deduce that $H^3(\gr{S}_3)\cong  H^5(\gr{S}_3)=0$, $H^7(\gr{S}_3)\cong  H^{11}(\gr{S}_3)=0$ and $H^9(\gr{S}_3)=0$ and the $E-$polynomial of $\gr{S}_3$ is 
$$ E(\gr{S}_3)(u,v)=u^3v^3-u^2v^2$$
\end{proof}
\subsubsection{Type $\mathrm{(iv)}$}

We now consider stable Higgs bundles of type $\mathrm{(iv)}$ and we prove the following result.
\begin{prop}\label{es4}
Let $\gr{S}_4$ be the locus of stable Higgs bundles whose underlying vector is a nontrivial extension of $\mathcal{O}$ by itself. Then the locus $\mathcal{S}_4$ of stable Higgs bundles of type $\mathrm{(iv)}$ is the union of 16 copies of $\gr{S}_4$ and its $E$-polynomial is 
$$ E(\mathcal{S}_4)=16u^4v^4-16u^2v^2$$
\end{prop}	 

\begin{proof}	 
As in type $\mathrm{(iii)}$, it is not restrictive to assume $L\cong \mathcal{O}$. Let $V$ be a nontrivial extensions of $\mathcal{O}$ by itself: the isomorphism classes of such bundles are parametrized by 
\begin{equation}\label{ses4}
\mathbb{P}(\mathrm{Ext}^1(\mathcal{O},\mathcal{O}))\cong \mathbb{P}^1.
\end{equation}
Thus there exists a universal extension bundle on $\mathbb{P}^1\times C$
$$0\rightarrow \mathcal{O}\rightarrow \vuniv\rightarrow\mathcal{O}\rightarrow 0.$$
Let $p\colon\mathbb{P}^1\times C\rightarrow \mathbb{P}^1$ be the projection: as in the type (ii) case one can tensor the short exact sequence above by $K_C$, apply the covariant functor $Hom(\vuniv,-)$ and pushforward to $\mathbb{P}^1$ and end up with the long exact sequence

\begin{equation}\label{*sl}
	\begin{tikzpicture}[baseline, descr/.style={fill=white,inner sep=1.5pt}]
	\matrix (m) [
	matrix of math nodes,
	row sep=1em,
	column sep=2.5em,
	text height=1.5ex, text depth=0.25ex
	]
	{ 0 & p_*Hom(\vuniv,K_C) & p_*(End_0(\vuniv)\otimes K_C) &p_*Hom(\vuniv,K_C)&\\
		& R^1p_*Hom(\vuniv,K_C) & R^1p_*(End_0(\vuniv)\otimes K_C) & R^1p_*Hom(\vuniv,K_C) &0. \\
	};
	
	\path[overlay,->, font=\scriptsize,>=latex]
	(m-1-1) edge (m-1-2)
	(m-1-2) edge (m-1-3)
	(m-1-3) edge (m-1-4)
	(m-1-4) edge[out=355,in=175] node[descr,yshift=0.3ex] {$ext$} (m-2-2)
	(m-2-2) edge (m-2-3)
	(m-2-3) edge (m-2-4)
	(m-2-4) edge (m-2-5);
	\end{tikzpicture}
\end{equation}
Stable Higgs bundles are those with Higgs field in the complement of the kernel of the map
$$p_*(End_0(\vuniv)\otimes K_C)\rightarrow p_*Hom(\vuniv,K_C)$$
or, equivalently, the complement of the image of $p_*Hom(\vuniv,K_C)$ in  $p_*(End_0(\vuniv)\otimes K_C)$.
First notice that $p_*Hom(\vuniv,K_C)\cong p_*K_C$, which is a vector bundle of rank 2 and similarly that $R^1p_*Hom(\vuniv,K_C)\cong R^1p_*K_C$. As the extension is nontrivial, the map \ita{ext} is non zero and its kernel has rank 1. 
Starting from (\ref{ses4}), tensoring with $K_C$, applying the functor $Hom(-,\mathcal{O})$  restricted to traceless endomorphisms and pushing forward to $\mathbb{P}^1$ one obtains another long exact sequence

\begin{equation}\label{**sl}
\begin{tikzpicture}[baseline]
\node (O) at (0,0) {0};
\node (A) at (2.5,0) {$p_*K_C$};
\node (B) at (5,0) {$p_*Hom(\vuniv,K_C)$};
\node (C) at (7.5,0) {$p_*K_C$};
\node (D) at (10,0) {$R^1p_*K_C$};
\node (O') at (12.5,0) {$\cdots$};
\draw[->,thick] (O) -- (A) node [midway,above] {};
\draw[->,thick] (A) -- (B) node [midway,above] {};
\draw[->,thick] (B) -- (C) node [midway,above] {};
\draw[->,thick] (C) -- (D) node [midway,above] {$ext$};
\draw[->,thick] (D) -- (O') node [midway,above] {};
\end{tikzpicture}
\end{equation}
Notice that since $R^1p_*K_C$ has rank 1 and the map $ext$ is non zero, the last map is surjective. Hence, the cokernel of $p_*Hom(\vuniv,K_C)\rightarrow p_*K_C$ has rank 1 and consequently $p_*Hom(\vuniv,K_C)$ has rank 3. 
Returning to \eqref{*sl}, we conclude that $p_*End_0(V)\otimes K_C$ is a vector bundle of rank 4, thus the locus of stable pairs is fibrewise the complement of a hyperplane. \\
Finally, automorphisms have to be taken into account: the group of automorphisms of a nontrivial extension of $\mathcal{O}$ by itself is the additive group $(\co,+)\subset \mathrm{SL}(2,\co)$, and an element $t\in (\co,+)$ acts on the Higgs field $\Phi$ as 
$$t.\Phi=\left(\begin{array}{cc}
1&t\\
0&1
\end{array}\right)\left(\begin{array}{cc}
a&b\\
c&-a
\end{array}\right)\left(\begin{array}{cc}
1&-t\\
0&1
\end{array}\right)=\left(\begin{array}{cc}
1a+tc&b-2ta-t^2c\\
c&-a-tc
\end{array}\right).
$$

\begin{lem}\label{s4}
	$\gr{S}_4$ is a $\co^2$- bundle over a $\co^*$- bundle over $\mathbb{P}^1$. All bundles are Zariski locally trivial. 
\end{lem}

\begin{proof}
	Let $A$ be the kernel of the extension map in \eqref{*sl}, minus the zero section: thus $A$ is a $\co^*$-bundle over $\mathbb{P}^1$. One can think of $p_*(End_0(\vuniv)\otimes K_C)-p_*Hom(\vuniv,K_C)$ as vector bundle of rank 3 over $A$. Similarly, the kernel of the extension map of \eqref{**sl} gives rise to a vector bundle $\mathcal{A}$ over $A$ of rank 1 and the map 
	$$p_*Hom(\vuniv)\rightarrow p_*End_0(\vuniv)\otimes K_C)$$ lifts to a $(\co,+)$-equivariant map 
	$$[p_*(End_0(\vuniv)\otimes K_C)-p_*Hom(\vuniv,K_C)]\rightarrow \mathcal{U}$$ 
	of vector bundles over $A$ whose kernel has rank 2. The automorphism action of $(\co,+)$ on $\mathcal{A}$ is linear and given by $a\mapsto a+tc$, hence the quotient $\mathcal{A}/\co$ is $A$ itself. As the map above is equivariant, one has that
	$$[p_*(End_0(\vuniv)\otimes K_C)-p_*Hom(\vuniv,K_C)]/\co\rightarrow \mathcal{A}/\co\cong A$$ 
	is a vector bundle of rank 2 over $A$.
\end{proof}

\begin{cor}\label{core}
	The locus of stable Higgs bundles of type $\mathrm(iv)$ is isomorphic to 16 copies of $\gr{S}_4$, one for each point of $\jo$. 
\end{cor}

Lemma (\ref{s4}) allows to compute the Betti numbers of $\gr{S}_4$: since $\gr{S}_4$ is homotopy equivalent to a $\co^*$-bundle on $\mathbb{P}^1$, the Gysin sequence reads as
\begin{equation}
H^0(\gr{S}_4) \cong H^0(\mathbb{P}^1) \cong \co;
\end{equation}
\begin{equation}\label{eul2}
\begin{tikzpicture}[baseline]
\node (O) at (0,0) {0};
\node (A) at (2,0) {$H^1(\gr{S}_4) $};
\node (B) at (4,0) {$\co$};
\node (C) at (6,0) {$\co$};
\node (D) at (8,0) {$H^2(\gr{S}_4)$};
\node (E) at (10,0) {0;};
\draw[->,thick] (O) -- (A) node [midway,above] {};
\draw[->,thick] (A) -- (B) node [midway,above] {};
\draw[->,thick] (B) -- (C) node [midway,above] {};
\draw[->,thick] (C) -- (D) node [midway,above] {};
\draw[->,thick] (D)--(E) node [midway,above] {};
\end{tikzpicture}
\end{equation}
\begin{equation}
H^3(\gr{S}_4) \cong H^2(\mathbb{P}^1) \cong \co;
\end{equation}
\begin{equation}
H^i(\gr{S}_4)=0  \text{ for all } i=4\ldots 8.
\end{equation}
As the central map of \eqref{eul2} is the cup product with the Euler class of the bundle $A$, which is nontrivial, therefore it is non-zero and we have $H^1(\gr{S}_4)=H^2(\gr{S}_4)=0$. 
By Poincaré duality, the $E-$polynomial of $\gr{S}_4$ is
$$ E(\gr{S}_4)=u^4v^4-u^2v^2$$
By corollary \ref{core} this completes the proof of proposition \ref{es4}.
\end{proof}
\subsection{Unstable case}
Consider the locus $\mathcal{U}$ of stable Higgs bundles $(V,\Phi)$ where $V$ is an unstable vector bundle with trivial determinant. Then there exists a line bundle $L$ of degree $d>0$ that fits an exact sequence
\[
\xymatrix{ 
0\ar[r] &L\ar[r] &V\ar[r] & L^{-1}\ar[r] &0}
\]

If $d>1$ then the bundle $L^{-2}K_C$ has no non-zero global section because it has negative degree, hence $L$ is $\Phi$-invariant for any Higgs field $\Phi\in H^0( End_0(V)\otimes K_C)$.
The only case we have to check is $\deg(L)=1$. Since the line bundle $L^{-2}K_C$ has degree 0, it has global sections if and only if it is trivial, i.e. is $L=K^{\frac{1}{2}}$ is one of the 16 roots of the canonical bundle $K_C$. 
As a consequence, if there exists an unstable vector bundle $V$ which is stable as a Higgs bundle, then it must be an extension of those bundles by their duals. 

\begin{prop}
The locus $\mathcal{U}$ of stable Higgs bundles $(V,\Phi)$ with $V$ unstable is isomorphic to 16 copies of $\co^3$, one for any root of the canonical bundle $K_C$. As a consequence its cohomology with compact support is given by
$$H^6_c(\mathcal{U})=\co^{16} \quad H^i_c(\mathcal{U})=0 \text{ otherwise.}$$
and the $E$-polynomial of $\mathcal{U}$ is $E(\mathcal{U})=16u^3v^3$.
\end{prop}
\begin{proof}

\gr{Trivial case}
If $
V=K^{\frac{1}{2}}\oplus K^{-\frac{1}{2}}$ then 
$$ H^0(End_0(V)\otimes K_C)\cong H^0(K_C)\oplus H^0(K_C^2)\oplus H^0(\mathcal{O}) \cong \co^2\oplus \co^3\oplus \co $$
Thus the generic Higgs field will be of the form
$$\Phi=\left(\begin{array}{cc}
a&b\\
c&-a\\
\end{array}\right) \text{ with } a\in H^0(K_C), b \in H^0(K^2_C), c\in H^0(\mathcal{O}).$$
Two Higgs fields define isomorphic Higgs bundles if and only if they are conjugate by an automorphism of the bundle, which will lie in $\co^*\times (H^0(K_C),+)\subset \mathrm{SL}(2,\co)$. The action of $\co^*$ on the Higgs field is precisely the one seen in the type $\mathrm{(i)}$ case. Therefore isomorphism classes of stable Higgs bundles are parametrized by the disjoint union of 16 copies of 
$$ H^0(K_C)\times \frac{(H^0(\mathcal{O})-\{0\})\times H^0(K^2_C)}{\co^*}\cong  H^0(K_C)\times H^0(K^2_C)\cong \co^5.$$
Consider the action of the automorphism group $(\co^2,+)$: any element $\zeta \in H^0(K_C)=\co^2$ acts on $\Phi$ as
$$\left(\begin{array}{cc}1 &\zeta \\
0 &1\\
\end{array}\right) \left(\begin{array}{cc}
a &b\\
c &-a\\
\end{array}\right) \left(\begin{array}{cc}
1 &-\zeta \\
0 &1\\
\end{array}\right) =\left(\begin{array}{cc}
a- \zeta c&b+2\zeta a -\zeta^2 c\\
c& -a+\zeta c
\end{array}\right).$$
Such an action is linear and free on $a\in H^0(K_C)$ and whenever we fix $a-\zeta c$ then the value of $b+2\zeta a-\zeta^2 c $ is fixed as well. Therefore the quotient of $H^0(K_C)\times H^0(K_C^2)$ by $(\co^2,+)$ is $\co^3$. 

\gr{Non trivial case}

Non-trivial extensions of $K^{\frac{1}{2}}$ by $K^{-\frac{1}{2}}$ are parametrized by $\mathbb{P}(H^1(K^{-1}))=\mathbb{P}^2$ and fit in the exact sequence
$$ 0 \rightarrow K^{\frac{1}{2}}\rightarrow V\rightarrow K^{-\frac{1}{2}}\rightarrow 0.$$
Tensor by $K_C$ and apply the functor $Hom(V,-)$ restricted to traceless endomorphisms. Taking global sections one obtains
\begin{equation*}
\begin{tikzpicture}[baseline, descr/.style={fill=white,inner sep=1.5pt}]
\matrix (m) [
matrix of math nodes,
row sep=1em,
column sep=2.5em,
text height=1.5ex, text depth=0.25ex
]
{ 0 & H^0(Hom(V,K^{\frac{1}{2}}K_C))& H^0(End_0(V)\otimes K_C) &H^0(Hom(K_C^{\frac{1}{2}},K_C^{\frac{1}{2}})) &\\
	& H^1(Hom(V,K^{\frac{1}{2}}K_C)) & H^1(End_0(V)\otimes K_C) & H^1(Hom(K_C^{\frac{1}{2}},K_C^{\frac{1}{2}})) &0. \\
};

\path[overlay,->, font=\scriptsize,>=latex]
(m-1-1) edge (m-1-2)
(m-1-2) edge (m-1-3)
(m-1-3) edge node[above]{$\rho$}(m-1-4)
(m-1-4) edge[out=355,in=175] node[descr,yshift=0.3ex] {$ext$} (m-2-2)
(m-2-2) edge (m-2-3)
(m-2-3) edge (m-2-4)
(m-2-4) edge (m-2-5);
\end{tikzpicture}
\end{equation*}

Again, a Higgs bundle with $V$ as underlying vector bundle is stable if and only if its Higgs field lies in the complement of the kernel of $\rho$. We shall prove that there are no stable bundles of this type, i.e. that $\rho=0$. \\
Since $H^0(Hom(K_C^{\frac{1}{2}},K_C^{\frac{1}{2}}))\cong \co$ this is equivalent to ask that the map $ext$ is non zero, which is the case as we are considering nontrivial extensions. 
As a result there are no non-trivial extensions of $K_C^{\frac{1}{2}}$ by its dual that give rise to a stable Higgs bundle.\\

\end{proof}


Having $E-$polynomials of all strata in $\mdol^{\mathrm{SL},s}$ one can sum them to obtain the $E-$polynomial of $\mdol^{\mathrm{SL},s}$. 
The following table summarizes the Betti numbers of strata in the table locus of $\mdol^{\mathrm{SL}}$.
\vspace{0.5 cm}
\par
\noindent
\begin{center}
\begin{tabular}{|c|c|c|c|c|c|c|c|c|c|c|c|c|c|}
\hline
&$H^0_c $&$H^1_c$ & $H^2_c$&$H^3_c$ & $H^4_c$& $H^5_c$& $H^6_c$& $H^7_c$&
\ $H^8_c$&$H^9_c$ &$H^{10}_c$ & $H^{11}_c$&$H^{12}_c$\\ \hline
\gr{$T^*\mathcal{S}$}&0&0&0&0&0&0&0&0&0&5&0&0&1\\ \hline
\gr{$\mathcal{S}_1$}&0&0&0&0&0&0&15&21&5&0&1&0&0\\\hline
\gr{$\mathcal{S}_3$}&0&0&0&0&0&16&16&0&0&0&0&0&0\\\hline
\gr{$\mathcal{S}_4$}&0&0&0&0&0&16&0&0&16&0&0&0&0\\\hline	
\gr{$\mathcal{U}$}&0&0&0&0&0&0&16&0&0&0&0&0&0\\ \hline
\end{tabular}
\end{center}
\vspace{0.5 cm}

\begin{proof}[Proof of proposition \ref{emdols}]
Observe that 
$$ \mdol^{\mathrm{SL},s}= T^*\mathcal{S}\sqcup \mathcal{S}_1\sqcup \mathcal{S}_3\sqcup \mathcal{S}_4\sqcup \mathcal{U}.$$ 
Summing $E-$polynomials of all strata we get 
$$E(\mdol^{\mathrm{SL},s})=u^6v^6+u^5v^5+16u^4v^4+13u^3v^3-u^2v^4-u^4v^2-17u^2v^2.$$
\end{proof}

\subsection{Cohomology of $\tsigma^{\mathrm{SL}}\setminus \tomega^{\mathrm{SL}}$ and $\tomega^{\mathrm{SL}}$}

\subsubsection{Cohomology of $\tomega^{\mathrm{SL}}$}
\begin{lem}
$$ E(\tomega)(u,v)=16u^3v^3+16u^2v^2+16uv+16$$
\end{lem}
\begin{proof}
Recall that $\tomega$ consists of 16 copies of a nonsingular hypersurface  $\mathcal{G}$ in $\mathbb{P}^4$. Therefore its cohomology is given by
\begin{align*}
H^0(\tomega)&=H^2(\tomega)=H^4(\tomega)=H^6(\tomega)=\co^{16}, \\
H^1(\tomega)&=H^3(\tomega)=H^5(\tomega)=0.
\end{align*}
The $E$-polynomial of $\tomega$ is 
$$ E(\tomega)(u,v)=16u^3v^3+16u^2v^2+16uv+16$$
\end{proof}
\subsubsection{Cohomology of $\tsigma^{\mathrm{SL}}\setminus \tomega^{\mathrm{SL}}$}
\begin{lem}
$$E(\tsigma^{\mathrm{SL}}\setminus \tomega^{\mathrm{SL}})(u,v)=u^5v^5+5u^4v^4+u^5v^3+u^3v^5+5u^3v^3+u^2v^4+u^4v^2+u^2v^2-16uv-16$$
\end{lem}
\begin{proof}
Observe that $\tsigma^{\mathrm{SL}}\setminus \tomega^{\mathrm{SL}}$ is $\mathbb{P}^1$-bundle over $\Sigma^{\mathrm{SL}}\setminus \Omega^{\mathrm{SL}}$.
Now $$\Sigma^{\mathrm{SL}}\setminus \Omega^{\mathrm{SL}} \cong \left( Jac(C)\times H^0(K_C)/\mathbb{Z}_2\right)\setminus \{16 \text{ points}\}.$$
First we notice that $\Sigma^{\mathrm{SL}}=(Jac(C)\times H^0(K_C)/\mathbb{Z}_2$ has the same cohomology as $\mathcal{K}$, so
$$\begin{array}{ll}
H_c^4(\Sigma^{\mathrm{SL}})\cong \co &\text{ of type }(2,2)\\
H_c^2(\Sigma^{\mathrm{SL}})\cong \co^6 &\text{ of types }4(3,3)+(2,4)+(4,2)\\
H_c^8(\Sigma^{\mathrm{SL}})\cong \co &\text{ of type }(4,4)\\
H_c^i(\Sigma^{\mathrm{SL}})=0& \text{ otherwise}
\end{array}$$
By proposition \ref{ap} $\Sigma^{\mathrm{SL}}\setminus \Omega^{\mathrm{SL}}$ has the same cohomology groups as $\Sigma^{\mathrm{SL}}$ except for $H^1_c(\Sigma^{\mathrm{SL}}\setminus \Omega^{SL})\cong \co^{16}$ of weight 0.
One has 
\begin{align*}
E(\tsigma^{\mathrm{SL}}\setminus \tomega^{\mathrm{SL}})(u,v)&=E(\mathbb{P}^1)E(\Sigma^{\mathrm{SL}}\setminus \Omega^{\mathrm{SL}})(u,v)=(uv+1)(u^4v^4+u^2v^4+u^4v^2+4u^3v^3+u^2v^2-16)\\
&=u^5v^5+5u^4v^4+u^5v^3+u^3v^5+5u^3v^3+u^2v^4+u^4v^2+u^2v^2-16uv-16
\end{align*}
\end{proof}
\subsection{Cohomology of $\tmdol^{\mathrm{SL}}$ and intersection cohomology $\mdol^{\mathrm{SL}}$}
Summing the contributions in equation \eqref{etmdol} one has the following result.
\begin{thm}
Let $\tmdol^{\mathrm{SL}}$ the semismall desingularization of $\mdol^{\mathrm{SL}}$. The $E-$polynomial of $\tmdol^{\mathrm{SL}}$ is 
$$E(\tmdol^{\mathrm{SL}})=u^6v^6+2u^5v^5+21u^4v^4+u^5v^3+u^3v^5+34u^3v^3.$$
Moreover, by lemma \ref{etob} we deduce
$$P_t(\tmdol^{\mathrm{SL}})=1+2t^2+23t^4+34t^6.$$
\end{thm}

We are now ready to prove the main theorem of the section.
\begin{proof}[Proof of theorem \ref{ihsl}  ]
By \eqref{edt}, subtracting from $E(\tmdol^{\mathrm{SL}})$ the contributions $E^{top}$ coming from top cohomology of the fibres gives the intersection $E-$polynomial of $\mdol$: 
$$IE(\mdol^{\mathrm{SL}})=u^6v^6+u^5v^5+17u^4v^4+17u^3v^3.$$
The intersection Betti numbers are obtained again by applying lemma \ref{etob}.
\end{proof}

\section{Intersection cohomology of $\mdol^{\mathrm{GL}}$}

The methods in section 6 can be applied to compute the intersection Betti numbers of the moduli space $\mdol^{\mathrm{GL}}$ of Higgs bundles of rank 2, degree 0 over a curve of genus 2.

Again, the proof of proposition \eqref{semismallness} shows that all the strata of the map $\tilde{\pi^{\mathrm{GL}}}:\tmdol^{\mathrm{GL}}\rightarrow \mdol^{\mathrm{GL}}$ are relevant and that
$$\mdol^{\mathrm{GL},s}=\mathcal{M}^{\mathrm{GL}}_{Dol,0}\quad \Sigma^{\mathrm{GL}}=\mathcal{M}^{\mathrm{GL}}_{Dol,1} \quad \Omega^{\mathrm{GL}}=\mathcal{M}^{\mathrm{GL}}_{Dol,3}$$
with the same fibres over the strata as in the $\mathrm{SL}(2,\co)$.
We stratify $\tmdol^{\mathrm{GL}}$ as follows
$$\tmdol^{\mathrm{GL}}=(\tilde{\pi}^{\mathrm{GL}})^{-1}(\mdol^{\mathrm{GL},s})\sqcup (\tsigma^{\mathrm{GL}}\setminus \tomega^{\mathrm{GL}})\sqcup \tomega^{\mathrm{GL}}.$$

Observe that we as the fibres of $\tilde{\pi}$ over both $\Omega^{\mathrm{GL}}$ and $\Sigma^{\mathrm{GL}}\setminus\Omega^{\mathrm{GL}}$ are irreducible, then the monodromy of the local system is trivial. Moreover since $\Omega^{\mathrm{GL}}$ is nonsingular and $\Sigma^{\mathrm{GL}}$ have finite quotient singularities we have  
$$IC_{\mdol^{\mathrm{GL}}}(\mathcal{L}_{\mdol^{\mathrm{GL}}})_{\mid\mdol^{\mathrm{GL},s}}=\mathbb{Q}[10]\quad IC_{\Sigma}(\mathcal{L}_{\Sigma})_{\mid\Sigma^0}\cong \mathbb{Q}[8](-1)\quad
IC_{\Omega}(\mathcal{L}_{\Omega})_{\mid pt}\cong\mathbb{Q}[4](-3)$$
where the shifts $(-1)$ and $(-3)$ correspond to the Hodge structures $\mathbb{Q}(-1)$ of respectively $H^2(\mathbb{P}^1)$ and $H^6(\mathcal{G})$.\\
Taking hypercohomology with compact support in \eqref{dtg}, we obtain the intersection cohomology groups and the splitting in the decomposition theorem becomes
$$ H_c^*(\tmdol^{\mathrm{GL}})=IH_c^*(\mdol) \oplus H_c^{*-2}(\Sigma^{\mathrm{GL}}, IC_{\Sigma}(\mathcal{L}_{\Sigma}) )\oplus H_c^{*-6}(\Omega^{\mathrm{GL}},IC_{\Omega}(\mathcal{L}_{\Omega}))$$
We compute the intersection $E-$polynomial and use lemma \ref{etob} to obtain intersection Betti numbers. 

\begin{thm}[\gr{Intersection cohomology of $\mdol^{\mathrm{GL}}$}]\label{ihgl} 
	The intersection Poincaré polynomial of $\mdol^{\mathrm{GL}}$ is 
	$$ IP_t(\mdol^{\mathrm{GL}})=1+4t+7t^2+8t^3+9t^4+12t^5+15t^6+16t^7+14t^8+8t^9+2t^{10}.$$
	Moreover the Hodge diamond is 
	$$\begin{tabular}{l|ccc}
	0&&(0,0)&\\
	1&2(1,0)&&2(0,1)\\
	2&(2,0)&5(1,1)&(0,2)\\
    3&4(2,1)&&4(1,2)\\
	4&(1,3)&7(2,2)&(3,1)\\
	5&6(3,2)&&6(2,3)\\
	6&2(4,2)&11(3,3)&2(2,4)\\
	7&8(4,3)&&8(3,4)\\
	8&2(5,3)&10(2,2)&2(3,5)\\
	9&4(5,4)&&4(4,5)\\
	10&&2(5,5)&.\\
	\end{tabular}
	$$
\end{thm}

Observe that 
\begin{equation}\label{dece}
E(\tmdol^{\mathrm{GL}})=E(\mdol^{\mathrm{GL},s})+E(\tsigma^{\mathrm{GL}}\setminus \tomega^{\mathrm{GL}}) + E(\tomega^{\mathrm{GL}})
\end{equation}
and that by \eqref{edt} and \eqref{etmdol}
$$IE(\mdol^{\mathrm{GL}})=E(\tmdol^{\mathrm{GL}})-E(\Sigma^{\mathrm{GL}}\times H^2(\mathbb{P}^1))-E(\Omega^{\mathrm{GL}}\times H^6(\mathcal{G}));$$

The rest of the section is devoted to computing each term of the equations above.

\subsection{Cohomology of $\mdol^{\mathrm{GL},s}$}

In this section we shall compute the $E-$polynomial of the locus $\mdol^{\mathrm{GL},s}$ of stable pairs $(V,\Phi)$.
\begin{prop}\label{emdolsgl}
Let $\mdol^{\mathrm{GL},s}$ be the locus of stable Higgs pairs on $C$. The $ E-$polynomial of $\mdol^{\mathrm{GL},s}$ is
\begin{align*}
E(\mdol^{GL,s})&=-2 u^4 v^4+4 u^5 v^4-3 u^6 v^4+2 u^7 v^4-u^8 v^4+4 u^4 v^5-10 u^5 v^5+10 u^6 v^5-6 u^7 v^5+2 u^8 v^5+\\
&-3 u^4 v^6+10 u^5 v^6-11 u^6 v^6+4 u^7 v^6+2 u^4 v^7-6 u^5 v^7+4 u^6 v^7+3 u^7 v^7-4 u^8 v^7+u^9 v^7+\\
&-u^4 v^8+2 u^5 v^8-4 u^7 v^8+6 u^8 v^8-4 u^9 v^8+u^10 v^8+u^7 v^9-4 u^8 v^9+5 u^9 v^9-2 u^{10} v^9+u^8 v^{10}+\\
&-2 u^9 v^{10}+u^{10} v^{10}.
\end{align*}
\end{prop}

As for $\mathrm{SL}(2,\co)$, we divide stable Higgs pairs in following three strata:
\begin{itemize}
\item pairs $(V,\Phi)$ with $V$ stable vector bundle;
\item pairs $(V,\Phi)$ with $V$ strictly semistable vector bundle;
\item pairs $(V,\Phi)$ with $V$ unstable vector bundle.
\end{itemize}

\subsection{The stable case}

We shall parametrize all stable Higgs pairs $(V,\Phi)$ where V is a stable vector bundle. 
\begin{prop}
 Let $\mathcal{N}^s$ be the locus of stable vector bundles on $C$.	The locus of stable Higgs pairs $(V,\Phi)$ with $V\in\mathcal{N}^s$ is isomorphic to the cotangent bundle $T^*\mathcal{N}^s$ and its $E-$polynomial is 
\begin{align*}
E(T^*\mathcal{N}^s)&=-u^7 v^5+2 u^8 v^5-u^9 v^5-3 u^6 v^6+8 u^7 v^6-7 u^8 v^6+2 u^9 v^6-u^5 v^7+8 u^6 v^7-14 u^7 v^7+8 u^8 v^7+\\
&-u^9 v^7+2 u^5 v^8-7 u^6 v^8+8 u^7 v^8-2 u^8 v^8-2 u^9 v^8+u^10 v^8-u^5 v^9+2 u^6 v^9-u^7 v^9-2 u^8 v^9+\\
&+4 u^9 v^9-2 u^{10} v^9+u^8 v^{10}-2 u^9 v^{10}+u^{10} v^{10}.
\end{align*}
\end{prop}
\begin{proof}
Suppose $V\in \mathcal{N}^s$. Since the stability of $V$ ensures the stability of the Higgs pair $(V,\Phi)$ for all $\Phi$, the locus of stable Higgs pairs with stable underlying vector bundle is isomorphic to the cotangent bundle $T^*\mathcal{N}^s$.\\
The moduli space $\mathcal{N}$ of semistable vector bundles on $C$ is isomorphic to a $\mathbb{P}^3$- bundle over $Jac(C)$. Namely, the fibre over a point $\zeta\in Jac(C)$ is a copy of the moduli space of semistable vector bundles of degree 0 and rank 2 with determinant $\zeta$. From now we denote $Jac(C)$ by $\mathcal{J}$.\\
A semistable vector bundle $V$ is non stable if and only if it is of the form
$$V=L\oplus M,\quad L,M\in \mathcal{J},$$
and therefore strictly semistable vector bundles are parametrized by the symmetric product $\mathcal{J}^{(2)}$.
As a consequence $\mathcal{N}^s$ is the complement of $\mathcal{J}^{(2)}$ in $\mathcal{N}$.

We compute $E-$polynomials of both $\mathcal{N}$ and $\mathcal{J}^{(2)}$, then $E(\mathcal{N}^s)=E(\mathcal{N})-E(\mathcal{J}^{(2)})$.
The cohomology of $\mathcal{N}$ has been computed by Kirwan in \cite{K2} and the $E-$ polynomial is
\begin{align*}
E(\mathcal{N})&=1-2 u+u^2-2 v+5 u v-4 u^2 v+u^3 v+v^2-4 u v^2+6 u^2 v^2-4 u^3 v^2+u^4 v^2+u v^3-4 u^2 v^3+\\
&+6 u^3 v^3-4 u^4 v^3+u^5 v^3+u^2 v^4-4 u^3 v^4+5 u^4 v^4-2 u^5 v^4+u^3 v^5-2 u^4 v^5+u^5 v^5.
\end{align*}
The cohomology of $\mathcal{J}^{(2)}$ is the $\mathbb{Z}_2$ invariant part of the cohomology of $\mathcal{J}\times\mathcal{J}$. Alternatively one can use Macdonald formula \cite{Md} for symmetric product of surfaces. One has
\begin{align*}
	E(\mathcal{J}^{(2)}) &=1-2 u+2 u^2-2 u^3+u^4-2 v+8 u v-12 u^2 v+8 u^3 v-2 u^4 v+2 v^2-12 u v^2+20 u^2 v^2+\\
	&-12 u^3 v^2+2 u^4 v^2-2 v^3+8 u v^3-12 u^2 v^3+8 u^3 v^3-2 u^4 v^3+v^4-2 u v^4+2 u^2 v^4-2 u^3 v^4+u^4 v^4.
\end{align*}
As a consequence
\begin{align*}
E(\mathcal{N}^s) &= -u^2+2 u^3-u^4-3 u v+8 u^2 v-7 u^3 v+2 u^4 v-v^2+8 u v^2-14 u^2 v^2+8 u^3 v^2-u^4 v^2+2 v^3+\\
&-7 u v^3+8 u^2 v^3-2 u^3 v^3-2 u^4 v^3+u^5 v^3-v^4+2 u v^4-u^2 v^4-2 u^3 v^4+4 u^4 v^4-2 u^5 v^4+u^3 v^5+\\
&-2 u^4 v^5+u^5 v^5.
\end{align*}
Since $T^*\mathcal{N}^s$ inherits the cohomology of $\mathcal{N}^s$ we have:
\begin{align*}
E(T^*\mathcal{N}^s)&=-u^7 v^5+2 u^8 v^5-u^9 v^5-3 u^6 v^6+8 u^7 v^6-7 u^8 v^6+2 u^9 v^6-u^5 v^7+8 u^6 v^7-14 u^7 v^7+8 u^8 v^7+\\
&-u^9 v^7+2 u^5 v^8-7 u^6 v^8+8 u^7 v^8-2 u^8 v^8-2 u^9 v^8+u^10 v^8-u^5 v^9+2 u^6 v^9-u^7 v^9-2 u^8 v^9+\\
&+4 u^9 v^9-2 u^{10} v^9+u^8 v^{10}-2 u^9 v^{10}+u^{10} v^{10}.
\end{align*}
\end{proof}
\subsection{Strictly semistable case}
We consider pairs $(V,\Phi)$ with $V$ strictly semistable and investigate under which conditions they are stable Higgs pairs. 
Again, we have to distinguish different cases:
\begin{enumerate}[(i)]
\item $V=L\oplus M$ where $L,M\in \mathcal{J}$ and $L\not\cong M$;
\item $V$ is a non trivial extension $\xymatrix{
	0\ar[r] &L \ar[r] &V\ar[r] & M\ar[r]&0
} $ with $L\not\cong M$;
\item $V=L\oplus L$ where $L\in \mathcal{J}$ ;
\item $V$ is a non trivial extension $\xymatrix{
	0\ar[r] &L \ar[r] &V\ar[r] & L\ar[r]&0
} $.
\end{enumerate}

\subsubsection{Type $\mathrm{(i)}$}
We shall determine stable Higgs bundles $(V,\Phi)$ with underlying vector bundle of type $\mathrm{(i)}$.
Strictly semistable vector bundles are parametrized by $\mathcal{J}^{(2)}$.
Let $\mathcal{J}_0$ be the diagonal in $\mathcal{J}^{(2)}$ fixed by the involution and let $\mathcal{J}^0:=\mathcal{J}^{(2)}-\mathcal{J}_0$ be its complement. 
The locus of stable Higgs bundles with underlying vector bundle of type (i) is a fibre bundle on $\mathcal{J}^0$. 
\begin{prop}
Let $\mathcal{N}_1$ be the locus of stable Higgs bundles with underlying vector bundle of type $\mathrm{(i)}$. Then $\mathcal{N}_1$ is a $(\co^4\times\co^*)$-bundle over $\mathcal{J}^0$ and its $E$-polynomial is
\begin{align*}
E(\mathcal{N}_1)&= -u^6 v^4+2 u^7 v^4-u^8 v^4-4 u^5 v^5+10 u^6 v^5-7 u^7 v^5+u^9 v^5-u^4 v^6+10 u^5 v^6-15 u^6 v^6+2 u^7 v^6+\\
&+6 u^8 v^6-2 u^9 v^6+2 u^4 v^7-7 u^5 v^7+2 u^6 v^7+11 u^7 v^7-10 u^8 v^7+2 u^9 v^7-u^4 v^8+6 u^6 v^8-10 u^7 v^8+\\
&+7 u^8 v^8-2 u^9 v^8+u^5 v^9-2 u^6 v^9+2 u^7 v^9-2 u^8 v^9+u^9 v^9.
\end{align*}
\end{prop}

\begin{proof}
To compute the fibre we consider $V=L\oplus M$ with $L,M\in\mathcal{J}$. 
$$ H^0(End(V)\otimes K_C)= H^0(K_C)\oplus H^0(L^{-1}MK_C)\oplus H^0(M^{-1}LK_C)\oplus H^0(K_C)$$
thus a Higgs field $\Phi\in H^0(End(V)\otimes K_C)$ takes the form 
$$\Phi=\left( \begin{array}{cc}
	a&b\\
	c&d\\
\end{array}\right)$$
with $a,d\in H^0(K_C)$, $b\in H^0(M^{-1}LK_C) $, $c\in H^0(L^{-1}MK_C)$. A pair $(V,\Phi)$ is stable if and only if both $L$ and $M$ are not preserved by $\Phi$, that is $b,c\neq 0$.
Since the automorphisms group of $V$ is $\co^*\times\co^*$, two Higgs pairs $(V,\Phi_1)$ and $(V,\Phi_2)$ for $\Phi_i=(a_i,b_i,c_i,d_i)$ are isomorphic if and only if
$$\Phi_1=\left( \begin{array}{cc}
	t&0\\
	0&s\\
\end{array}\right)\Phi_2\left( \begin{array}{cc}
	t^{-1}&0\\
	0&s^{-1}\\
\end{array}\right)$$
that is $a_1=a_2$, $b_1=t^{-1}s b_2$, $c_1=s^{-1}t c_2$, $d_1=d_2$. Therefore, the stable Higgs pairs $(V,\Phi)$ with fixed underline vector bundle $V$ are parametrized by
$$H^0(K_C)^2\times \dfrac{(H^0(L^{-1}MK_C)-\{0\})\times (H^0(M^{-1}LK_C)-\{0\})}{\co^*\times \co^*} \cong\co^4\times \co^*$$
(note that $\co^*\times \co^*$ acts with stabilizer $\co^*$). Letting $V$ vary, one obtains a $(\co^4\times \co^*)$-bundle $\mathcal{S}_1$ over $\mathcal{J}^0$.
The $E-$polynomial is the product $E(\mathcal{J}^0)E(\co^4)E(\co^*)$. Since
\begin{align*}
E(\mathcal{J}^0)&=E(\mathcal{J}^{(2)})-E(\mathcal{J}_0)= u^2-2 u^3+u^4+4 u v-10 u^2 v+8 u^3 v-2 u^4 v+v^2-10 u v^2+19 u^2 v^2+\\
&-12 u^3 v^2+2 u^4 v^2-2 v^3+8 u v^3-12 u^2 v^3+8 u^3 v^3-2 u^4 v^3+v^4-2 u v^4+2 u^2 v^4-2 u^3 v^4+u^4 v^4,
\end{align*}
the $E-$polynomial of $\mathcal{N}_1$ is 
\begin{align*}
E(\mathcal{N}_1)&= -u^6 v^4+2 u^7 v^4-u^8 v^4-4 u^5 v^5+10 u^6 v^5-7 u^7 v^5+u^9 v^5-u^4 v^6+10 u^5 v^6-15 u^6 v^6+2 u^7 v^6+\\
&+6 u^8 v^6-2 u^9 v^6+2 u^4 v^7-7 u^5 v^7+2 u^6 v^7+11 u^7 v^7-10 u^8 v^7+2 u^9 v^7-u^4 v^8+6 u^6 v^8-10 u^7 v^8+\\
&+7 u^8 v^8-2 u^9 v^8+u^5 v^9-2 u^6 v^9+2 u^7 v^9-2 u^8 v^9+u^9 v^9.
\end{align*}
\end{proof}
\subsubsection{Type (ii)}
Assume now that $V$ is a nontrivial extension of $L$ by $M$ with $L\not\cong M$: we compute the cohomology of the locus of stable pairs $(V,\Phi)$.

\begin{prop}\label{type2}
Let $V$ be a semistable vector bundle of type $\mathrm{(ii)}$. Then there is no Higgs field $\Phi$ such that the pair $(V,\Phi)$ is stable.
\end{prop}
\begin{proof}
Consider the universal bundle $(\luniv,\mathcal{M})\rightarrow \mathcal{J}^0\times C$ and let $p:\mathcal{J}^0\times C \rightarrow \mathcal{J}^0$ be the projection onto the first factor. Non trivial extensions of $\luniv$ by $\mathcal{M}$ are parametrized by $\mathbb{P}(R^1p_{*}\mathcal{M}^{-1}\luniv)$: as  $R^1p_{*}\mathcal{M}^{-1}\luniv$ is a local system on $\mathcal{J}^0$ of rank one, there exists a unique nontrivial extension $\vuniv$ up to isomorphism. This bundle fits in the short exact sequence
\begin{equation}\label{sesgl}
0\rightarrow \luniv\rightarrow \vuniv \rightarrow \mathcal{M}\rightarrow 0
\end{equation}
and parametrizes all the vector bundles $V$ on $C$ of type (ii).\\
A Higgs pair $(\vuniv,\Phi)$ is stable if and only if the Higgs field $\Phi$ lies in the complement of the kernel of the restriction map $\rho: p_*End(\vuniv)\otimes K_C\rightarrow p_*Hom(\luniv,\mathcal{M} K_C)$, i.e. $\Phi$ does not preserve $\luniv$. \\
Tensoring the sequence (\ref{sesgl}) by $K_C$ and applying the covariant functor $Hom(\vuniv,-)$ restricted to traceless endomorphisms one obtains
$$0\rightarrow Hom(\vuniv,\luniv\otimes K_C) \rightarrow End(\vuniv)\otimes K_C \rightarrow Hom(\vuniv,\mathcal{M}\otimes K_C)\rightarrow 0.$$ 
Pushing forward to $\mathcal{J}^0$ one gets a long exact sequence
\begin{equation}\label{les1gl}
\begin{tikzpicture}[baseline, descr/.style={fill=white,inner sep=1.5pt}]
\matrix (m) [
matrix of math nodes,
row sep=1em,
column sep=2.5em,
text height=1.5ex, text depth=0.25ex
]
{ 0 & p_*Hom (\vuniv,\luniv K_C) & p_* End(\vuniv)\otimes K_C &p_*Hom(\vuniv,\muniv K_C)&\\
	& R^1p_*Hom(\muniv,\luniv K_C) & R^1p_* End(\vuniv)\otimes K_C &R^1p_*Hom(\vuniv,\muniv K_C)&0. \\
};
\path[overlay,->, font=\scriptsize,>=latex]
(m-1-1) edge (m-1-2)
(m-1-2) edge (m-1-3)
(m-1-3) edge (m-1-4)
(m-1-4) edge[out=355,in=175] node[descr,yshift=0.3ex] {$ext$} (m-2-2)
(m-2-2) edge (m-2-3)
(m-2-3) edge (m-2-4)
(m-2-4) edge (m-2-5);
\end{tikzpicture}
\end{equation}

Applying the functor $p_*Hom(-,\muniv K_C)$ to \eqref{sesgl} yields a long exact sequence

\begin{equation}\label{les2gl}
	\begin{tikzpicture}[ baseline, descr/.style={fill=white,inner sep=1.5pt}]
	\matrix (m) [
	matrix of math nodes,
	row sep=1em,
	column sep=2.5em,
	text height=1.5ex, text depth=0.25ex
	]
	{ 0 & p_* Hom(\muniv,\muniv K_C) & p_*Hom(\vuniv,\muniv K_C) &p_*Hom(\luniv,\muniv K_C) &\\
		& R^1p_*Hom(\muniv,\muniv K_C) & R^1p_* Hom(\vuniv,\muniv K_C) &R^1p_*Hom(\luniv,\muniv K_C)&0. \\
	};
	
	\path[overlay,->, font=\scriptsize,>=latex]
	(m-1-1) edge (m-1-2)
	(m-1-2) edge (m-1-3)
	(m-1-3) edge (m-1-4)
	(m-1-4) edge[out=355,in=175] node[descr,yshift=0.3ex] {$ext$} (m-2-2)
	(m-2-2) edge (m-2-3)
	(m-2-3) edge (m-2-4)
	(m-2-4) edge (m-2-5);
	\end{tikzpicture}
\end{equation}
The map $\rho$ is the composition $$p_* End(\vuniv)\otimes K_C\rightarrow p_*Hom(\vuniv,\muniv K_C)\rightarrow p_*Hom(\luniv,\muniv K_C).$$
We prove that the second map is 0, i.e. there are no stable Higgs bundles of type (ii).\\
Consider the fibre of (\ref{les2gl})  on a point $(L,M)\in \mathcal{J}^0$: one has
$$ H^1(K_C)\rightarrow H^1(\mathcal{H}om(V,MK_C))\rightarrow H^1(M^{-1}LK_C)\rightarrow 0.$$
Since the extension map is non zero and $H^1(M^{-1}LK_C)=0$ we have that  $H^0(L^{-1}MK_C)\cong H^1(K_C)\cong \co$, thus $H^1(Hom(V,M K_C)$ is 0 and $H^0(Hom(V,MK_C))\cong H^0(K_C)\cong \co^2$. \\
In particular the map $p_*Hom(\vuniv,\muniv K_C)\rightarrow p_*Hom(\luniv,\muniv K_C)$ is zero.
\end{proof}

\subsubsection{Type $\mathrm{(iii)}$}
We now consider stable Higgs bundle with underlying vector bundle $V=L\oplus L$ with $L\in \mathcal{J}$.

\begin{prop} Let $\gr{N}_3$ be the locus of stable Higgs bundles with underlying vector bundle $L\oplus L$ with $L\in \mathcal{J}$ . Then the locus $\mathcal{N}_3$ of stable Higgs pairs of type (iii) is a Zariski locally trivial $\gr{N}_3$-bundle over $\mathcal{J}$ and its $E-$polynomial is
$$ E(\mathcal{N}_3)(u,v)= -u^4 v^4+2 u^5 v^4-u^6 v^4+2 u^4 v^5-3 u^5 v^5+u^7 v^5-u^4 v^6+3 u^6 v^6-2 u^7 v^6+u^5 v^7-2 u^6 v^7+u^7 v^7.$$

\end{prop}
\begin{proof}
Consider $V=L\oplus L$. In this case $H^0(End(V)\otimes K_C)\cong H^0(K_C)\otimes gl(2)\cong \co^2\otimes gl(2)$ and the Higgs field is of the form $$\Phi=\left( \begin{array}{cc}
	a&b\\
	c&d\\
\end{array}\right),$$
with $a,b,c,d\in H^0(K_C)$.
The bundle is not stable if and only if $\Phi$ is conjugate to an upper triangular matrix of elements of $H^0(K_C)$.  The action of $\mathrm{PGL}(2,\co) $ on $H^0(K_C)\otimes gl(2)$ is trivial on the first factor, so one can proceed as in the case of $\mathrm{SL}(2,\co)$ looking for the couples of matrices $(A,B)\in gl(2)\oplus gl(2)$ that are not simultaneously triangulable. As before, $A$ and $B$ are simultaneously triangulable if $\det([A,B])=0$.
Writing \begin{equation}\label{matgl}
A=\left( \begin{array}{cc}
x_1&x_2\\
x_3&x_4\\
\end{array}\right)\qquad B=\left( \begin{array}{cc}
y_1&y_2\\
y_3&y_4\\
\end{array}\right)\end{equation}
one has
$$[A,B]=\left( \begin{array}{cc}
	x_2y_3-y_2x_3& x_1y_2-x_4y_2+x_2y_4-x_2y_1\\
	x_3y_1-x_2y_4+x4y_3-x_1y_3&-(x_2y_3-y_2x_3)\\
\end{array}\right)=\left( \begin{array}{cc}
	x_2y_3-y_2x_3& sy_2-x_2t\\
	x_3t-sy_3&-(x_2y_3-y_2x_3)\\
\end{array}\right)$$
where $s=(x_1-x_4)$ and $t=y_1-y_4$.\\
Thanks to this substitution, one can interpret the locus of simultaneously triangulable matrices $(A,B)\in gl(2)\oplus gl(2)$ as a fibration over 
$$Q: (x_2y_3-y_2x_3)^2+( sy_2-x_2t)(x_3t-sy_3)=0 \subset \co^6$$
in the coordinates $(s,x_2,x_3,t,y_2,y_3)\in \co^6$.
The fibre over a point in $Q$ is the 2-dimensional vector space
$$\begin{cases}
x_1-x_4=s\\
y_1-y_4=t
\end{cases}$$ in the coordinates $(x_1,x_4,y_1,y_4)$.\\
Let $H$ be the total space of this bundle over $Q$. 
One can summarize the above considerations in the following lemma.
\begin{lem}
A Higgs bundle $(V,\Phi)$ of type (iii) is stable if and only if $\Phi$ lies in 
$$ \gr{N}_3:=(\co^8-H)\sslash \mathrm{PGL}(2,\co)$$ 
where the action of $\mathrm{PGL}(2,\co)$ is the simultaneous conjugation on the matrices $A$ and $B$ as in (\ref{matgl}).
\end{lem}
Since $H$ is a $\co^2$-bundle over $Q$,
$$E(H)=E(\co^2)E(Q)=u^5v^5(u^2v^2+uv-1).$$
$$ E(\gr{N}_3)(u,v)=\dfrac{E(\co^8)-E(H)}{E(\mathrm{PGL}(2,\co))}=u^5v^5-u^4v^4.$$
To obtain $E(\mathcal{N}_3)$ one needs to multiply $E(\gr{N}_3)$ by $E(\mathcal{J})$. 
\end{proof}
\subsubsection{Type $\mathrm{(iv)}$}

We now consider stable Higgs bundles of type (iv) and we prove the following result.
\begin{prop}
Fix $L\in \mathcal{J}$ and let $\gr{N}_4$ be the locus of stable Higgs bundles whose underlying vector bundle is a nontrivial extension of $L$ by itself. Then the locus $\mathcal{N}_4$ of stable Higgs bundles of type $\mathrm{(iv)}$ is a Zariski locally trivial $\gr{N}_4$-bundle over $\mathcal{J}$ and its $E-$polynomial is 
\begin{align*}
E(\mathcal{N}_4)&=-u^4 v^4+2 u^5 v^4-u^6 v^4+2 u^4 v^5-4 u^5 v^5+2 u^6 v^5-u^4 v^6+2 u^5 v^6-2 u^7 v^6+u^8 v^6-2 u^6 v^7+\\
&+4 u^7 v^7-2 u^8 v^7+u^6 v^8-2 u^7 v^8+u^8 v^8.
\end{align*}
\end{prop}	 
\begin{proof}	 
Let $L\in \mathcal{J}$ and let $V$ be a nontrivial extension of $L$ by itself: the isomorphism classes of such bundles are parametrized by 
\begin{equation}\label{ses4gl}
\mathbb{P}(\mathrm{Ext}^1(L,L))\cong \mathbb{P}^1.
\end{equation}
Thus there exists a universal extension bundle on $\mathbb{P}^1\times C$
$$0\rightarrow \mathcal{L}\rightarrow \vuniv\rightarrow\mathcal{L}\rightarrow 0.$$
Let $p:\mathbb{P}^1\times C\rightarrow \mathbb{P}^1$ be the projection map: as in the type (ii) case, one can tensor the above short exact sequence by $K_C$, apply the covariant functor $Hom(\vuniv,-)$ and pushforward to $\mathbb{P}^1$, getting the long exact sequence
\begin{equation}\label{*gl}
\begin{tikzpicture}[baseline, descr/.style={fill=white,inner sep=1.5pt}]
\matrix (m) [
matrix of math nodes,
row sep=1em,
column sep=2.5em,
text height=1.5ex, text depth=0.25ex
]
{ 0 &  p_*Hom(\vuniv,\luniv K_C) & p_*(End(\vuniv)\otimes K_C) &p_*Hom(\vuniv,\luniv K_C)&\\
	& R^1p_*Hom(\vuniv,\luniv K_C)& R^1p_*(End(\vuniv)\otimes K_C) &  R^1p_*Hom(\vuniv,\luniv K_C) &0. \\
};

\path[overlay,->, font=\scriptsize,>=latex]
(m-1-1) edge (m-1-2)
(m-1-2) edge (m-1-3)
(m-1-3) edge (m-1-4)
(m-1-4) edge[out=355,in=175] node[descr,yshift=0.3ex] {$ext$} (m-2-2)
(m-2-2) edge (m-2-3)
(m-2-3) edge (m-2-4)
(m-2-4) edge (m-2-5);
\end{tikzpicture}
\end{equation}

Starting again from (\ref{ses4gl}), tensoring with $K_C$, applying the contravariant functor $Hom(-,\mathcal{L})$ and pushing forward to $\mathbb{P}^1$ one obtains another long exact sequence
\begin{equation}\label{**gl}
\begin{tikzpicture}[baseline]
\node (A) at (0,0) {0};
\node (B) at (3.5,0) {$p_*K_C$};
\node (C) at (7,0) {$p_*Hom(\vuniv,\luniv K_C)$};
\node (D) at (10.5,0) {$p_*K_C$};
\node (E) at (14,0) {$R^1p_*K_C.$};
\draw[->,thick] (A) -- (B) node [midway,above] {};
\draw[->,thick] (B)-- (C) node [midway,above] {};
\draw[->,thick] (C) -- (D) node [midway,above] {};
\draw[->,thick] (D) -- (E) node [midway,above] {$ext$};
\end{tikzpicture}
\end{equation}

As before, stable Higgs bundles are those whose Higgs field is in the complement of the kernel of the map
\begin{center}
	\begin{tikzpicture}
	\node (A) at (0,0) {$p_*End(\vuniv)\otimes K_C$};
	\node (B) at (3.5,0) {$p_*Hom(\vuniv,\luniv K_C)$};
	\node (C) at (7,0) {$p_*K_C.$};
	\draw[->,thick] (A) -- (B) node [midway,above] {};
	\draw[->,thick] (B) -- (C) node [midway,above] {};
	\end{tikzpicture}
\end{center}
Consider the sequence \eqref{**gl}: $p_*K_C$ has rank 2. Since $R^1p_*K_C$ has rank 1, being non zero, the map $ext$ is surjective. Hence, the cokernel of $p_*Hom(\vuniv,K_C)\rightarrow p_*K_C$ has rank 1 and consequently $p_*Hom(\vuniv,K_C)$ has rank 3. 
Looking at \eqref{*gl} one concludes that $p_*End(V)\otimes K_C$ is a vector bundle of rank 5.
The group of automorphisms of a nontrivial extension of $\mathcal{L}$ by itself is the additive group $(\co,+)\subset \mathrm{GL}(2,\co)$, and an element $t\in \co$ acts on the Higgs field $\Phi$ by conjugation: 
$$t.\Phi=\left(\begin{array}{cc}
	1&t\\
	0&1
\end{array}\right)\left(\begin{array}{cc}
	a&b\\
	c&-a
\end{array}\right)\left(\begin{array}{cc}
	1&-t\\
	0&1
\end{array}\right)=\left(\begin{array}{cc}
	1a+tc&b-2ta-t^2c\\
	c&-a-tc
\end{array}\right).
$$

\begin{lem}\label{n4}
$\gr{N}_4$ is a $\co^2$-bundle over a $\co^2$-bundle over a $\co^*$-bundle over $\mathbb{P}^1$.
\end{lem}
\begin{proof}
	Let $A$ be the kernel of the extension map in \eqref{*gl} minus the zero section: $A$ is a $(\co^2-\{0\})$-bundle over $\mathbb{P}^1$. We can think of $p_*(End_0(\vuniv)\otimes K_C)-p_*Hom(\vuniv,K_C)$ as a vector bundle of rank 3 over $A$. Similarly, the kernel of the extension map in \eqref{**gl} gives rise to a vector bundle $\mathcal{A}$ over $A$ of rank 1 and the map 
	$$p_*Hom(\vuniv)\rightarrow p_*(End(\vuniv)\otimes K_C)$$ lifts to a $(\co,+)$-equivariant map 
	$$[p_*(End(\vuniv)\otimes K_C)-p_*Hom(\vuniv,K_C)]\rightarrow \mathcal{A}$$ 
	of vector bundles over $A$ whose kernel is of rank 2. Observe that $\mathcal{A}$ is invariant under the automorphism action, while $A$ and $p_*(End(\vuniv)\otimes K_C)-p_*Hom(\vuniv,K_C)]$ are not. 
	In this way we have that $p_*(End(\vuniv)\otimes K_C)-p_*Hom(\vuniv,K_C)]/\co$ is a $\co^2$-bundle on $\mathcal{A}$ which is a $\co^2$-bundle over $A/\co^*$, which is a $\co^*$-bundle over $\mathbb{P}^1$.
\end{proof}
As all bundles are Zariski locally trivial one has:
\begin{align*}
E(\mathcal{N}_4)&=-u^4 v^4+2 u^5 v^4-u^6 v^4+2 u^4 v^5-4 u^5 v^5+2 u^6 v^5-u^4 v^6+2 u^5 v^6-2 u^7 v^6+u^8 v^6-2 u^6 v^7+\\
&+4 u^7 v^7-2 u^8 v^7+u^6 v^8-2 u^7 v^8+u^8 v^8.
\end{align*}

\end{proof}
\subsection{Unstable case}
Consider the locus $\mathcal{NU}$ of stable Higgs bundles $(V,\Phi)$ where $V$ is an unstable vector bundle. Then there exists a line bundle $L$ of degree $d>0$ that fits an exact sequence
\[
\xymatrix{ 
	0\ar[r] &L\ar[r] &V\ar[r] & M\ar[r] &0}
\]

with $M\in Pic^d(C)$.
If $d>1$ the bundle $L^{-1}MK_C$ has no non-zero global section because it has negative degree, hence $L$ is $\Phi$-invariant for any Higgs field $\Phi\in H^0( End(V)\otimes K_C)$.
The only case to check is $deg(L)=1$. The line bundle $L^{-1}MK_C$ has degree 0: it has global sections if and only if it is trivial, that is $M=LK_C^{-1}$ with $L\in Pic^1(C)$. 
As a consequence, if there exists an unstable vector bundle $V$ which is stable as a Higgs bundle, then it must be an extension of the above form. 
\begin{prop}
The locus $\mathcal{NU}$ of stable Higgs bundles $(V,\Phi)$ with $V$ unstable is isomorphic to a Zariski locally trivial $\co^5$-bundle on $Pic^1(C)$. As a consequence its $E-$polynomial  is $$E(\mathcal{NU})=u^5v^5E(Pic^1(C))=u^5 v^5-2 u^6 v^5+u^7 v^5-2 u^5 v^6+4 u^6 v^6-2 u^7 v^6+u^5 v^7-2 u^6 v^7+u^7 v^7.$$
\end{prop}

\begin{proof}

\gr{Trivial case}\\
Consider $
V=L\oplus M$ with $L\in Pic^1(C)$ and $M=LK_C^{-1}$. Then 
$$ H^0(End(V)\otimes K_C)\cong H^0(K_C)\oplus H^0(K_C^2)\oplus H^0(\mathcal{O})\oplus H^0(K_C) \cong \co^2\oplus \co^3\oplus \co \oplus\co^2$$
Thus the generic Higgs field will be of the form
$$\Phi=\left(\begin{array}{cc}
	a&b\\
	c&d\\
\end{array}\right)$$
with $a,d \in H^0(K_C)$, $b \in H^0(L^2 K_C)$, $c\in H^0(L^{-2}K_C)$.
Two Higgs fields define isomorphic Higgs bundles if and only if they are conjugate by an automorphism of the bundle, which lies in $\co^*\times\co^*\times (H^0(K_C),+)\subset \mathrm{GL}(2,\co)$. The action of $\co^*\times \co^*$ on the Higgs field is the one seen in the type (i) case and it has stabilizer $\co^*$. Therefore isomorphism classes of stable Higgs bundles are parametrized by 
$$ H^0(K_C)^2\times \frac{H^0(K_C^2)\times (H^0(\mathcal{O})\setminus\{0\})}{\co^*}\cong  H^0(K_C)^2\times H^0(K_C^2)\cong \co^7.$$
Thanks to the action of $\co^*\times \co^*$ we can suppose $c=1$. Then we have to consider the action of $\zeta \in (H^0(K_C),+)$:
$$\left(\begin{array}{cc}1 &\zeta \\
	0 &1\\
\end{array}\right) \left(\begin{array}{cc}
	a &b\\
	1 &d\\
\end{array}\right) \left(\begin{array}{cc}
	1 &-\zeta \\
	0 &1\\
\end{array}\right) =\left(\begin{array}{cc}
	a+ \zeta &b+\zeta (d-a) -\zeta^2\\
	1& d-\zeta 
\end{array}\right).$$
Such an action is linear and free on $a\in H^0(K_C)$ and we can fix $a+\zeta=0$. Therefore the quotient of $H^0(K_C)^2\times H^0(K_C^2)$ by $(H^0(K_C),+)$ is $H^0(K_C)^2\times H^0(K_C^2)\cong\co^5$. \\

\gr{Non trivial case}

Non-trivial extensions of $L$ by $M$ are parametrized by $\mathbb{P}(H^1(L^{-1}M))=\mathbb{P}^2$ and fit the exact sequence
$$ 0 \rightarrow L\rightarrow V\rightarrow M\rightarrow 0.$$
Tensoring by $K_C$, applying the functor $Hom(V,-)$ and taking global sections gives

\begin{center}
	\begin{tikzpicture}[descr/.style={fill=white,inner sep=1.5pt}]
	\matrix (m) [
	matrix of math nodes,
	row sep=1em,
	column sep=2.5em,
	text height=1.5ex, text depth=0.25ex
	]
	{ 0 &  H^0(Hom(V,LK_C))& H^0(End(V)\otimes K_C)&H^0(Hom(V,MK_C))&\\
		& H^1(Hom(V,LK_C))& H^1(End(V)\otimes K_C)&H^1(Hom(V,MK_C))&0. \\
	};
	
	\path[overlay,->, font=\scriptsize,>=latex]
	(m-1-1) edge (m-1-2)
	(m-1-2) edge (m-1-3)
	(m-1-3) edge (m-1-4)
	(m-1-4) edge[out=355,in=175] node[descr,yshift=0.3ex] {$ext$} (m-2-2)
	(m-2-2) edge (m-2-3)
	(m-2-3) edge (m-2-4)
	(m-2-4) edge (m-2-5);
	\end{tikzpicture}
\end{center}

On the other hand, applying the functor $Hom(-,MK_C)$ and taking global sections one has
\begin{center}
	\begin{tikzpicture}[descr/.style={fill=white,inner sep=1.5pt}]
	\matrix (m) [
	matrix of math nodes,
	row sep=1em,
	column sep=2.5em,
	text height=1.5ex, text depth=0.25ex
	]
	{ 0 &  H^0(K_C)& H^0(Hom(V,MK_C))&H^0(L^{-1}MK_C)&\\
		& H^1(K_C)& H^1(Hom(V,MK_C)) & H^1(L^{-1}MK_C) &0. \\
	};
	
	\path[overlay,->, font=\scriptsize,>=latex]
	(m-1-1) edge (m-1-2)
	(m-1-2) edge (m-1-3)
	(m-1-3) edge (m-1-4)
	(m-1-4) edge[out=355,in=175] node[descr,yshift=0.3ex] {$ext$} (m-2-2)
	(m-2-2) edge (m-2-3)
	(m-2-3) edge (m-2-4)
	(m-2-4) edge (m-2-5);
	\end{tikzpicture}
\end{center}

Again, a Higgs bundle that has $V$ as underlying vector bundle is stable if and only if its Higgs field lies in the complement of the kernel of 
\begin{center}
	\begin{tikzpicture}
	\node (A) at (0,0) {$H^0(End(V)\otimes K_C)$};
	\node (B) at (4,0) {$H^0(Hom(V,MK_C))$};
	\node (C) at (8,0) {$H^0(L^{-1}MK_C)).$};
	\draw[->,thick] (A) -- (B) node [midway,above] {};
	\draw[->,thick] (B) -- (C) node [midway,above] {};
	\end{tikzpicture}
\end{center}

Observe that $H^0(L^{-1}MK_C)\cong H^0(\mathcal{O})\cong \co$ and $H^1(K_C)\cong\co$. As the extension map is non zero then it is an isomorphism. As a result the map $H^0(Hom(V,MK_C))\rightarrow H^0(L^{-1}MK_C))$ is 0, thus no nontrivial extensions give a stable Higgs bundle. 
\end{proof}

\begin{proof}[Proof of proposition \ref{emdolsgl}]
As in the $\mathrm{SL}(2,\co)$ case 
$$ \mdol^{\mathrm{GL},s}=T^*\mathcal{N}^s \sqcup\mathcal{N}_1\sqcup \mathcal{N}_3\sqcup \mathcal{N}_4\sqcup \mathcal{NU}.$$
Summing $E-$polynomials of all strata we get the $E-$polynomial of $\mdol^{\mathrm{GL},s}$:
\begin{align*}
E(\mdol^{GL,s})&=-2 u^4 v^4+4 u^5 v^4-3 u^6 v^4+2 u^7 v^4-u^8 v^4+4 u^4 v^5-10 u^5 v^5+10 u^6 v^5-6 u^7 v^5+2 u^8 v^5+\\
&-3 u^4 v^6+10 u^5 v^6-11 u^6 v^6+4 u^7 v^6+2 u^4 v^7-6 u^5 v^7+4 u^6 v^7+3 u^7 v^7-4 u^8 v^7+u^9 v^7+\\
&-u^4 v^8+2 u^5 v^8-4 u^7 v^8+6 u^8 v^8-4 u^9 v^8+u^10 v^8+u^7 v^9-4 u^8 v^9+5 u^9 v^9-2 u^{10} v^9+u^8 v^{10}+\\
&-2 u^9 v^{10}+u^{10} v^{10}.
\end{align*}
\end{proof}

\subsection{Cohomology of $\tsigma^{\mathrm{GL}}\setminus \tomega^{\mathrm{GL}}$ and $\tomega^{\mathrm{GL}}$}

\subsubsection{Cohomology of $\tomega^{\mathrm{GL}}$}
\begin{lem}
\begin{align*}
E(\tomega^{\mathrm{GL}})&=u^7 v^7 - 2 u^7 v^6 + u^7 v^5 - 2 u^6 v^7 + 5 u^6 v^6 - 4 u^6 v^5 + u^6 v^4 + u^5 v^7 - 4 u^5 v^6 + 6 u^5 v^5 - 4 u^5 v^4 +\\
&+ u^5 v^3 + u^4 v^6 - 4 u^4 v^5 + 6 u^4 v^4 - 4 u^4 v^3 + u^4 v^2 + u^3 v^5 - 4 u^3 v^4 + 5 u^3 v^3 - 2 u^3 v^2 + u^2 v^4+\\
&- 2 u^2 v^3 + u^2 v^2.
\end{align*}
\end{lem}
\begin{proof}
Recall that $\tomega^{\mathrm{GL}}$ is a $\mathcal{G}$-bundle on $T^*\mathcal{J}$, thus the $E$-polynomial of $\tomega^{\mathrm{GL}}$ is 
$$E(\tomega^{\mathrm{GL}})=(1+uv+u^2v^2+u^3v^3)E(\mathcal{J})E(\co^2).$$
\end{proof}
\subsubsection{Cohomology of $\tsigma^{\mathrm{GL}}\setminus \tomega^{\mathrm{GL}}$}
\begin{lem}
\begin{align*}
E(\tsigma^{\mathrm{GL}}\setminus \tomega^{\mathrm{GL}})=&-u^2 v^2+2 u^3 v^2-u^4 v^2+2 u^2 v^3-5 u^3 v^3+4 u^4 v^3-u^5 v^3-u^2 v^4+4 u^3 v^4-4 u^4 v^4+\\
&+2 u^6 v^4-2 u^7 v^4+u^8 v^4-u^3 v^5+8 u^5 v^5-14 u^6 v^5+10 u^7 v^5-4 u^8 v^5+u^9 v^5+2 u^4 v^6+\\
&-14 u^5 v^6+28 u^6 v^6-24 u^7 v^6+10 u^8 v^6-2 u^9 v^6-2 u^4 v^7+10 u^5 v^7-24 u^6 v^7+28 u^7 v^7+\\
&-14 u^8 v^7+2 u^9 v^7+u^4 v^8-4 u^5 v^8+10 u^6 v^8-14 u^7 v^8+9 u^8 v^8-2 u^9 v^8+u^5 v^9-2 u^6 v^9+\\
&+2 u^7 v^9-2 u^8 v^9+u^9 v^9.
\end{align*}
\end{lem}
\begin{proof}
Observe that $\tsigma^{\mathrm{GL}}\setminus \tomega^{\mathrm{GL}}$ is $\mathbb{P}^1$-bundle over $\Sigma^{\mathrm{GL}}\setminus \Omega^{\mathrm{GL}}$ and that $\Sigma^{\mathrm{GL}}\setminus \Omega^{\mathrm{GL}}$ is isomorphic to $\left(\mathcal{J}\times H^0(K_C)\right)^{(2)}$ minus the diagonal $\mathcal{J}\times H^0(K_C)$. Then 
$$E(\tsigma^{\mathrm{GL}})=\left(E(\mathcal{J}^{(2)})u^4v^4-E(\mathcal{J})u^2v^2\right)E(\mathbb{P}^1).$$
\end{proof}
\subsection{Cohomology of $\tmdol^{\mathrm{GL}}$ and intersection cohomology of $\mdol^{\mathrm{GL}}$}
The $E-$polynomial of $\tmdol^{\mathrm{GL}}$ is the sum of the $E-$polynomials of each term in \eqref{etmdol}. We have proved the following.
\begin{thm}
Let $\tmdol^{\mathrm{GL}}$ the semismall desingularization of $\mdol^{\mathrm{GL}}$ constructed in Section 4. The $E-$polynomial of $\tmdol^{\mathrm{GL}}$ is 
\begin{align*}	
E(\tmdol^{\mathrm{GL}})&= 4u^5 v^5-8 u^6 v^5-8 u^5 v^6+5 u^7 v^5+22 u^6 v^6+5 u^5 v^7-2 u^8 v^5-22 u^7 v^6-22 u^6 v^7-2 u^5 v^8+\\
&+u^9 v^5+10 u^8 v^6+32 u^7 v^7+10 u^6 v^8+u^5 v^9-2 u^9 v^6-18 u^8 v^7-18 u^7 v^8-2 u^6 v^9+3 u^9 v^7+\\
&+15 u^8 v^8+3 u^7 v^9-6 u^9 v^8-6 u^8 v^9+u^{10}v^8+6 u^9 v^9+u^8 v^{10}-2 u^{10} v^9-2 u^9 v^{10}+u^{10} v^{10}.
\end{align*}
Moreover, by lemma \ref{etob}, the Poincaré polynomial of $\tmdol^{\mathrm{GL}}$ is 
$$P_t(\tmdol^{\mathrm{GL}})=1+4t+8t^2+12t^3+21t^4+40t^5+54t^6+48t^7+32t^8+16t^9+4t^{10}.$$
\end{thm}
Finally, one can prove theorem \ref{ihgl} computing the intersection Betti numbers of $\mdol^{\mathrm{GL}}$.
\begin{proof}[Proof of theorem \ref{ihgl}]
By equation \eqref{dece}, we have to subtract from $E(\tmdol^{\mathrm{GL}})$ the contributions $E^{top}$ coming from the top cohomology of the fibres.
The intersection $E-$polynomial of $\mdol^{\mathrm{GL}}$ is
\begin{align*}
IE(\mdol^{\mathrm{GL}})&= 2 u^5 v^5-4 u^6 v^5-4 u^5 v^6+2 u^7 v^5+10 u^6 v^6+2 u^5 v^7-8 u^7 v^6-8 u^6 v^7+2 u^8 v^6+11 u^7 v^7+\\
&+2 u^6 v^8-6 u^8 v^7-6 u^7 v^8+u^9 v^7+7 u^8 v^8+u^7 v^9-4 u^9 v^8-4 u^8 v^9+u^{10} v^8+5 u^9 v^9+\\
&+u^8 v^{10}-2 u^{10} v^9-2 u^9 v^{10}+u^{10} v^{10}.
\end{align*}
By lemma \ref{etob}, we get the Poincaré polynomial and the Hodge diamond. 
\end{proof}

\section*{Ackowledgements}

This article comes from my Ph.D thesis. I wish to thank my supervisor Luca Migliorini for suggesting me this problem as long as for the helpful countless discussions and corrections. Then I would like to thank Filippo Viviani and Gabriele Mondello for comments and suggestions. Special thanks to Jochen Heinloth, Sebastian Mejia Schlegel and J\"org Schr\"urmann, who pointed out mistakes in previous versions of this paper. Also I am grateful to Enrico Fatighenti, Giovanni Mongardi, Lorenzo Ruffoni and Marco Trozzo for the support in writing this article and for useful discussions. I would also like to thank the anonymous referee for the quick and careful review, which substantially improved the exposition of this article.

\vspace{40pt}
Camilla Felisetti\\
Section de Mathematiques\\
Université de Genève\\
\noindent
e-mail: camilla.felisetti@unige.ch\\
\end{document}